\documentclass[amscd, leqno]{amsart}
\usepackage{amscd, amsmath, amssymb, amsfonts, mathrsfs}
\usepackage[all]{xy}
\usepackage{url}

\theoremstyle{plain}
\newtheorem{theorem}{Theorem}[section]
\newtheorem{proposition}[theorem]{Proposition}
\newtheorem{lemma}[theorem]{Lemma}

\newtheorem{corollary}[theorem]{Corollary}
\theoremstyle{definition}
\newtheorem{definition}[theorem]{Definition}
\newtheorem{example}[theorem]{Example}
\theoremstyle{remark}
\newtheorem{remark}[theorem]{Remark}
\newtheorem{question}[theorem]{Question}
\newtheorem{conjecture}[theorem]{Conjecture}
\newtheorem{problem}[theorem]{Problem}
\newenvironment{pf}{\begin{proof}}{\end{proof}}
\makeatletter

\@addtoreset{equation}{section}
\makeatother

\begin{document}
\title[analytic torsion under degenerations]
{On the behavior of analytic torsion for twisted canonical bundles under degenerations}
\author{Ken-Ichi Yoshikawa}
\address{
Department of Mathematics,
Faculty of Science,
Kyoto University,
Kyoto 606-8502, JAPAN}
\email{yosikawa@math.kyoto-u.ac.jp}

\begin{abstract}
Consider a degeneration of projective algebraic manifolds equipped with a compact group action over a curve.
Suppose that the total space carries a Nakano semi-positive vector bundle, which is equivariant with respect to this action. 
We consider the relative canonical bundle twisted by this bundle.
Under this setting, we prove that the logarithm of the equivariant analytic torsion of the regular fibers for this coefficient
admits an asymptotic expansion near the discriminant locus.
The leading term is given by a logarithmic singularity, while the subdominant term is given by a $\log\log$-type singularity. 
In the non-equivariant case, we provide a formula for the coefficient of the leading term in terms of an integral of characteristic classes associated with the semi-stable reduction of the family. 
To establish these results, we prove the existence of an asymptotic expansion for both the equivariant Quillen metrics and the $L^{2}$-metrics 
in the above setting. We also calculate the leading term of the fiber integral of the Bott-Chern classes associated with the degeneration. 
\end{abstract}

\maketitle
\tableofcontents

\section*{Introduction}
\label{sect:0}
\par
Let ${\mathcal X}$ be a connected complex projective algebraic manifold of dimension $n+1$ and let $C$ be a compact Riemann surface.
Let $G$ be a compact Lie group acting holomorphically on ${\mathcal X}$ and trivially on $C$.
Let $\pi\colon {\mathcal X} \to C$ be a surjective proper holomorphic $G$-equivariant map with connected fibers. 
Then $G$ preserves the fibers of $\pi$. Let $\Sigma_{\pi}$ be the critical locus of $\pi$. 
Let $S \subset C$ be a simply connected domain and fix an identification between $S$ and the unit disc of ${\mathbf C}$. 
Under this identification, assume that $\pi(\Sigma_{\pi}) \cap S=\{0\}$. 
Set $X = \pi^{-1}(S)$. Then $\pi\colon X \to S$ is a degeneration of projective algebraic manifolds over the unit disc. 
\par
Let $h_{\mathcal X}$ be a $G$-invariant K\"ahler metric on ${\mathcal X}$ and set $h_{X} = h_{\mathcal X}|_{X}$. 
Let $\xi \to {\mathcal X}$ be a $G$-equivariant holomorphic vector bundle on ${\mathcal X}$ endowed with a $G$-invariant 
Hermitian metric $h_{\xi}$. Let $K_{X/S}$ be the relative canonical bundle of $\pi$ and set $K_{X/S}(\xi) = K_{X/S}\otimes\xi$. 
Assume that $(\xi,h_{\xi})|_{X}$ is Nakano semi-positive. 
For $s\in S$, we set $X_{s} = \pi^{-1}(s)$ and $\xi_{s} = \xi|_{X_{s}}$.
Under this assumption, $R^{q}\pi_{*}K_{X/S}(\xi)$ is a $G$-equivariant locally free sheaf on $S$ 
of rank $h^{q}= \dim H^{q}(X_{s}, K_{X_{s}}(\xi_{s}))$ for all $q \geq0$ (\cite{Takegoshi95, MourouganeTakayama08}). 
Let $R^{q}\pi_{*}K_{X/S}(\xi) = \bigoplus_{W\in\widehat{G}} R^{q}\pi_{*}K_{X/S}(\xi)_{W}$
be the isotypic decomposition, where $\widehat{G}$ is the isomorphism classes of the irreducible representations of $G$ and 
$R^{q}\pi_{*}K_{X/S}(\xi)_{W}$ is the subbundle of $R^{q}\pi_{*}K_{X/S}(\xi)$ corresponding to $W$. We denote by $h_{W}^{q}$ the rank of
$R^{q}\pi_{*}K_{X/S}(\xi)_{W}$.
\par
For $g \in G$, let $\tau_{G}(X_{s},K_{X_{s}}(\xi_{s}))(g)$ be the equivariant analytic torsion of $K_{X_{s}}(\xi_{s})$
with respect to the metrics induced from $h_{X}$, $h_{\xi}$ (\cite{Bismut95}). By \cite{BGS88, Ma00}, its logarithm yields a smooth function 
on $S\setminus\{0\}$. The purpose of the present article is to describe the structure of the singularity of the function 
$\log \tau_{G}(X_{s},K_{X_{s}}(\xi_{s}))(g)$ at $s = 0$. (We refer the reader to 
\cite{Wolpert87, BismutBost90, Bismut97, Yoshikawa04, FLY08, EFM21, ErikssonFreixas26} for known results.) 
This subject was previously investigated in the author's preprint \cite{Yoshikawa10a}. However, since that work contained certain inaccuracies 
in the proof (see Remark~\ref{remark:inaccuracy}) and lacked a general local formula for the leading term in the asymptotic expansion,
the present article aims to provide a rigorous proof for the asymptotic expansion of $\log\tau_{G}(X_{s},K_{X_{s}}(\xi_{s}))(g)$, 
as well as a topological formula for the leading term expressed via integrals of certain characteristic classes (when $G$ is trivial).
\par
To describe the leading term of the singularity of the equivariant analytic torsion, we first introduce further notation.
Let ${\mathbf P}(T{\mathcal X})^{\lor}$ be the projective space bundle over ${\mathcal X}$ parametrizing the hyperplanes in $T{\mathcal X}$
and let $\gamma_{\pi} \colon {\mathcal X} \setminus \Sigma_{\pi} \to {\mathbf P}(T{\mathcal X})^{\lor}$ be the Gauss map for $\pi$. 
Let $q \colon \widetilde{\mathcal X} \to {\mathcal X}$ be a $G$-equivariant resolution of the indeterminacy of $\gamma_{\pi}$. 
There exists a $G$-equivariant holomorphic map
$\widetilde{\gamma}_{\pi} \colon \widetilde{\mathcal X} \to {\mathbf P}(T{\mathcal X})^{\lor}$ such that 
$\gamma_{\pi} = \widetilde{\gamma}_{\pi} \circ q^{-1}$ on ${\mathcal X} \setminus \Sigma_{\pi}$. 
Let ${\mathcal X}^{g}$ and $\widetilde{\mathcal X}^{g}$ denote the fixed point sets of the $g$-action on ${\mathcal X}$ and 
$\widetilde{\mathcal X}$, respectively. We can further decompose them into the horizontal and vertical components:
${\mathcal X}^{g} = {\mathcal X}^{g}_{H} \amalg {\mathcal X}^{g}_{V}$ and
$\widetilde{\mathcal X}^{g} = \widetilde{\mathcal X}^{g}_{H} \amalg \widetilde{\mathcal X}^{g}_{V}$.
We set $E = q^{-1}(\Sigma_{\pi} \cap X_{0})$. Then we define a complex number $\alpha_{\pi, K_{X/S}(\xi)}(g)$ as
$$
\begin{aligned}
\alpha_{\pi,K_{X/S}(\xi)}(g)
&=
\int_{E\cap\widetilde{\mathcal X}^{g}_{H}}
\widetilde{\gamma}^{*}
\left\{
\frac{{\rm Td}({\mathcal H}^{\lor})^{-1}-1}{c_{1}({\mathcal H}^{\lor})}
\right\}\,
q^{*}\{{\rm Td}_{g}(TX){\rm ch}_{g}(K_{X}(\xi))\}
\\
&\quad
-\int_{{\mathcal X}^{g}_{V}\cap X_{0}}{\rm Td}_{g}(TX){\rm ch}_{g}(K_{X}(\xi)).
\end{aligned}
$$
Here ${\rm Td}_{g}(\cdot)$ and ${\rm ch}_{g}(\cdot)$ denote the equivariant Todd class and the equivariant Chern character class,
respectively, and ${\mathcal H}$ is the line bundle on ${\mathbf P}(T{\mathcal X})^{\lor}$ defined as the quotient 
$\varPi^{*}(T{\mathcal X})/{\mathcal U}$, where $\varPi \colon {\mathbf P}(T{\mathcal X})^{\lor} \to {\mathcal X}$ is the projection 
and ${\mathcal U}$ is the universal rank $n$ bundle on ${\mathbf P}(T{\mathcal X})^{\lor}$. 
\par
Let $T$ be another unit disc in ${\mathbf C}$. We fix a semi-stable reduction (\cite{Mumford73}): 
$$
\begin{CD}
(Y,Y_{0})@> F >>  (X,X_{0})
\\
@V f VV  @V \pi VV
\\
(T,0) @>\mu>> (S,0)
\end{CD}
$$
where $Y$ is a complex manifold of dimension $n+1$, $f \colon Y \to T$ and $F \colon Y \to X$ are proper surjective holomorphic maps, 
$\mu\colon(T,0)\to(S,0)$ is given by $\mu(t)=t^{d}$, $Y_{t}=f^{-1}(t)$ is isomorphic to $X_{\mu(t)}$ by $F$ for $t \not=0$, 
and $Y_{0}=f^{-1}(0)$ is a {\em reduced} normal crossing divisor on $Y$. 
Then the $G$-action on $X\setminus X_{0}$ lifts to a $G$-action on $Y \setminus Y_{0}$.
Note that the $G$-action on $X$ is not assumed to lift to a $G$-action on $Y$. (We do not know whether this is always possible.)
Since $(F^{*}\xi, F^{*}h_{\xi})$ is Nakano semi-positive, $R^{q}f_{*}K_{Y/T}(F^{*}\xi)$ is a locally free sheaf on $T$ of rank $h^{q}$. 
Even though $G$ may not act on $Y$, it can be proved that $R^{q}f_{*}K_{Y/T}(F^{*}\xi)$ is endowed with the structure of a $G$-equivariant 
vector bundle on $T$ in a natural way. Consequently, we have the isotypic decomposition 
$R^{q}f_{*}K_{Y/T}(F^{*}\xi) = \bigoplus_{W\in\widehat{G}} R^{q}f_{*}K_{Y/T}(F^{*}\xi)_{W}$.
By \cite{MourouganeTakayama09}, there is a natural injective homomorphism of $G$-equivariant locally free sheaves
$R^{q}f_{*}K_{Y/T}(F^{*}\xi) \hookrightarrow \mu^{*}R^{q}\pi_{*}K_{X/S}(\xi)$. For $g \in G$, we define 
$$
\kappa_{g} = \alpha_{\pi, K_{X/S}(\xi)}(g) - \sum_{q \geq 0,\, W \in \widehat{G}} \frac{(-1)^{q} \chi_{W}(g)}{\dim W} \delta_{W}^{q},
$$
where $\chi_{W}$ is the character of $W$ and
$$
\delta_{W}^{q}
=
\frac{1}{\deg \mu} \dim_{\mathbf C} \left( \frac{\mu^{*}R^{q}\pi_{*} K_{X/S}(\xi)_{W}}{R^{q}f_{*} K_{Y/T}(F^{*}\xi)_{W}} \right)_{0}
\in {\mathbf Q}_{\geq0}.
$$
Our first theorem is then stated as follows (Theorem~\ref{thm:main:1}).

\begin{theorem}
\label{Mthm:asym:exp}
Let $g \in G$. Near $s=0$, $\log \tau_{G}(X_{s},K_{X_{s}}(\xi_{s}))(g)$ is expressed as
$$
\begin{aligned}
\,&
\log \tau_{G}(X_{s},K_{X_{s}}(\xi_{s}))(g) 
= 
\kappa_{g} \log |s|^{2} + c_{g} + \sum_{0 \leq m \leq n} \sum_{i\in I} |s|^{2r_{i}} (\log |s|^{-2})^{m} \phi_{i,m,g}(s)
\\
&- \sum_{q \geq 0, W \in \widehat{G}} \frac{(-1)^{q}\chi_{W}(g)}{\dim W}
\log \left( \sum_{0 \leq k \leq nh^{q}_{W}} \{ a_{k,W}^{q} + \sum_{1 \leq j \leq d} |s|^{\frac{2j}{d}} \psi_{j,k,W}^{q}(s) \} (\log |s|^{-2})^{k} \right),
\end{aligned}
$$
where $c_{g} \in {\mathbf C}$, $\{ r_{i} \}_{i \in I} \subset {\mathbf Q}\cap (0,1]$ is a finite set of positive rational numbers,
$(a_{0,W}^{q}, \ldots, a_{nh^{q}_{W},W}^{q}) \not= (0,\ldots,0)$ is a non-zero real vector,
the $\phi_{i,m,g}(s)$ are smooth ${\mathbf C}$-valued functions on $S$, and the $\psi_{j,k,W}^{q}(s)$ are smooth ${\mathbf R}$-valued 
functions on $S$. 
Note that, if $R^{q}\pi_{*}K_{X/S}(\xi)_{W}=0$, then $a_{0,W}^{q}=1$, $a_{k,W}^{q}=0$ $(k>0)$ and $\psi_{j,k,W}^{q}(s)=0$.
Let $\varrho^{q}_{W} = \max\{ 0 \leq k \leq nh^{q}_{W};\, a_{k,W}^{q} \not=0 \} \in {\mathbf Z}_{\geq0}$ and set
$$
\varrho_{g} = \sum_{q\geq0, \,W\in \widehat{G}} (-1)^{q} \chi_{W}(g) \varrho^{q}_{W} / \dim W.
$$
Then there exists a constant $\gamma_{g} \in {\mathbf C}$ such that as $s\to0$,
$$
\log \tau_{G}(X_{s},K_{X_{s}}(\xi_{s}))(g) = \kappa_{g} \log |s|^{2} - \varrho_{g}\,\log\log( |s|^{-2}) + \gamma_{g} + O\left(1/\log (|s|^{-1}) \right).
$$
Here, if $g=1$, the identity element of $G$, 
then $\kappa_{1} \in {\mathbf Q}$, $\varrho_{1} \in {\mathbf Z}$, and $\gamma_{1} \in {\mathbf R}$. 
\end{theorem}

The leading term of the asymptotic expansion is of particular significance, 
as it is closely related to the asymptotic behavior of holomorphic torsion invariants (e.g. \cite{EFM21, FLY08, Imaike24, Yoshikawa04}).
Although Theorem~\ref{Mthm:asym:exp} is effective for calculating $\kappa_{g}$ when $X_{0}$ has only mild singularities
(Theorems~\ref{thm:sing:tor:can} and \ref{thm:sing:tor:log:can}), 
it remains unclear whether $\kappa_{g}$ is a local invariant of the degeneration due to the absence of a local formula for 
$\sum (-1)^{q}\chi_{W}(g)\delta_{W}^{q}/\dim W$. 
To elucidate the local nature of $\kappa_{g}$, we provide an alternative formula for $\kappa_{g}$ via the integral of 
certain characteristic classes associated to the semi-stable reduction, assuming $G$ is trivial. 
\par
From now on, we assume $G$ to be trivial and write $\kappa = \kappa_{\pi, K_{X/S}(\xi)}$ for $\kappa_{1}$.
Let $J_{F} \colon Tf \to F^{*}T\pi$ be the differential of $F$ defined on $Y \setminus F^{-1}(\Sigma_{\pi})$, 
where $Tf$ and $T\pi$ are the relative tangent bundles of $f$ and $\pi$, respectively.
By resolving the indeterminacy of Gauss-type maps associated to $f$, 
there exist a proper modification $p \colon Y^{\natural} \to Y$, 
holomorphic vector bundles $(Tf)^{\natural}$, $(F^{*}T\pi)^{\natural}$ on $Y^{\natural}$, and a homomorphism of holomorphic vector bundles
$J_{F}^{\natural} \colon (Tf)^{\natural} \to (F^{*}T\pi)^{\natural}$ such that, on $Y^{\natural}\setminus (F\circ p)^{-1}(\Sigma_{\pi})$,
$J_{F}^{\natural} \colon (Tf)^{\natural} \to (F^{*}T\pi)^{\natural}$ is identified with the pullback of $J_{F} \colon Tf \to F^{*}T\pi$ by $p$
(Sections~\ref{sect:6} and \ref{sect:7}). 
Let $N \subset Y^{\natural}$ be the degeneracy locus of $J_{F}^{\natural}$.
By \cite{Fulton98}, there is a cohomology class $({\rm Td}^{\lor})^{Y^{\natural}}_{N}( J_{F}^{\natural})$ in 
$H_{c}^{*}(Y^{\natural}, {\mathbf Q})$ with
${\rm Td}^{\lor}( (Tf)^{\natural} ) - {\rm Td}^{\lor}( (F^{*}T\pi)^{\natural} ) = ({\rm Td}^{\lor})^{Y^{\natural}}_{N}( J_{F}^{\natural})$, 
where ${\rm Td}^{\lor}(\cdot)$ is the multiplicative characteristic class associated to ${\rm Td}^{\lor}(x) = x/(e^{x}-1)$ 
(Sections~\ref{sect:6.3} and \ref{sect:7.1}). Set
$$
\beta_{Y/X, \xi} = \int_{N} ({\rm Td}^{\lor})^{Y^{\natural}}_{N}( J_{F}^{\natural} ) \, {\rm ch}((F \circ p)^{*}(\xi)).
$$
Our second theorem provides a local formula for $\kappa$ as follows (Theorem~\ref{thm:leading:term:torsion}).

\begin{theorem}
\label{Mthm:kappa:local}
The rational number $\kappa$ in Theorem~\ref{Mthm:asym:exp} is given by
$$
\kappa = \frac{1}{\deg \mu} ( \alpha_{f,K_{Y/T}(F^{*}\xi)} + \beta_{Y/X, \xi} ).
$$
\end{theorem}

By this theorem, if $H$ is a positive line bundle on ${\mathcal X}$, there exists a polynomial $p(m) \in {\mathbf Q}[m]$
of degree $\leq \dim {\rm Sing}\,X_{0}$ such that $\kappa_{\pi, K_{X/S}(\xi\otimes H^{m})} = p(m)$ for $m \geq 0$.
Another consequence of Theorem~\ref{Mthm:kappa:local} is a locality of $\kappa$ (Theorem~\ref{thm:locality:sing:tors}). 
Roughly speaking, this implies that $\kappa$ is determined
by $\pi$ viewed as a function on a neighborhood of $\Sigma_{\pi}$, the topological vector bundle underlying $\xi$ near $\Sigma_{\pi}$, 
and the combinatorial structure of the pair $(X_{0}, \Sigma_{\pi})$.
Consequently, when $X_{0}$ has only isolated singularities, $\kappa$ can be expressed in terms of two types of invariants associated with 
the isolated critical points of a holomorphic function. 
For a germ $\pi_{x} \in {\mathcal O}_{X,x}$ at a point $x \in \Sigma_{\pi}$,
let $\mu(\pi_{x})$ be the Milnor number of $\pi_{x}$ and let $\widetilde{p}_{g}(\pi_{x})$ be the spectral genus of $\pi_{x}$, introduced recently 
by Eriksson and Freixas i Montplet \cite{ErikssonFreixas26}. 
Our third theorem is stated as follows (Theorem~\ref{thm:sing:tors:N:pos} (2)).

\begin{theorem}
\label{Mthm:kappa:IHS}
Suppose that $X_{0}$ has only isolated singularities. Then 
$$
\kappa = -{\rm rk}(\xi) \sum_{x\in {\rm Sing}\,X_{0}} \left( \frac{\mu(\pi_{x})}{(n+2)!} - \widetilde{p}_{g}(\pi_{x}) \right).
$$
\end{theorem}

When $\xi$ is a trivial Hermitian line bundle, this formula is due to Eriksson and Freixas i Montplet \cite{ErikssonFreixas26}.
When $\dim {\rm Sing}\,X_{0} >0$ and $\xi$ is topologically trivial near ${\rm Sing}\,X_{0}$, 
a similar formula for $\kappa$ holds, where $\kappa$ is expressed by $\alpha_{\pi,K_{X/S}(\xi)}$, ${\rm rk}(\xi)$ and 
the eigenvalues of the monodromy (Theorem~\ref{thm:sing:tors:N:pos} (1)). 
When $\pi$ is given locally by a quadric polynomial of rank $2$ near $\Sigma_{\pi}$, $\kappa$ is expressed as
the integral of the Bismut $E$-class \cite{Bismut97} (Proposition~\ref{prop:kappa:q:rk2}). 
Calculating $\kappa$ for such examples, we observe that $\kappa$ can take an arbitrary sign when $n=2$ and $\dim {\rm Sing}\,X_{0} =1$
(Section~\ref{sect:5.3.4} and Example~\ref{ex:quad:sing}), which is in contrast to various results and conjectures in \cite{ErikssonFreixas26}
concerning the sign of $\kappa$ for isolated hypersurface singularities. 
\par
Applications of these results to the holomorphic torsion invariants of Calabi-Yau manifolds will be presented in our forthcoming papers
\cite{KawaguchiYoshikawa26, Yoshikawa26}. For an application to the holomorphic torsion invariant of irreducible holomorphic 
symplectic fourfolds of $K3^{[2]}$-type with involution \cite{Imaike25}, we refer the reader to \cite{Imaike24}.
\par
We briefly outline the strategy used to prove the above theorems.
Since the equivariant analytic torsion is the ratio of the equivariant Quillen metric to the equivariant $L^{2}$-metric on the equivariant
determinant of the cohomology \cite{BGS88, Bismut95}, Theorem~\ref{Mthm:asym:exp} reduces to determining the behavior of these metrics 
as $s \to 0$. In the non-equivariant setting, the singularity of Quillen metric was calculated by Bismut \cite{Bismut97} and 
the author \cite{Yoshikawa07}, where the Bismut-Lebeau embedding theorem for Quillen metrics \cite{BismutLebeau91}
played a central role. In the equivariant setting, the same strategy works by replacing the Bismut-Lebeau embedding theorem 
with the Bismut embedding theorem \cite{Bismut95} for equivariant Quillen metrics (Theorems~\ref{Theorem4.1} and \ref{Thm:Sing:Q:adj}).
In the case where $\xi$ is trivial, the singularity of the $L^{2}$-metric on $R^{q}\pi_{*}K_{X/S}$ reduces to known results 
in variations of Hodge structures (VHS) \cite{Schmid73, EFM21}. However, since the techniques in VHS are not applicable to non-trivial $\xi$, 
a new approach is required for general case. 
Following Takegoshi \cite{Takegoshi95} and Mourougane-Takayama \cite{MourouganeTakayama09}, we represent 
the basis elements of $R^{q}f_{*}K_{Y/T}(F^{*}\xi)$ using harmonic forms. Since their restrictions to the fibers remain harmonic
\cite{MourouganeTakayama08}, the entries of the $L^{2}$-metric on $R^{q}f_{*}K_{Y/T}(F^{*}\xi)$ can be expressed as fiber integrals 
of specific differential forms. 
Consequently, the $L^{2}$-metric admits an asymptotic expansion of Barlet-Takayama type \cite{Barlet82, Takayama21, Takayama22}, 
which can be proved to be uniformly non-degenerate as $t \to 0$. 
The structure of the singularity of the $L^{2}$-metric, which is of independent interest, is described as follows 
(see Theorem~\ref{thm:str:sing:L2} for more details). 
Let ${\mathfrak m}^{\infty}_{0}(T)$ be the smooth functions on $T$ vanishing at $t=0$.

\begin{theorem}
\label{Mthm:Sing:L2}
By choosing suitable bases $\{ {\mathbf e}_{1},\ldots,{\mathbf e}_{h_{W}^{q}} \}$ of $R^{q}\pi_{*}K_{X/S}(\xi)_{W}$
and $\{ \widetilde{\mathbf e}_{1}, \ldots, \widetilde{\mathbf e}_{h_{W}^{q}}\}$ of $R^{q}f_{*}K_{Y/T}(F^{*}\xi)_{W}$,
there exist integers $e^{q}_{1},\ldots, e^{q}_{h_{W}^{q}} \geq0$ with the following properties:
\par{\rm (1)}
The $(h_{W}^{q}, h_{W}^{q})$-Hermitian matrices $H(s) = ( H_{\alpha\beta}(s) )$, 
$H_{\alpha\beta}(s) = ( {\mathbf e}_{\alpha} , {\mathbf e}_{\beta} )_{L^{2},X_{s}}$, and 
$\widetilde{H}(t) = ( \widetilde{H}_{\alpha\beta}(t) )$, 
$\widetilde{H}_{\alpha\beta}(t) = ( \widetilde{\mathbf e}_{\alpha} , \widetilde{\mathbf e}_{\beta} )_{L^{2},Y_{t}}$, 
are related via the following identity:
$$
H(\mu(t)) = D(t)\cdot \widetilde{H}(t)\cdot\overline{D(t)},
\qquad
D(t) = {\rm diag}( t^{-e^{q}_{1}},\ldots,t^{-e^{q}_{h_{W}^{q}}} ).
$$
Moreover, $\widetilde{H}(t)$ admits an asymptotic expansion of Barlet-Takayama type
$$
\widetilde{H}(t) \equiv \sum_{0\leq m \leq n} (\log |t|^{-2})^{m} A_{m} 
\mod \bigoplus_{0\leq k \leq n} (\log |t|^{-2})^{k} {\mathfrak m}^{\infty}_{0}(T) \otimes M_{h_{W}^{q}}( {\mathbf C} )
$$
with constant $(h_{W}^{q}, h_{W}^{q})$-Hermitian matrices $A_{m}$ $(1\leq m \leq n)$. 
In particular, there exist constants $a_{m,W}^{q} \in {\mathbf R}$ $( 1 \leq m \leq nh_{W}^{q})$ such that 
$$
\det \widetilde{H}(t) = \sum_{0 \leq m \leq nh_{W}^{q}} a_{m,W}^{q} (\log |t|^{-2})^{m} 
\mod \bigoplus_{0\leq k \leq nh_{W}^{q}} (\log |t|^{-2})^{k} {\mathfrak m}^{\infty}_{0}(T).
$$
\par{\rm (2)}
There exists a constant $C>0$ such that $\widetilde{H}(t) \geq C\,I_{h_{W}^{q}}$ on $T\setminus\{0\}$ as positive definite Hermitian matrices.
In particular, $a_{m,W}^{q} \not=0$ for some $0 \leq m \leq nh_{W}^{q}$.
\par{\rm (3)}
There exists a constant $c_{W}^{q} \in {\mathbf R}$ such that as $s\to 0$,
$$
\log\| {\mathbf e}_{1}\wedge \cdots\wedge {\mathbf e}_{h^{q}} \|_{L^{2}}^{2} 
= -\delta_{W}^{q}\,\log |s|^{2} + \varrho_{W}^{q}\log\log(|s|^{-2}) + c_{W}^{q} + O\left(1/\log (|s|^{-1}) \right).
$$
\end{theorem}

When $q=0$, ${\rm rk}\,\xi=1$, and if $X_{0}$ is a reduced normal crossing divisor, this theorem is due to Takayama \cite{Takayama22}.
\par
To prove Theorem~\ref{Mthm:kappa:local}, we compare the Quillen metrics on $\lambda(K_{Y/T}(F^{*}\xi))$ with respect to 
the degenerate metric $F^{*}h_{X} + f^{*}h_{T}$ and a non-degenerate K\"ahler metric $h_{Y}$ on $Y$. 
The singularity of the Quillen metric with respect to $h_{Y}$, $F^{*}h_{\xi}$ was previously determined in \cite{Yoshikawa07},
while the singularities of the $L^{2}$-metrics with respect to both $F^{*}h_{X} + f^{*}h_{T}$ and $h_{Y}$ are established in
Theorem~\ref{Mthm:Sing:L2}. Consequently, the problem reduces to comparing the ratio of these Quillen metrics. 
By the anomaly formula of Bismut-Gillet-Soul\'e \cite{BGS88}, this further reduces to determining the asymptotic behavior of 
the fiber integral of a certain Bott-Chern form on $Tf$ associated with the metrics $F^{*}h_{X} + f^{*}h_{T}$ and $h_{Y}$. 
Although $Tf$ is not a vector bundle on $Y$ and the metric $F^{*}h_{X} + f^{*}h_{T}$ degenerates along certain divisors,
we determine the asymptotic behavior of its fiber integral by extending the theory of Bott-Chern forms to this degenerate setting
(Theorems~\ref{thm:pusf:forward:Bott:Chern}, \ref{thm:pusf:forward:Bott:Chern:2}, \ref{thm:singularity:BottChern}). 
Theorem~\ref{Mthm:kappa:IHS} follows from Theorem~\ref{Mthm:kappa:local} and a formula due to
Eriksson and Freixas i Montplet \cite{ErikssonFreixas26}.
\par
This article is organized as follows.
In Section~\ref{sect:1}, we recall the notion of equivariant analytic torsion, equivariant Quillen metrics, equivariant determinants of
cohomology, and we prove the smoothness of equivariant Quillen metrics. 
In Section~\ref{sect:2}, we recall some results of Takegoshi \cite{Takegoshi95} and Mourougane-Takayama \cite{MourouganeTakayama08}
concerning Nakano semi-positive vector bundles.
In Section~\ref{sect:3}, we determine the singularity of equivariant Quillen metrics.
In Section~\ref{sect:4}, we prove Theorem~\ref{Mthm:Sing:L2}.
In Section~\ref{sect:5}, we prove Theorem~\ref{Mthm:asym:exp}. When $X_{0}$ has only canonical singularities, we also prove that
the $L^{2}$-metric on $R^{q}\pi_{*}K_{X/S}(\xi)|_{S\setminus\{0\}}$ extends to a continuous Hermitian metric defined on $S$.
In Section~\ref{sect:6}, we construct Bott-Chern type currents for homomorphisms between holomorphic Hermitian vector bundles of equal rank, 
assuming a degeneracy locus of codimension one, and determine the asymptotic behavior of their fiber integrals. 
In Section~\ref{sect:7}, we prove Theorems~\ref{Mthm:kappa:local} and \ref{Mthm:kappa:IHS}.
Applying these results, we prove that, when $\dim {\rm Sing}\,X_{0}=0$, up to a nowhere vanishing continuous function on $S$, 
the ratio of analytic torsions for the trivial and ample line bundles is given by 
the ratio of the associated period-type integrals (Theorem~\ref{thm:McKean:Singer:prod}).
In Section~\ref{sect:8}, we describe the asymptotic behavior of the first Chern form of 
the $L^{2}$-metric on $R^{q}\pi_{*}K_{X/S}(\xi)$ and establish the Poincar\'e boundedness of its curvature (Theorem~\ref{thm:c1:dir:im:K}).
In Section~\ref{sect:9}, we extend a theorem of Takegoshi \cite{Takegoshi95} under the degeneracy of K\"ahler metric
dealt with in Section~\ref{sect:4}, and we compare the Bott-Chern type currents with the Bott-Chern singular currents \cite{BGS90b}
in some special situation. 
\par
Finally, this article consolidates and extends the results of the author's three earlier preprints \cite{Yoshikawa10a}, \cite{Yoshikawa10b}, 
and \cite[Appendix A]{Imaike24}.

\medskip
\par{\bf Notation }
For a complex manifold, we set $d^{c}=\frac{1}{4\pi i}(\partial-\bar{\partial})$. Hence $dd^{c}=\frac{1}{2\pi i}\bar{\partial}\partial$.
For a complex manifold $M$ and a holomorphic vector bundle $E$, $A^{p,q}(M,E)$ denotes the smooth $(p,q)$-forms on $M$
with values in $E$. When $E={\mathcal O}_{M}$, we write $A^{p,q}(M)$.
We set $\widetilde{A}(M)=\bigoplus_{p\geq0}A^{p,p}(M)/{\rm Im}\,\partial+{\rm Im}\,\bar{\partial}$. We define
$$
{\mathcal B}(S) = C^{\infty}(S) \oplus \bigoplus_{r\in{\bf Q}\cap(0,1),\, 0 \leq k \leq n} |s|^{2r}(\log|s|)^{k}C^{\infty}(S) \oplus
\bigoplus_{1\leq k \leq n} |s|^{2}(\log|s|)^{k}C^{\infty}(S).
$$
If $\phi,\psi\in C^{\infty}(S^{o})$ satisfies $\phi-\psi\in{\mathcal B}(S)\subset C^{0}(S)$, we write $\phi\equiv_{\mathcal B}\psi$.

\medskip
\par
{\bf Acknowledgements }
The author is grateful to Professors Jean-Michel Bismut, Dennis Eriksson, Gerard Freixas i Montplet, Christophe Mourougane, 
and Shigeharu Takayama for enlightening and helpful discussions. 
The author was partially supported by JSPS KAKENHI Grant Numbers 21H00984, 21H04429.
DeepL and Gemini were used solely to refine the English phrasing. All mathematical content was developed solely by the author.

\section
{Equivariant analytic torsion and equivariant Quillen metrics}
\label{sect:1}
\par
In this section, we recall equivariant analytic torsion and equivariant Quillen metrics and prove the smoothness of 
equivariant Quillen metrics for smooth projective morphisms. 
For a more general treatment including smooth K\"ahler morphisms, we refer the reader to \cite[III, Sects.\,2 and 3]{BGS88}.

\subsection
{\bf Equivariant analytic torsion and equivariant Quillen metrics}
\label{sect:1.1}
\par
Let $G$ be a compact Lie group. In the rest of this article, $\widehat{G}$ denotes the set of equivalence classes of complex irreducible representations of $G$. For $W\in\widehat{G}$, the corresponding irreducible character is denoted by $\chi_{W}$. 
Let $V$ be a compact K\"ahler manifold. Assume that $G$ acts on $V$ as holomorphic automorphisms. 
Let $h_{V}$ be a $G$-invariant K\"ahler metric on $V$. 
Let $F$ be a $G$-equivariant holomorphic vector bundle on $V$ endowed with a $G$-invariant Hermitian metric $h_{F}$.
We set $\overline{V}=(V,h_{V})$ and $\overline{F}=(F,h_{F})$.
Then $A^{p,q}(V,F)$ is endowed with the $L^{2}$ metric $(\cdot,\cdot)_{L^{2}}$ with respect to $h_{V}$ and $h_{F}$,
which is $G$-invariant with respect to the standard $G$-action on $A^{p,q}(V,F)$.
\par
Let $\square_{F}^{p,q} = (\bar{\partial}_{F}+\bar{\partial}^{*}_{F})^{2}$ be the Hodge-Kodaira Laplacian acting on $A^{p,q}(V,F)$. 
We denote by $\sigma(\square_{F}^{p,q})$ the eigenvalues of $\square_{F}^{p,q}$.
Let $E(\lambda; \square_{F}^{p,q})$ be the eigenspace of $\square_{F}^{p,q}$ with eigenvalue $\lambda\in\sigma(\square_{F}^{p,q})$. 
Since $G$ preserves the metrics $h_{V}$ and $h_{F}$, $\square_{F}^{p,q}$ commutes with the $G$-action on $A^{p,q}(V,F)$. 
Consequently, $G$ acts on $E(\lambda; \square_{F}^{p,q})$.
For $g\in G$ and $s\in{\mathbf C}$ with ${\rm Re}\,s\gg0$, we define
$$
\zeta_{G}(g)(s) := 
\sum_{q\geq 0} (-1)^{q}q \sum_{\lambda\in\sigma(\square_{F}^{p,q})\setminus\{0\}} 
\lambda^{-s}\,{\rm Tr}\left[ g|_{E(\lambda; \square_{F}^{p,q})} \right].
$$ 
It is classical that $\zeta_{G}(g)(s)$ extends to a meromorphic function on ${\mathbf C}$ and is holomorphic at $s=0$. 
For $g\in G$, we define
$$
\log\tau_{G}(\overline{V},\overline{F})(g) := -\zeta_{G}'(g)(0).
$$
In this article, $\tau_{G}(\overline{V},\overline{F})(g)$ is called the equivariant analytic torsion of $(\overline{V},\overline{F})$.
\par
Since $F$ is a $G$-equivariant holomorphic vector bundle on $V$, $G$ acts on $H(V,F) = \bigoplus_{q\in{\mathbf Z}} H^{q}(V,F)$ and preserves its ${\mathbf Z}$-grading. Consider the isotypic decomposition of the ${\mathbf Z}$-graded $G$-vector space
$$
H(V,F) \cong \bigoplus_{W\in\widehat{G}} {\rm Hom}_{G}(W,H(V,F))\otimes W.
$$
Since $H(V,F)$ is finite dimensional, ${\rm Hom}_{G}(W,H(V,F))=0$ except for finitely many $W\in\widehat{G}$. We set
$$
\begin{aligned}
\lambda_{W}(F) 
&= \det{\rm Hom}_{G}(W,H(V,F))\otimes W 
\\
&:= \bigotimes_{q\geq0}(\det {\rm Hom}_{G}(W,H^{q}(V,F))\otimes W)^{(-1)^{q}}.
\end{aligned}
$$
The equivariant determinant of the cohomology of $F$ is defined as
$$
\lambda_{G}(F):=\prod_{W\in\widehat{G}}\lambda_{W}(F).
$$
Notice that our sign convention is different form the one in \cite[Eq.\,(2.9)]{Bismut95}. 
When ${\rm Hom}_{G}(W,H(V,F))=0$, $\lambda_{W}(F)$ is canonically isomorphic to ${\mathbf C}$ by definition. 
In this case, the canonical element of $\lambda_{W}(F)$ corresponding to $1\in{\mathbf C}$ is denoted by $1_{\lambda_{W}(F)}$.
A vector $\alpha = (\alpha_{W})_{W\in\widehat{G}} \in \lambda_{G}(F)$ is said to be {\em admissible} 
if $\alpha_{W}\not=0$ for all $W\in\widehat{G}$ and if $\alpha_{W}=1_{\lambda_{W}(F)}$ except for finitely many $W\in\widehat{G}$. 
The set of admissible elements of $\lambda_{G}(F)$ is identified with the direct sum 
$\bigoplus_{W\in\widehat{G}}\lambda_{W}(F)^{\times}$, 
where $\lambda_{W}(F)^{\times}:=\lambda_{W}(F)\setminus\{0\}$ is the set of non-zero elements of $\lambda_{W}(F)$.
\par
By Hodge theory, we have an isomorphism of ${\mathbf Z}$-graded $G$-spaces $H(V,F) \cong \bigoplus_{q\geq0} E(0; \square_{F}^{0,q})$. 
The $G$-invariant metric on $H(V,F)$ induced from the $L^{2}$-metric on 
$\bigoplus_{q\geq0} E(0; \square_{F}^{0,q}) \subset \bigoplus_{q\geq 0} A^{0,q}(V,F)$ via this isomorphism is denoted by $h_{H(V,F)}$. 
Then the isotypic decomposition of $H(V,F)$ is orthogonal with respect to $h_{H(V,F)}$. 
Let $\|\cdot\|_{L^{2},\lambda_{W}(F)}$ be the Hermitian metric on $\lambda_{W}(F)$ induced from $h_{H(V,F)}$. 
Recall that $\chi_{W}$ is the character of $W\in\widehat{G}$. For an admissible element
$\alpha = (\alpha_{W})_{W\in\widehat{G}} \in \bigoplus_{W\in\widehat{G}}\lambda_{W}(F)^{\times}$, we set
$$
\log\|\alpha\|^{2}_{Q,\lambda_{G}(F)}(g) 
:= -\zeta_{G}'(g)(0) + \sum_{W\in\widehat{G}}\frac{\chi_{W}(g)}{\dim W} \log\|\alpha_{W}\|_{L^{2},\lambda_{W}(F)}^{2}.
$$
The ${\mathbf C}$-valued function $\log\|\cdot\|^{2}_{Q,\lambda_{G}(F)}(g)$ on $\bigoplus_{W\in\widehat{G}}\lambda_{W}(F)^{\times}$ 
is called the equivariant Quillen metric on $\lambda_{G}(F)$ with respect to $h_{V}$, $h_{F}$. 
We refer the reader to \cite{BGS88, Bismut95, KohlerRoessler01, Ma00} 
for further details about equivariant analytic torsion, equivariant determinants, and equivariant Quillen metrics.

\subsection
{\bf Characteristic forms}
\label{sect:1.2}
\par
We denote by $c_{i}(F,h_{F}) \in \bigoplus_{p\geq0}A^{p,p}(V)$ the $i$-th Chern form of $(F,h_{F})$.
For $g \in G$, ${\rm Td}_{g}(F, h_{F}),{\rm ch}_{g}(F,h_{F}) \in \bigoplus_{p\geq0}A^{p,p}(V^{g})$ denote
the equivariant Todd form and the equivariant Chern character form of $(F,h_{F})$, respectively. 
For an explicit formulas for ${\rm Td}_{g}(F, h_{F})$ and ${\rm ch}_{g}(F,h_{F})$ in terms of the curvature form of $(F,h_{F})$,
we refer the reader to \cite[Def.\,2.4]{Bismut95}.

\subsection
{\bf Smoothness of equivariant Quillen metrics}
\label{sect:1.3}
\par
In this subsection, we prove the smoothness of equivariant Quillen metrics.
Let $M$ be a connected complex manifold with holomorphic $G$-action. Let $B$ be a connected complex manifold with trivial $G$-action.
Let $\pi \colon M \to B$ be a flat proper morphism. 
Assume that $\pi$ is $G$-equivariant and that there is a $G$-equivariant ample line bundle on $M$. 
Let $F$ be a $G$-equivariant holomorphic vector bundle on $M$. 
Let $h_{M}$ be a $G$-invariant K\"ahler metric on $M$ and let $h_{F}$ be a $G$-invariant Hermitian metric on $F$. 
We set $M_{b}:=\pi^{-1}(b)$ and $F_{b}:=F|_{M_{b}}$ for $b\in B$. 
Then $G$ preserves the fibers $M_{b}$ and $F_{b}$, and $\pi \colon M\to B$ is a family of projective algebraic manifolds.
\par
For $W\in\widehat{G}$ and an open subset $U \subset B$, we define 
$$
{\rm Hom}_{G}(W,R^{q}\pi_{*}F)(U) 
:= \{ \phi \in {\rm Hom}\left( W|_{\pi^{-1}(U)} , R^{q}\pi_{*}F|_{\pi^{-1}(U)}\right) ;\, g \phi = \phi g \,\,(\forall\, g \in G)\},
$$
where we identify $W$ with $W \otimes {\mathcal O}_{S}$. We define
${\rm Hom}_{G}(W,R^{q}\pi_{*}F)$ to be the sheaf on $B$ associated to the presheaf
$U \mapsto {\rm Hom}_{G}(W,R^{q}\pi_{*}F)(U)$.

\subsubsection
{The isotypic decomposition of the direct image sheaves}
\label{sect:1.3.1}
\par
Since there is a $G$-equivariant ample line bundle on $M$, by replacing $B$ with a relatively compact open subset of $B$ if necessary, 
there exist a complex of $G$-equivariant holomorphic vector bundles 
$$
F_{\bullet} \colon 0\longrightarrow F_{0}\longrightarrow F_{1}\longrightarrow \cdots\longrightarrow F_{m}\longrightarrow0
$$
over $M$ and a homomorphism $i \colon F\to F_{0}$ of $G$-equivariant vector bundles with the following properties \cite[Sect.\,7.2.7]{Quillen73}:
\begin{itemize}
\item[(i)]
The complex $0\to F\to F_{0}\to\cdots\to F_{m}\to0$ is acyclic.
\item[(ii)]
$H^{q}(M_{b},F_{i}|_{M_{b}})=0$ for all $q>0$, $i\geq0$, $b \in B$. 
\end{itemize}
We remark that since $F$ and the ample line bundle on $M$ are $G$-equivariant, the construction in \cite[Sect.\,7.2.7]{Quillen73} 
yields the $G$-equivariance of the vector bundles $F_{k}$ and the homomorphisms $F_{k} \to F_{k+1}$.
Then the cohomology sheaves of the complex of $G$-equivariant locally free sheaves
\begin{equation}
\label{eqn:complex:v:b}
\pi_{*}F_{\bullet} \colon 0 \longrightarrow \pi_{*}F_{0} \longrightarrow \pi_{*}F_{1} \longrightarrow\cdots\longrightarrow 
\pi_{*}F_{m} \longrightarrow 0
\end{equation}
calculate the direct image sheaves of $F$. Namely, for all $q\geq0$, 
we have the following canonical isomorphism of $G$-equivariant coherent sheaves on $B$:
\begin{equation}
\label{eqn:direct:im}
R^{q}\pi_{*}F \cong {\mathcal H}^{q}(\pi_{*}F_{\bullet}) := \ker\{\pi_{*}F_{q}\to\pi_{*}F_{q+1}\}/{\rm Im}\{\pi_{*}F_{q-1}\to\pi_{*}F_{q}\}.
\end{equation}
Since $\pi_{*}F_{i}$ is a $G$-equivariant holomorphic vector bundle on $B$, there exist $\nu \in {\mathbf N}$, irreducible representations 
$W_{1},\ldots,W_{\nu} \in \widehat{G}$, holomorphic vector bundles $E_{i}^{1}, \ldots, E_{i}^{\nu}$ on $B$ with trivial $G$-action, 
and homomorphisms of holomorphic vector bundles $\phi_{i}^{\alpha} \colon E_{i}^{\alpha} \to E_{i+1}^{\alpha}$ such that 
${\rm Hom}_{G}(W, \pi_{*}F_{i}) = 0$ for $W \not= W_{\alpha}$ $(1 \leq \alpha \leq \nu)$, 
\begin{equation}
\label{eqn:equiv:str:v:b}
\pi_{*}F_{i} = \bigoplus_{1 \leq \alpha \leq \nu} E_{i}^{\alpha} \otimes W_{\alpha} ,
\qquad
E_{i}^{\alpha} \cong {\rm Hom}_{G}(W_{\alpha}, \pi_{*}F_{i})
\end{equation}
and such that under the splitting \eqref{eqn:equiv:str:v:b}, the homomorphism $\pi_{*}F_{i} \to \pi_{*}F_{i+1}$ in \eqref{eqn:complex:v:b} 
is given by $\bigoplus_{\alpha} \phi_{i}^{\alpha} \otimes 1_{W_{\alpha}}$. 
By \eqref{eqn:direct:im}, \eqref{eqn:equiv:str:v:b}, we have the following canonical isomorphism of $G$-equivariant coherent sheaves on $B$
\begin{equation}
\label{eqnequiv:str:coh}
R^{q}\pi_{*}F \cong {\mathcal H}^{q}(\pi_{*}F_{\bullet}) =
\bigoplus_{1 \leq \alpha \leq \nu} {\mathcal H}^{q}( E_{\bullet}^{\alpha}, \phi_{\bullet}^{\alpha}) \otimes W_{\alpha},
\end{equation}
where ${\mathcal H}^{q}( E_{\bullet}^{\alpha}, \phi_{\bullet}^{\alpha})$ is the $q$-th cohomology sheaf of the complex of locally free sheaves
$( E_{\bullet}^{\alpha}, \phi_{\bullet}^{\alpha})$ with trivial $G$-action. 
Since ${\mathcal H}^{q}( E_{\bullet}^{\alpha}, \phi_{\bullet}^{\alpha}) \cong {\mathcal H}^{q}({\rm Hom}_{G}(W_{\alpha},\pi_{*}F_{\bullet}))$ 
by \eqref{eqn:equiv:str:v:b}, we deduce from \eqref{eqnequiv:str:coh}
the following canonical isomorphism of sheaves on $B$:
\begin{equation}
\label{eqn:isotyp}
{\rm Hom}_{G}(W_{\alpha}, R^{q}\pi_{*}F) \cong {\mathcal H}^{q}({\rm Hom}_{G}(W_{\alpha}, \pi_{*}F_{\bullet})).
\end{equation}
Since ${\rm Hom}_{G}(W_{\alpha},\pi_{*}F_{\bullet}) \cong (E^{\alpha}_{\bullet}, \phi^{\alpha}_{\bullet})$ 
is a complex of locally free sheaves on $B$, 
all ${\rm Hom}_{G}(W_{\alpha},R^{q}\pi_{*}F)$ $(q\geq0)$ are coherent sheaves on $B$ by \eqref{eqn:isotyp}. 
For $W \in \widehat{G}$, set 
$$
(R^{q}\pi_{*}F)_{W} :=  {\rm Hom}_{G}(W,R^{q}\pi_{*}F)\otimes W.
$$
The $(R^{q}\pi_{*}F)_{W}$ are coherent sheaves on $B$.
We deduce from \eqref{eqnequiv:str:coh}, \eqref{eqn:isotyp} the isotypic decomposition of $R^{q}\pi_{*}F$ as coherent sheaves on $B$ 
\begin{equation}
\label{eqn:dir:sum:dir:im}
R^{q}\pi_{*}F \cong \bigoplus_{W\in\widehat{G}} (R^{q}\pi_{*}F)_{W}.
\end{equation}

\subsubsection
{Equivariant determinant of the cohomology}
\label{sect:1.3.2}
\par
For any coherent sheaf ${\mathscr S}$ on $B$, we can associate the invertible sheaf $\det{\mathscr S}$ on $B$ 
(\cite{KnudsenMumford76}, \cite[III, Sect.\,3]{BGS88}). We define
$$
\lambda_{W}(F) := \bigotimes_{q\geq0} \det (R^{q}\pi_{*}F)_{W}^{(-1)^{q}}
\qquad
(W \in \widehat{G}),
$$
$$
\lambda_{G}(F) := \prod_{W\in\widehat{G}}\lambda_{W}(F).
$$
Let $F_{\bullet}$ be the complex of holomorphic vector bundles on $M$ as in Section~\ref{sect:1.3.1}.
Since ${\rm Hom}_{G}(W, R^{q}\pi_{*}F)$ is canonically identified with the $q$-th cohomology sheaf of the complex of locally free sheaves
${\rm Hom}_{G}(W, \pi_{*}F_{\bullet})$, we deduce from \eqref{eqn:isotyp} the following canonical isomorphisms of line bundles on $B$:
\begin{equation}
\label{eqn:det:equiv:coh}
\begin{aligned}
\lambda_{W}(F) 
&\cong 
\bigotimes_{q\geq0} \det\left({\rm Hom}_{G}(W, {\mathcal H}^{q}(\pi_{*}F_{\bullet}))\otimes W\right)^{(-1)^{q}}
\\
&\cong
\bigotimes_{i\geq0} \det\left({\rm Hom}_{G}(W, \pi_{*}F_{i})\otimes W\right)^{(-1)^{i}}.
\end{aligned}
\end{equation}
If ${\rm Hom}_{G}(W,R^{q}\pi_{*}F)=0$ for all $q\geq0$, then $\lambda_{W}(F)$ is canonically isomorphic to ${\mathcal O}_{B}$.
In this case, the canonical section of $\lambda_{W}(F)$ corresponding to $1\in H^{0}(B,{\mathcal O}_{B})$ is denoted by $1_{\lambda_{W}(F)}$. 
By \eqref{eqn:det:equiv:coh} and the base change theorem, for any $b\in B$, there is a canonical isomorphism of complex lines
\begin{equation}
\label{eqn:fiber}
\lambda_{W}(F)_{b} := \lambda_{W}(F) \otimes ({\mathcal O}_{B}/{\mathcal I}_{b}) \cong \lambda_{W}(F|_{M_{b}}),
\end{equation}
where ${\mathcal I}_{b} \subset {\mathcal O}_{B}$ is the ideal sheaf defined by the point $b \in B$.
\par
For an open subset $U\subset B$, a holomorphic section $\sigma=(\sigma_{W})_{W\in\widehat{G}}$ of $\lambda_{G}(F)|_{U}$ is said to be {\em admissible}
if $\sigma_{W}$ is nowhere vanishing on $U$ for all $W\in\widehat{G}$ and if $\sigma_{W}=1_{\lambda_{W}(F)}$ except for finitely many $W\in\widehat{G}$.

\subsubsection
{Smoothness of equivariant Quillen metrics}
\label{sect:1.3.3}
\par
Assume that $\pi \colon M \to B$ is a $G$-equivariant proper holomorphic submersion. Let $TM/B$ be the relative tangent bundle of $\pi$.
Let $g \in G$. By the canonical isomorphism \eqref{eqn:fiber}, $\lambda_{G}(F)$ is endowed with the equivariant Quillen metric 
$\|\cdot\|_{Q,\lambda_{G}(F)}(g)$ with respect to $h_{M}|_{TM/B}$, $h_{F}$ such that for all $b\in B$,
$$
\|\cdot\|_{Q,\lambda_{G}(F)}(g)(b):=\|\cdot\|_{Q,\lambda_{G}(F_{b})}(g).
$$

\begin{theorem}
\label{thm:smooth:eq:Quillen}
Let $g \in G$. Let $U\subset B$ be a small open subset such that $\lambda_{W}(F_{i})|_{U}\cong{\mathcal O}_{U}$ for all $i$ and $W\in\widehat{G}$.
Let $\sigma=(\sigma_{W})_{W\in\widehat{G}}$ be an admissible holomorphic section of $\lambda_{G}(F)|_{U}$. 
Then $\log\|\sigma\|_{Q,\lambda_{G}(F)}^{2}(g)$ is a smooth function on $U$.
\end{theorem}

\begin{pf}
By \eqref{eqn:det:equiv:coh}, we have the canonical isomorphism of holomorphic line bundles 
$\varphi_{W} \colon \lambda_{W}(F) \cong \bigotimes_{i\geq0}\lambda_{W}(F_{i})^{(-1)^{i}}$ on $B$. 
Since $U$ is a small open subset, for every $i$, there is an admissible holomorphic section 
$\sigma_{i}=(({\sigma}_{i})_{W})_{W\in\widehat{G}}$ of $\lambda_{G}(F_{i})|_{U}$ 
such that $\varphi_{W}(\sigma_{W}) = \bigotimes_{i\geq0}({\sigma}_{i})_{W}^{(-1)^{i}}$ for all $W\in\widehat{G}$. 
\par
Let $h_{F_{i}}$ be a $G$-invariant Hermitian metric on $F_{i}$ and 
let $\|\cdot\|_{Q,\lambda_{G}(F_{i})}(g)$ be the equivariant Quillen metric on $\lambda_{G}(F_{i})$ 
with respect to the $G$-invariant metrics $h_{M/B}=h_{M}|_{TM/B}$ and $h_{F_{i}}$. 
Since $0 \to F \to F_{\bullet} \to 0$ is an exact sequence of holomorphic vector bundles on $M$, 
there exists by \cite[I, f)]{BGS88}, \cite[Sect.\,1.2.3]{GilletSoule90}, \cite[Sect.\,3.3]{KohlerRoessler01} 
the Bott-Chern secondary class $\widetilde{\rm ch}_{g}(\overline{F},\overline{F}_{\bullet})\in\widetilde{A}(M^{g})$ such that
$$
dd^{c}\widetilde{\rm ch}_{g}(\overline{F},\overline{F}_{\bullet}) = {\rm ch}_{g}(F,h_{F}) - \sum_{i\geq0}(-1)^{i}{\rm ch}_{g}(F_{i},h_{F_{i}}).
$$
By the immersion formula of Bismut \cite[Th.\,0.1]{Bismut95} applied to the immersion $\emptyset\hookrightarrow M_{b}$, $b\in U$, 
we obtain the following identity of complex-valued functions on $U$
\begin{equation}
\label{eqn:anomaly}
\begin{aligned}
\log \| \sigma \|_{Q,\lambda_{G}(F)}^{2}(g)
&=
\sum_{i\geq0} (-1)^{i} \log\|\sigma_{i}\|_{Q,\lambda_{G}(F_{i})}^{2}(g)
-
\left[ \pi_{*} \{ {\rm Td}_{g}( \overline{TM/B} ) \widetilde{\rm ch}_{g}( \overline{F}, \overline{F}_{\bullet} ) \} \right]^{(0)}
\\
&\equiv
\sum_{i\geq0} (-1)^{i} \log \| \sigma_{i} \|_{Q,\lambda_{G}(F_{i})}^{2}(g) \mod C^{\infty}(U),
\end{aligned}
\end{equation}
where we set $\overline{TM/B} := (TM/B,h_{M/B})$, $\overline{F} := (F,h_{F})$ and we used \cite[Cor.\,3.10]{KohlerRoessler01} to identify 
the Bott-Chern current $T_{g}(\overline{F},\overline{F}_{\bullet})$ with the Bott-Chern class $-\widetilde{\rm ch}_{g}(\overline{F},\overline{F}_{\bullet})$ in the first equality. In \eqref{eqn:anomaly}, $\pi_{*}$ denotes the integration along the fibers of $\pi$ and $[\omega]^{(2d)}$ denotes the component of degree $2d$ of a differential form $\omega$. 
Since the morphism $\pi\colon M \to B$ is a proper holomorphic submersion and since $h^{0}(F_{i}|_{M_{b}})$ is a constant function on $U$, 
$(\pi_{*}F_{i}, h_{L^{2}})$ is a $G$-equivariant holomorphic Hermitian vector bundle for all $i$. In particular, so is its direct summand
$(W\otimes{\rm Hom}_{G}(W, \pi_{*}F_{i}), h_{L^{2}})$. Since the function $b \mapsto \log \tau_{G}(M_{b},F_{i}|_{M_{b}})(g)$ is smooth by \cite[Th.\,2.12]{Ma00},
$\log \| \sigma_{i} \|_{Q,\lambda_{G}(F_{i})}^{2}(g)$ is a smooth function on $U$. This, together with \eqref{eqn:anomaly}, implies the result.
\end{pf}

The curvature $-dd^{c}\log\|\sigma\|^{2}_{Q,\lambda_{G}(F)}(g)$ was calculated by Bismut-Gillet-Soul\'e \cite[Th.\,0.1]{BGS88} 
when $G$ is trivial and by Ma \cite[Th.\,2.12]{Ma00} when all $R^{q}\pi_{*}F$, $q\geq0$ are locally free. 
Recently, the curvature formula is further generalized by Ma \cite[Th.\,3.6]{Ma21} for submersions of orbifolds.
By Theorem~\ref{thm:smooth:eq:Quillen}, the curvature formula of Bismut-Gillet-Soul\'e and Ma remains valid in the following sense, 
without the assumption of the local freeness of the direct image sheaves $R^{q}\pi_{*}F$.

\begin{theorem}
\label{thm:curvature}
Suppose that $\pi \colon M \to B$ is a $G$-equivariant proper holomorphic submersion.
Let $\sigma$ be a locally defined admissible holomorphic section of $\lambda_{G}(F)$. Then the following identity holds:
\begin{equation}
\label{eqn:curv:eq:Q}
-dd^{c} \log \left\| \sigma \right\|_{Q,\lambda_{G}(F)}^{2}(g)
=
\left[\pi_{*}\{{\rm Td}_{g}(\overline{TM/B})\,{\rm ch}_{g}(\overline{F})\}\right]^{(2)}.
\end{equation}
\end{theorem}

\begin{pf}
Taking $dd^{c}$ of the both sides of the first line of \eqref{eqn:anomaly}, we get \eqref{eqn:curv:eq:Q}.
\end{pf}

\section
{Nakano semi-positive vector bundles}
\label{sect:2}
\par
In this section, we recall some results of Takegoshi \cite{Takegoshi95} and Mourougane-Takayama \cite{MourouganeTakayama08} 
about Nakano semi-positive vector bundles. 
\par
Let $M$ be a complex manifold of dimension $n+1$. Let $(E,h)$ be a holomorphic Hermitian vector bundle of rank $r$ on $M$
and let $D^{E}=D^{(E,h)}$ be its Chern connection. Then $D^{E} = (D^{E})^{1,0} + \bar{\partial}$, 
where $(D^{E})^{1,0}(A^{p,q}(M,E)) \subset A^{p+1,q}(M,E)$ and $\bar{\partial}$ is the Dolbeault operator.
Let $R^{E}$ be the curvature of $(E, D^{E})$. Then we can express
$$
h_{E}(i R^{E}(\cdot),\cdot)
=
\sum_{a,b,\alpha,\beta}R_{\alpha\bar{\beta}a\bar{b}}\,(e^{\lor}_{\alpha}\otimes\bar{e}_{\beta}^{\lor})\otimes(\theta_{a}\wedge\bar{\theta}_{b}),
$$
where $\{e_{\alpha}^{\lor}\}$ (resp. $\{\theta_{a}\}$) is a local {\em unitary} frame of $E^{\lor}$ (resp. $\Omega^{1}_{M}$).
The holomorphic Hermitian vector bundle $(E,h)$ is said to be {\em Nakano semi-positive} 
(resp. {\em semi-negative in the dual Nakano sense}) if for all $(\zeta_{a}^{\alpha})\in T^{(1,0)}M\otimes E$,
$$
\sum_{a,b,\alpha,\beta}R_{a\bar{b}\alpha\bar{\beta}}\zeta_{a}^{\alpha}\bar{\zeta}_{b}^{\beta} \geq 0 
\qquad
( \text{resp.}\quad \sum_{a,b,\alpha,\beta}R_{a\bar{b}\alpha\bar{\beta}}\zeta_{b}^{\alpha}\bar{\zeta}_{a}^{\beta} \leq 0).
$$ 
If the strict inequality holds for $(\zeta_{a}^{\alpha})\not=(0)$, then $(E,h)$ is called {\em Nakano positive} 
(resp. {\em negative in the dual Nakano sense}). Note the difference of indices in these two definitions.
By \cite[Lemma 4.3]{Siu82}, $(E,h)$ is Nakano semi-positive if and only if $(E^{\lor},h^{\lor})$ is semi-negative in the dual Nakano sense,
where $h^{\lor}$ is the dual metric. 
\par
Set $\Omega_{M}^{p} := \Lambda^{p}\Omega_{M}^{1}$. 
Then $K_{M} := \Omega_{M}^{n+1}$ is the canonical bundle of $M$. Let $\varDelta\subset{\bf C}$ be the unit disc. 
Let $\pi\colon M\to\varDelta$ be a proper surjective holomorphic map with critical locus $\varSigma$. 
We define $\Omega_{M/\varDelta}^{1} := \Omega_{M}^{1}/\pi^{*}\Omega_{\varDelta}^{1}$ and 
$\Omega_{M/\varDelta}^{p} := \Lambda^{p} \Omega_{M/\varDelta}^{1}$. 
Then $\Omega_{M/\varDelta}^{p}|_{M\setminus\varSigma}$ is a holomorphic vector bundle on $M \setminus \varSigma$. 
The relative canonical bundle $K_{M/\varDelta}$ is defined as $K_{M/\varDelta} := K_{M}\otimes \pi^{*}K_{\varDelta}^{-1}$.
We write $\Omega_{M}^{p}(E)$ and $\Omega_{M/\varDelta}^{p}(E)$ for $\Omega_{M}^{p}\otimes E$ and 
$\Omega_{M/\varDelta}^{p}\otimes E$, respectively. Similarly, we write $K_{M}(E)$ and $K_{M/\varDelta}(E)$ in place of $K_{M}\otimes E$
and $K_{M/\varDelta}\otimes E$, respectively. 
\par
Let $\kappa$ be a K\"ahler form on $M$.
For $t\in\varDelta$, we set $M_{t}:=\pi^{-1}(t)$, $\kappa_{t}:=\kappa|_{M_{t}}$ and $h_{t}:=h|_{M_{t}}$. 
The Dolbeault operator on $M_{t}$ is denoted by $\bar{\partial}_{t}$.

\begin{theorem}[\cite{Takegoshi95, MourouganeTakayama08}]
\label{Thm:Takegoshi}
Suppose that $(E,h)$ is Nakano semi-positive. Let $\alpha \in H^{q}(M, K_{M}(E))$ with $q\geq 1$. Then the following hold.
\begin{itemize}
\item[(1)]
There exists $\sigma\in \Gamma(M,\Omega^{n+1-q}_{M}(E))$ such that
$$
\alpha = [ \sigma\wedge\kappa^{q} ],
\qquad
D^{E}\sigma = 0,
\qquad
(\pi^{*}dt)\wedge\sigma=0.
$$
\item[(2)]
There exists $v\in \Gamma(M\setminus\varSigma,\Omega^{n-q}_{M/\varDelta}(E))$ such that
$$
\alpha|_{M\setminus\varSigma} = [ v\wedge\kappa^{q}\wedge(\pi^{*}dt) ],
\qquad
\sigma|_{M\setminus\varSigma} = v \wedge (\pi^{*}dt).
$$
\item[(3)]
For all $t\in\varDelta\setminus\pi(\varSigma)$, $v|_{M_{t}}\wedge\kappa_{t}^{q}\in A^{n,q}(M_{t},E|_{M_{t}})$ is a harmonic form 
with respect to $\kappa_{t}$, $h_{t}$.
\end{itemize}
\end{theorem}

\begin{pf}
Write $\star$ (resp. $\star_{t}$) for the Hodge star operator of the K\"ahler manifold $(M, \kappa)$ (resp. $(M_{t}, \kappa_{t})$). 
Let $\bar{\partial}^{*}_{h}$ (resp. $\bar{\partial}^{*}_{h_{t}}$) be the formal adjoint of $\bar{\partial}$ (resp. $\bar{\partial}_{t}$) with respect to 
the $L^{2}$-inner product on $A^{p,q}(M,E)$ (resp. $A^{p,q}(M_{t},E_{t})$) defined by $\kappa$, $h$ (resp. $\kappa_{t}$, $h_{t}$).
By \cite[Th.\,5.2 (i) and p.405 l.8]{Takegoshi95}, there exists a unique $u \in A^{n+1,q}(M,E)$ representing $\alpha$ such that 
$\bar{\partial}u=0$, 
$\bar{\partial}^{*}_{h} u =0$, $\epsilon(\pi^{*}d\bar{t})^{*}u=0$, where $\epsilon(\pi^{*}d\bar{t})^{*}$ is the adjoint of $\epsilon(\pi^{*}d\bar{t})$, 
the exterior multiplication by $\pi^{*}d\bar{t}$. We set $\sigma := C\star u \in A^{n+1-q, 0}(M,E)$, 
where $C=C_{n+1,q}$ is a universal constant depending only on $n+1$, $q$ such that $u = \sigma \wedge \kappa^{q}$.
By \cite[Th.\,4.3 (i)]{Takegoshi95}, we have $\bar{\partial} \sigma =0$. Namely, $\sigma \in \Gamma(M,\Omega^{n+1-q}_{M}(E))$.
Since $0= \bar{\partial}^{*}_{h} u = -\star (D^{E})^{1,0} \star u$, we get $(D^{E})^{1,0}\sigma=0$. Hence $D^{E}\sigma = 0$.
Since $0 = \epsilon(\pi^{*}d\bar{t})^{*}u = \pm \star \epsilon(\pi^{*}dt) \star u$, we get $(\pi^{*}dt)\wedge\sigma=0$. 
This proves (1).  
By \cite[Th.\,5.2 (ii)]{Takegoshi95} applied to $u$, we get (2). 
Set $v_{t} := v|_{M_{t}}$. Let $D^{E_{t}} = (D^{E_{t}})^{1,0} + \bar{\partial}_{t}$ be the Chern connection of $(E_{t}, h_{t})$.
Since $D^{E}\sigma=0$, we get $(D^{E_{t}})^{1,0}v_{t} =0$ for all $t\in\varDelta\setminus\pi(\varSigma)$
(cf. \cite[Lemmas~4.5 and 4.7 (1)]{MourouganeTakayama08}).
Since $\star_{t}(v_{t} \wedge \kappa_{t}^{q}) = C' v_{t}$ with $C'=C'_{n,q}$ being a universal constant depending only on $n$, $q$, 
this implies $\bar{\partial}_{h_{t}}^{*} (v_{t} \wedge \kappa_{t}^{q}) =0$.
Since $\bar{\partial}_{t} (v_{t} \wedge \kappa_{t}^{q}) =0$ is obvious, we get (3). 
\end{pf}

We remark that, by Theorem~\ref{Thm:Takegoshi} (3), $\int_{M_{t}} h(v_{t}, v_{t}) [\kappa_{t}]^{q}$ is well-defined
for all $t \in \Delta\setminus\pi(\varSigma)$, where $[\kappa_{t}]$ is the K\"ahler class of $\kappa_{t}$.
In Section~\ref{sect:4}, Theorem~\ref{Thm:Takegoshi} and its extension by Mourougane-Takayama \cite[Prop.\,4.4]{MourouganeTakayama09} 
will play a key role to determine the structure of the singularity of the $L^{2}$-metric on the direct image sheaves 
$R^{q}\pi_{*}K_{M/\varDelta}(E)$.
We refer the reader to \cite[Sect.\,4]{Siu82}, \cite[Chap.\,VII]{Demailly12} for further details about the positivity and negativity of vector bundles.
\par
By Takegoshi \cite[Th.\,6.5 (i)]{Takegoshi95}, $R^{q}\pi_{*}K_{M/\varDelta}(E)$ is a torsion free sheaf on $\varDelta$ for $q\geq0$.
Since $\dim \varDelta=1$, this implies the local freeness of $R^{q}\pi_{*}K_{M/\varDelta}(E)$. In fact, we have the following:

\begin{lemma}[\cite{MourouganeTakayama08}]
\label{lemma:local:freeness}
For all $q\geq 0$, $t \mapsto h^{q}(M_{t}, K_{M}(E)|_{M_{t}})$ is a constant function on $\varDelta$. 
In particular, $R^{q}\pi_{*}K_{M/\varDelta}(E)$ is a locally free sheaf on $\varDelta$ for all $q\geq0$. 
\end{lemma}

\begin{pf}
The proof is parallel to that of \cite[Lemma 4.9]{MourouganeTakayama08}. For the sake of completeness, we give it here.
Let $t_{0}\in\varDelta$. By \cite[Th.\,6.8 (i)]{Takegoshi95}, the homomorphism of direct image sheaves
$R^{q}\pi_{*}(\pi^{*}(t-t_{0})) \colon R^{q}\pi_{*}K_{M}(E) \to R^{q}\pi_{*}K_{M}(E)$ induced by the homomorphism of sheaves 
$\otimes\pi^{*}(t-t_{0}) \colon K_{M}(E) \to K_{M}(E)$ is injective for all $q\geq0$. 
By the long exact sequence of direct image sheaves associated to the short exact sequence
$\begin{CD} 0 \longrightarrow K_{M}(E)  @>\otimes\pi^{*}(t-t_{0})>> K_{M}(E) \longrightarrow K_{M}(E) \otimes {\mathcal O}_{M_{t_{0}}} \longrightarrow 0 \end{CD}$, 
this implies that the natural restriction map $R^{q}\pi_{*}K_{M}(E) \to H^{q}(M_{t_{0}}, K_{M}(E)|_{M_{t_{0}}})$ is surjective 
for all $q\geq0$ and $t_{0}\in \varDelta$. The result then follows from e.g. \cite[Lemma 4.8]{MourouganeTakayama08}.
\end{pf}

We refer the reader to \cite{Berndtsson09, MourouganeTakayama08, MourouganeTakayama09, Takegoshi95} 
for further details about the direct images of the relative canonical bundle twisted by a Nakano semi-positive vector bundle.

\section
{Equivariant Quillen metrics under degenerations}
\label{sect:3}
\par

\subsection
{Set up}
\label{sect:3.1}
\par
We introduce the basic set up and the notation used throughout this article.
Let ${\mathcal X}$ be a connected projective algebraic manifold of dimension $n+1$ and let $C$ be a compact Riemann surface.
Let $\pi\colon{\mathcal X}\to C$ be a surjective holomorphic map with connected fibers. Let $\Sigma_{\pi}$ be the critical locus of $\pi$.
\par
Let $G$ be a compact Lie group acting holomorphically on ${\mathcal X}$ and preserving the fibers of $\pi$. 
Assume that there exists a $G$-equivariant ample line bundle on ${\mathcal X}$.
Set 
$$
\Delta=\pi(\Sigma_{\pi}), 
\quad
C^{o}=C\setminus\Delta, 
\quad
{\mathcal X}^{o}={\mathcal X}|_{\pi^{-1}(C^{o})}, 
\quad
\pi^{o}=\pi|_{{\mathcal X}^{o}}.
$$
For $s\in C$, we set $X_{s}:=\pi^{-1}(s)$. 
Then $\pi^{o}\colon{\mathcal X}^{o}\to C^{o}$ is a family of projective algebraic manifolds with $G$-action.
\par
Let $T{\mathcal X}/C$ be the $G$-equivariant subbundle of $T{\mathcal X}|_{{\mathcal X}\setminus\Sigma_{\pi}}$ defined as
$T{\mathcal X}/C=\ker\pi_{*}|_{{\mathcal X}\setminus\Sigma_{\pi}}$.
Let $h_{\mathcal X}$ be a $G$-invariant K\"ahler metric on ${\mathcal X}$ and set $h_{{\mathcal X}/C}=h_{\mathcal X}|_{T{\mathcal X}/C}$.
Let $K_{\mathcal X}$ be the canonical bundle of ${\mathcal X}$ and let $K_{{\mathcal X}/C} := K_{\mathcal X}\otimes(\pi^{*} K_{C})^{-1}$ be the relative canonical bundle of $\pi\colon{\mathcal X}\to C$. Let $\xi\to{\mathcal X}$ be a $G$-equivariant holomorphic vector bundle on ${\mathcal X}$ endowed with a $G$-invariant Hermitian metric $h_{\xi}$. In what follows, we write $K_{{\mathcal X}/C}(\xi) = K_{{\mathcal X}/C}\otimes\xi$. 
We set $\xi_{s}=\xi|_{X_{s}}$ for $s\in C$.
\par
For $g\in G$, let 
$$
{\mathcal X}^{g}=\{x\in{\mathcal X};\,g\cdot x=x\}
$$ 
be its fixed-point set. Since $g$ is an isometry of ${\mathcal X}$, ${\mathcal X}^{g}$ is the disjoint union of compact complex submanifolds of ${\mathcal X}$. Then we have the following decomposition
$$
{\mathcal X}^{g}={\mathcal X}^{g}_{H}\amalg{\mathcal X}^{g}_{V},
$$
where ${\mathcal X}^{g}_{H}$ is the horizontal submanifold, i.e., $\pi|_{{\mathcal X}^{g}_{H}}\colon{\mathcal X}^{g}_{H}\to C$ is a flat holomorphic map and ${\mathcal X}^{g}_{V}$ is the vertical submanifold, i.e., $\pi({\mathcal X}_{V}^{g})$ is a proper subset of $C$. 
Since the $G$-action on $C$ is trivial, we have ${\mathcal X}^{g}_{V}\subset\Sigma_{\pi}$ and $\pi({\mathcal X}^{g}_{V})\subset\Delta$ by the $G$-equivariance of $\pi$.
\par
Let $0\in\Delta$ be a critical value of $\pi$. Let $(S,s)$ be a coordinate neighborhood of $C$ centered at $0$ such that $S \cap\Delta=\{0\}$. 
We set 
$$
X=\pi^{-1}(S),
\qquad
S^{o}=S\setminus\{0\}.
$$
\par
For $g\in G$ and $s\in S^{o}$, let $\tau_{G}(X_{s}, K_{X_{s}}(\xi_{s}))(g)$
be the equivariant analytic torsion of $(X_{s}, K_{X_{s}}(\xi_{s}))$ with respect to $h_{X_{s}}=h_{\mathcal X}|_{X_{s}}$ and 
$h_{\xi_{s}}=h_{\xi}|_{X_{s}}$, where $K_{X_{s}}$ is the canonical bundle of $X_{s}$ for $s\not=0$. 
In this article, we determine the asymptotic behavior of $\log \tau_{G}(X_{s}, K_{X_{s}}(\xi_{s}))(g)$ as $s \to 0$. 
To this end, we focus on the singularity of the equivariant Quillen metrics in this section.

\par
Let $\lambda_{G}(\xi)$ be the equivariant determinant of the cohomology of $\xi$. 
Then $\lambda_{G}(\xi)|_{C^{o}}$ is endowed with the equivariant Quillen metric $\|\cdot\|^{2}_{\lambda_{G}(\xi),Q}(\cdot)$ with respect to 
the $G$-invariant metrics $h_{{\mathcal X}/C}$, $h_{\xi}$. Let $\sigma$ be an admissible holomorphic section of $\lambda_{G}(\xi)|_{S}$. 
Let $g \in G$. By Theorem~\ref{thm:smooth:eq:Quillen}, $\log\|\sigma(s)\|^{2}_{\lambda_{G}(\xi),Q}(g)$ is a smooth function on $S^{o}$. 
In this section, we determine the behavior of $\log\|\sigma(s)\|^{2}_{\lambda_{G}(\xi),Q}(g)$ as $s\to0$.

\subsection
{Gauss map and its equivariant resolution}
\label{sect:3.2}
\par
Let $\Omega^{1}_{\mathcal X}$ be the holomorphic cotangent bundle of ${\mathcal X}$.
Let $\varPi\colon{\mathbf P}(\Omega^{1}_{\mathcal X}\otimes\pi^{*}TC)\to{\mathcal X}$ 
be the projective-space bundle associated with $\Omega^{1}_{\mathcal X}\otimes\pi^{*}TC$. 
Let $\varPi^{\lor}\colon{\mathbf P}(T{\mathcal X})^{\lor}\to{\mathcal X}$ 
be the dual projective-space bundle of ${\mathbf P}(T{\mathcal X})$, 
whose fiber ${\bf P}(T_{x}{\mathcal X})^{\lor}$ is the hyperplanes of $T_{x}{\mathcal X}$ passing through $0_{x}\in T_{x}{\mathcal X}$. 
Since $\dim C=1$, we have 
${\mathbf P}(\Omega^{1}_{\mathcal X}\otimes\pi^{*}TC) = {\mathbf P}(\Omega^{1}_{\mathcal X})\cong{\bf P}(T{\mathcal X})^{\lor}$. 
Since $T{\mathcal X}$ and $\Omega_{\mathcal X}^{1}$ are $G$-equivariant holomorphic vector bundles on ${\mathcal X}$, 
the projective-space bundles ${\mathbf P}(T{\mathcal X})$, ${\mathbf P}(T{\mathcal X})^{\lor}$ and 
${\mathbf P}(\Omega_{\mathcal X}^{1}\otimes \pi^{*}TC)$ are endowed with the natural $G$-action. 
\par
Let $x\in{\mathcal X}\setminus\Sigma_{\pi}$. Let $s$ be a holomorphic local coordinate of $C$ near $\pi(x)\in C$. 
Define the Gauss maps
$\nu\colon{\mathcal X}\setminus\Sigma_{\pi}\to{\bf P}(\Omega^{1}_{\mathcal X}\otimes\pi^{*}TC)$ 
and $\gamma\colon{\mathcal X}\setminus\Sigma_{\pi} \to {\mathbf P}(T{\mathcal X})^{\lor}$ by
$$
\nu(x) := [ d\pi_{x} ] = \left[ \sum_{i=0}^{n}\frac{\partial(s\circ\pi)} {\partial z_{i}}(x)\,dz_{i} \otimes \frac{\partial}{\partial s} \right],
\qquad
\gamma(x) := [T_{x}{\mathcal X}_{\pi(x)}].
$$
Under the canonical isomorphism ${\mathbf P}(\Omega^{1}_{\mathcal X}\otimes\pi^{*}TC) \cong {\mathbf P}(T{\mathcal X})^{\lor}$, we have $\nu=\gamma$.
\par
Let ${\mathcal H}:={\mathcal O}_{{\mathbf P}(T{\mathcal X})^{\lor}}(1)$ be the tautological quotient bundle 
and let ${\mathcal U}$ be the universal hyperplane bundle of ${\mathbf P}(T{\mathcal X})^{\lor}$. 
We have the following exact sequence of $G$-equivariant holomorphic vector bundles on ${\mathbf P}(T{\mathcal X})^{\lor}$
$$
{\mathcal S}^{\lor} \colon 0 \longrightarrow  {\mathcal U} \longrightarrow (\varPi^{\lor})^{*}T{\mathcal X} \longrightarrow {\mathcal H} \longrightarrow 0.
$$
\par
Let $h_{\mathcal U}$ be the Hermitian metric on ${\mathcal U}$ induced from $(\varPi^{\lor})^{*}h_{\mathcal X}$, 
and let $h_{\mathcal H}$ be the Hermitian metric on ${\mathcal H}$ induced from 
$(\varPi^{\lor})^{*}h_{\mathcal X}$ by the $C^{\infty}$-isomorphism ${\mathcal H}\cong{\mathcal U}^{\perp}$.
Set $\overline{\mathcal U}:=({\mathcal U},h_{\mathcal U})$ and $\overline{\mathcal H}:=({\mathcal H},h_{\mathcal H})$.
On ${\mathcal X}\setminus\Sigma_{\pi}$, we have
$(T{\mathcal X}/C,h_{{\mathcal X}/C})=\gamma^{*}\overline{\mathcal U}$.
\par
Let 
$$
{\mathcal L} := 
{\mathcal O}_{{\mathbf P}(\Omega^{1}_{\mathcal X}\otimes\pi^{*}TC)}(-1) \subset \varPi^{*}(\Omega^{1}_{\mathcal X}\otimes\pi^{*}TC)
$$
be the tautological line bundle over ${\mathbf P}(\Omega^{1}_{\mathcal X}\otimes\pi^{*}TC)$. Let $h_{C}$ be a Hermitian metric on $C$. 
Let $h_{\Omega^{1}_{\mathcal X}}$ be the Hermitian metric on $\Omega^{1}_{\mathcal X}$ induced from $h_{\mathcal X}$.
Let $h_{\mathcal L}$ be the Hermitian metric on ${\mathcal L}$ induced from the metric $\varPi^{*}(h_{\Omega^{1}_{\mathcal X}}\otimes\pi^{*}h_{C})$ 
via the inclusion ${\mathcal L}\subset\varPi^{*}(\Omega^{1}_{\mathcal X}\otimes\pi^{*}TC)$. 
\par
Since $\Sigma_{\pi}$ is a proper subvariety of ${\mathcal X}$, the Gauss maps $\nu$ and $\gamma$ extend to rational maps
$\nu\colon{\mathcal X}\dashrightarrow{\bf P}(\Omega^{1}_{\mathcal X}\otimes\pi^{*}TC)$
and $\gamma\colon{\mathcal X}\dashrightarrow{\mathbf P}(T{\mathcal X})^{\lor}$.
By \cite[Th.\,13.2]{BierstoneMilman97}, there exist a projective algebraic manifold $\widetilde{\mathcal X}$, 
a normal crossing divisor $E\subset\widetilde{\mathcal X}$, 
a birational holomorphic map $q\colon\widetilde{\mathcal X}\to{\mathcal X}$ with $E=q^{-1}(\Sigma_{\pi})$
and holomorphic maps 
$\widetilde{\nu}\colon\widetilde{\mathcal X} \to {\mathbf P}(\Omega^{1}_{\mathcal X}\otimes\pi^{*}TC)$,
$\widetilde{\gamma}\colon\widetilde{\mathcal X} \to {\mathbf P}(T{\mathcal X})^{\lor}$ with the following properties:
\begin{itemize}
\item[(a)]
The $G$-action on ${\mathcal X}$ lifts to a $G$-action on $\widetilde{\mathcal X}$ and 
$q|_{\widetilde{\mathcal X}\setminus E} \colon \widetilde{\mathcal X}\setminus E\to{\mathcal X}\setminus\Sigma_{\pi}$ is a $G$-equivariant isomorphism.
\item[(b)]
$(\pi\circ q)^{-1}(b)$ is a normal crossing divisor of $\widetilde{\mathcal X}$ for all $b\in\Delta$.
\item[(c)]
$\widetilde{\nu}=\nu\circ q$ and $\widetilde{\gamma}=\gamma\circ q$ on $\widetilde{\mathcal X}\setminus E$.
\end{itemize}
Then $\widetilde{\nu}=\widetilde{\gamma}$ under the canonical isomorphism 
${\mathbf P}(\Omega^{1}_{\mathcal X}\otimes\pi^{*}TC)\cong{\mathbf P}(T{\mathcal X})^{\lor}$.
\par
Let $\widetilde{\mathcal X}^{g}_{H}\subset\widetilde{\mathcal X}$ be the proper transform of ${\mathcal X}^{g}_{H}\subset{\mathcal X}$.
Since $\widetilde{\mathcal X}^{g}_{H}\subset(\widetilde{\mathcal X})^{g}$, we get 
$\widetilde{\gamma}(\widetilde{\mathcal X}^{g}_{H})\subset({\bf P}(T{\mathcal X})^{\lor})^{g}$ 
by the $G$-equivariance of $\widetilde{\gamma}$.
Consequently, $g\in G$ preserves the fibers of $(\widetilde{\gamma}^{*}{\mathcal U})|_{\widetilde{\mathcal X}^{g}_{H}}$.

\subsection
{The singularity of equivariant Quillen metrics}
\label{sect:3.3}
\par
Let $\Gamma\subset{\mathcal X}\times C$ be the graph of $\pi$. 
Then $\Gamma$ is a smooth divisor on ${\mathcal X}\times C$ preserved by the $G$-action on ${\mathcal X}\times C$. 
Let $[\Gamma]$ be the holomorphic line bundle on ${\mathcal X}\times C$ associated to $\Gamma$. 
Let $\varsigma_{\Gamma} \in H^{0}({\mathcal X}\times C,[\Gamma])$ be the canonical section of $[\Gamma]$ such that ${\rm div}(\varsigma_{\Gamma})=\Gamma$. 
Then $[\Gamma]$ is endowed with the $G$-equivariant structure such that $\varsigma_{\Gamma}$ is a $G$-invariant section. 
Namely, the $G$-equivariant structure on $[\Gamma]|_{({\mathcal X}\times C) \setminus \Gamma}$ induced by the nowhere vanishing $G$-invariant section
$\varsigma_{\Gamma}$ and the $G$-equivariant structure on $[\Gamma]|_{\Gamma}$ induced from the one on the normal bundle
$N_{\Gamma/{\mathcal X}\times C}$ via  the canonical isomorphism $ [\Gamma]|_{\Gamma} \cong N_{\Gamma/{\mathcal X}\times C}$ are compatible
and give rise to the structure of a $G$-equivariant holomorphic line bundle on $[\Gamma]$.
We identify ${\mathcal X}$ with $\Gamma$ via the natural projection $\Gamma\to{\mathcal X}$.
\par
Let $i\colon\Gamma\hookrightarrow{\mathcal X}\times C$ be the inclusion.
Let $p_{1}\colon{\mathcal X}\times C\to{\mathcal X}$ and $p_{2}\colon{\mathcal X}\times C\to C$ be the projections.
By the $G$-equivariance of $i$, $p_{1}$, $p_{2}$ and the $G$-invariance of $\varsigma_{\Gamma}$, 
we have the exact sequence of $G$-equivariant coherent sheaves on ${\mathcal X}\times C$,
\begin{equation}
\label{eqn:ex:seq:total}
\begin{CD}
0
\longrightarrow
{\mathcal O}_{{\mathcal X}\times C}( [\Gamma]^{-1}\otimes p_{1}^{*}\xi ) 
@>\otimes \varsigma_{\Gamma} >>
{\mathcal O}_{{\mathcal X}\times C}( p_{1}^{*}\xi )
\longrightarrow
i_{*}{\mathcal O}_{\Gamma}( p_{1}^{*}\xi )
\longrightarrow0.
\end{CD}
\end{equation}
\par
Let $\lambda_{G}(p_{1}^{*}\xi)$, $\lambda_{G}([\Gamma]^{-1}\otimes p_{1}^{*}\xi)$, $\lambda_{G}(\xi)$ be the equivariant determinants 
of the cohomology for the bundles $p_{1}^{*}\xi$, $[\Gamma]^{-1}\otimes p_{1}^{*}\xi$ on ${\mathcal X}\times C$ 
and the bundle $\xi$ on ${\mathcal X}$, respectively (cf. Section~\ref{sect:1.3.2}). 
Under the isomorphism $p_{1}^{*}\xi|_{\Gamma}\cong\xi$ induced from the identification $p_{1}\colon\Gamma\to{\mathcal X}$, 
the holomorphic vector bundle $\lambda_{G}$ on $C$ defined as
$$
\begin{array}{ll}
\lambda_{G}
&=
\lambda_{G}\left([\Gamma]^{-1}\otimes p_{1}^{*}\xi\right) \otimes \lambda_{G}(p_{1}^{*}\xi)^{-1} \otimes \lambda_{G}(\xi)
:=
\prod_{W\in\widehat{G}}\lambda_{W},
\\
\lambda_{W}
&:=
\lambda_{W}\left([\Gamma]^{-1}\otimes p_{1}^{*}\xi\right) \otimes \lambda_{W}(p_{1}^{*}\xi)^{-1} \otimes \lambda_{W}(\xi)
\end{array}
$$
carries the canonical holomorphic section $\sigma_{KM}=((\sigma_{KM})_{W})_{W\in\widehat{G}}$ such that 
$(\sigma_{KM})_{W}$ is identified with $1\in H^{0}(C,{\mathcal O}_{C})$ under the canonical isomorphism 
$\lambda_{W}\cong{\mathcal O}_{C}$.
By \eqref{eqn:dir:sum:dir:im}, $\lambda_{W}$ is the determinant of the acyclic complex of coherent sheaves on $C$ 
obtained as the $W$-component of the long exact sequence of the direct image sheaves associated to \eqref{eqn:ex:seq:total}.
Hence $\lambda_{W}$ is canonically isomorphic to ${\mathcal O}_{C}$ (\cite{Bismut95, BismutLebeau91, KnudsenMumford76}). 
Then $\sigma_{\rm KM}$ is admissible.
\par
Let $U\subset S$ be a relatively compact neighborhood of $0\in\Delta$ and set $U^{o}:=U\setminus\{0\}$. 
On $X=\pi^{-1}(S)$, we identify $\pi$ (resp. $d\pi$) with $s\circ\pi$ (resp. $d(s\circ\pi)$). 
Hence $\pi\in{\mathcal O}(X)$ and $d\pi\in H^{0}(X,\Omega^{1}_{X})$ in what follows.
\par
Let $h_{[\Gamma]}$ be a $G$-invariant smooth Hermitian metric on $[\Gamma]$ such that
\begin{equation}
\label{eqn:herm:met}
h_{[\Gamma]}(\varsigma_{\Gamma},\varsigma_{\Gamma})(w,t)
=
\begin{cases}
\begin{array}{lcr}
|\pi(w)-t|^{2}&\hbox{if}&(w,t)\in\pi^{-1}(U)\times U,
\\
1&\hbox{if}&(w,t)\in({\mathcal X}\setminus X)\times U
\end{array}
\end{cases}
\end{equation}
and let $h_{[\Gamma]^{-1}}$ be the metric on $[\Gamma]^{-1}$ induced from $h_{[\Gamma]}$.
\par
For $g\in G$, let $\|\cdot\|_{Q,\lambda_{G}(\xi)}(g)$ be the equivariant Quillen metric on $\lambda_{G}(\xi)$ with respect to 
$h_{{\mathcal X}/C}$, $h_{\xi}$.
Let $\|\cdot\|_{Q,\lambda_{G}([\Gamma]^{-1}\otimes p_{1}^{*}\xi)}(g)$ (resp. $\|\cdot\|_{Q,\lambda_{G}(p_{1}^{*}\xi)}(g)$) be 
the equivariant Quillen metric on $\lambda_{G}([\Gamma]^{-1}\otimes p_{1}^{*}\xi)$ (resp. $\lambda_{G}(p_{1}^{*}\xi)$) 
with respect to $h_{\mathcal X}$, $h_{[\Gamma]^{-1}}\otimes h_{\xi}$ (resp. $h_{\mathcal X}$, $h_{\xi}$).
Let $\|\cdot\|_{Q,\lambda_{G}}(g)$ be the equivariant Quillen metric on $\lambda_{G}$ defined as the component-wise tensor product of 
those on $\lambda_{G}([\Gamma]^{-1}\otimes p_{1}^{*}\xi)$, $\lambda_{G}(p_{1}^{*}\xi)^{-1}$, $\lambda_{G}(\xi)$.
\par
For the germ $(\pi\colon({\mathcal X},X_{0})\to(S,0),\xi)$, we define its topological invariant by
$$
{\frak a}_{\pi,\xi}(g)
=
\int_{E_{0}\cap\widetilde{\mathcal X}^{g}_{H}}
\widetilde{\gamma}^{*} 
\left\{ \frac{1-{\rm Td}({\mathcal H})^{-1}}{c_{1}({\mathcal H})} \right\}\, q^{*}\{ {\rm Td}_{g}(T{\mathcal X}){\rm ch}_{g}(\xi) \}
-
\int_{{\mathcal X}^{g}_{V}\cap X_{0}} {\rm Td}_{g}(T{\mathcal X}){\rm ch}_{g}(\xi).
$$

\begin{theorem}
\label{Theorem4.1}
For any $g\in G$, the following identity of functions on $S^{o}$ holds:
$$
\log \left\| \sigma_{KM} \right\|^{2}_{Q,\lambda_{G}}(g) \equiv_{\mathcal B} {\frak a}_{\pi,\xi}(g)\,\log |s|^{2}.
$$
\end{theorem}

\begin{pf}
We follow \cite[Sect.\,5]{Bismut97}, \cite[Th.\,6.3]{Yoshikawa04}, \cite[Th.\,5.1]{Yoshikawa07}, \cite[Appendix]{Imaike24}. 
The proof is quite parallel to those of \cite[Th.\,5.1]{Yoshikawa07}, \cite[Th.\,A.3]{Imaike24}. 
The primary differences from \cite[Th.\,5.1]{Yoshikawa07} arise from the presence of both horizontal and vertical components 
 (${\mathcal X}^{g}_{H}$ and ${\mathcal X}^{g}_{V}$) in the fixed locus ${\mathcal X}^{g}$. 
 These two components yield distinct contributions to the singularity of the equivariant Quillen metric.
In \cite[Th.\,A.3]{Imaike24}, $g$ was assumed to be an involution. For the sake of completeness, we give a detailed proof.

\par{\em (Step 1) }
Let $[X_{s}]=[\Gamma]|_{X_{s}}$ be the holomorphic line bundle on ${\mathcal X}$ associated to the divisor $X_{s}$. 
The canonical section of $[X_{s}]$ is defined as $\varsigma_{s}=\varsigma_{\Gamma}|_{{\mathcal X}\times\{s\}}\in H^{0}({\mathcal X},[X_{s}])$. 
Then ${\rm div}(\varsigma_{s})=X_{s}$. Let $i_{s}\colon X_{s}\hookrightarrow{\mathcal X}$ be the natural embedding.
By \eqref{eqn:ex:seq:total}, we get the exact sequence of $G$-equivariant coherent sheaves on ${\mathcal X}$,
\begin{equation}
\label{eqn:ex:seq}
\begin{CD}
0
\longrightarrow
{\mathcal O}_{\mathcal X}([X_{s}]^{-1}\otimes\xi) @>\otimes\varsigma_{s}>> {\mathcal O}_{\mathcal X}(\xi) \longrightarrow 
(i_{s})_{*}{\mathcal O}_{X_{s}}(\xi)
\longrightarrow
0
\end{CD}
\end{equation}
and hence the canonical isomorphism 
$(\lambda_{G})_{s} \cong \lambda_{G}([X_{s}]^{-1}\otimes\xi)\otimes\lambda_{G}(\xi)^{-1}\otimes\lambda_{G}(\xi_{s})$.
\par
Set $h_{[X_{s}]}=h_{[\Gamma]}|_{{\mathcal X}\times\{s\}}$, which is a $G$-invariant Hermitian metric on $[X_{s}]$. 
Let $h_{[X_{s}]^{-1}}$ be the $G$-invariant Hermitian metric on $[X_{s}]^{-1}$ induced from $h_{[X_{s}]}$.
Let $N_{s}=N_{X_{s}/{\mathcal X}}$ (resp. $N_{s}^{*}=N^{*}_{X_{s}/{\mathcal X}}$) be the normal (resp. conormal) bundle of 
$X_{s}$ in ${\mathcal X}$. 
Then $d\pi|_{X_{s}}\in H^{0}(X_{s},N_{s}^{*})$ generates $N^{*}_{s}$ for $s\in S^{o}$ and $d\pi|_{X_{s}}$ is $G$-invariant 
(cf. \cite[Eq.\,(2.2)]{Bismut95}). 
Let $a_{N^{*}_{s}}$ be the $G$-invariant Hermitian metric on $N^{*}_{s}$ defined by $a_{N^{*}_{s}}(d\pi|_{X_{s}},d\pi|_{X_{s}})=1$.
Let $a_{N_{s}}$ be the $G$-invariant Hermitian metric on $N_{s}$ induced from $a_{N^{*}_{s}}$. 
We have the equality $c_{1}(N_{s},a_{N_{s}})=0$ for $s\in U^{o}$.
By \cite[Proof of Th.\,5.1 Step 1]{Yoshikawa07}, the $G$-invariant metrics $h_{[X_{s}]^{-1}}\otimes h_{\xi}$ 
and $h_{\xi}$ verify assumption (A) of Bismut \cite[Def.1.5]{Bismut90} with respect to $a_{N_{s}}$ and $h_{\xi}|_{X_{s}}$.

\par{\em (Step 2) }
Let ${\mathcal E}_{s}$ be the exact sequence of $G$-equivariant holomorphic vector bundles on $X_{s}$ 
defined as ${\mathcal E}_{s}\colon0\to TX_{s}\to T{\mathcal X}|_{X_{s}}\to N_{s}\to0$.
By \cite[I, Th.\,1.29]{BGS88}, \cite[Th.\,1.2.2]{GilletSoule90}, \cite[Th.\,3.4]{KohlerRoessler01}, we have the Bott-Chern class
$\widetilde{\rm Td}_{g}({\mathcal E}_{s};h_{X_{s}},h_{\mathcal X},a_{N_{s}})\in\widetilde{A}_{X^{g}_{s}}$ 
such that (we follow the sign convention in \cite{GilletSoule90})
$$
dd^{c}\widetilde{\rm Td}_{g}({\mathcal E}_{s};h_{X_{s}},h_{\mathcal X},a_{N_{s}})
=
{\rm Td}_{g}(TX_{s},h_{X_{s}})\,{\rm Td}(N_{s},a_{N_{s}})|_{X^{g}_{s}} - {\rm Td}_{g}(T{\mathcal X},h_{\mathcal X}).
$$
Here, to obtain the equality ${\rm Td}_{g}(N_{s},a_{N_{s}})={\rm Td}(N_{s},a_{N_{s}})|_{X_{s}^{g}}$, 
we used the triviality of the $G$-action on $N_{s}|_{X_{s}^{g}}$.
Set $(X_{H}^{g})_{s}:={\mathcal X}_{H}^{g}\cap X_{s}$.
Applying the embedding formula of Bismut \cite[Th.\,0.1]{Bismut95} (see also \cite[Th.\,5.6]{Bismut97}) to the $G$-equivariant embedding
$i_{s}\colon X_{s}\hookrightarrow{\mathcal X}$ and to the exact sequence \eqref{eqn:ex:seq}, for all $s\in U^{o}$, we obtain
\begin{equation}
\label{eqn:im:formula}
\begin{aligned}
\log\|\sigma_{KM}(s)\|^{2}_{Q,\lambda_{G}}(g)
&
=
\int_{({\mathcal X}^{g}_{V}\times\{s\})\amalg({\mathcal X}^{g}_{H}\times\{s\})}
-\frac{{\rm Td}_{g}(T{\mathcal X},h_{\mathcal X})\,{\rm ch}_{g}(\xi,h_{\xi})}{{\rm Td}([\Gamma],h_{[\Gamma]})} 
\log h_{[\Gamma]}(\varsigma_{\Gamma},\varsigma_{\Gamma})
\\
&
\quad
-\int_{(X^{g}_{H})_{s}} \frac{\widetilde{\rm Td}_{g} ({\mathcal E}_{s};\,h_{X_{s}},h_{\mathcal X},a_{N_{s}})\,
{\rm ch}_{g}(\xi,h_{\xi})}{{\rm Td}(N_{s},a_{N_{s}})} +C(g),
\end{aligned}
\end{equation}
where $C(g)$ is a topological constant independent of $s\in U^{o}$.
Here we used the triviality of the $G$-action on $[X_{s}]|_{X_{s}^{g}} \cong N_{X_{s}/X}|_{X_{s}^{g}}$ and the explicit formula for 
the Bott-Chern current \cite[Rem.\,3.5, especially (3.23), Th.\,3.15, Th.\,3.17]{BGS90} 
to obtain the first term of the right hand side of \eqref{eqn:im:formula}.
Substituting \eqref{eqn:herm:met} and $c_{1}(N_{s},a_{N_{s}})=0$ into \eqref{eqn:im:formula}, for $s\in U^{o}$, we find
\begin{equation}
\label{eqn:KM}
\begin{aligned}
\,
&
\log\|\sigma_{KM}(s)\|^{2}_{Q,\lambda_{G}}(g)
\equiv_{\mathcal B}
-
\int_{{\mathcal X}^{g}_{H}\times\{s\}} {\rm Td}_{g}(T{\mathcal X},h_{\mathcal X})\,{\rm ch}_{g}(\xi,h_{\xi}) \log|\pi-s|^{2}
\\
&
-\int_{({\mathcal X}^{g}_{V}\cap X_{0})\times\{s\}} {\rm Td}_{g}(T{\mathcal X})\,{\rm ch}_{g}(\xi)\log|s|^{2}
-\int_{(X^{g}_{H})_{s}} \widetilde{\rm Td}_{g}({\mathcal E}_{s};\,h_{X_{s}},h_{\mathcal X},a_{N_{s}})\,{\rm ch}_{g}(\xi,h_{\xi})
\\
&\equiv_{\mathcal B} -\{\int_{{\mathcal X}^{g}_{V}\cap X_{0}}{\rm Td}_{g}(T{\mathcal X})\,{\rm ch}_{g}(\xi)\}\,\log|s|^{2}
-
\int_{(X^{g}_{H})_{s}} \widetilde{\rm Td}_{g}({\mathcal E}_{s};\,h_{X_{s}},h_{\mathcal X},a_{N_{s}})\,{\rm ch}_{g}(\xi,h_{\xi}),
\end{aligned}
\end{equation}
where we used the equalities $h_{[\Gamma]}(\varsigma_{\Gamma},\varsigma_{\Gamma})|_{X_{0}\times\{s\}}=|s|^{2}$ and 
${\mathcal X}_{V}^{g}\cap\pi^{-1}(U)={\mathcal X}_{V}^{g}\cap X_{0}$ to get the first equality and \cite[Th.\,9.1]{Yoshikawa07} 
to get the second equality.

\par{\em (Step 3) }
Let $h_{N_{s}}$ be the Hermitian metric on $N_{s}$ induced from $h_{\mathcal X}$ by the smooth isomorphism $N_{s}\cong(TX_{s})^{\perp}$. 
Let $\widetilde{\rm Td}(N_{s};\,a_{N_{s}},h_{N_{s}})\in\widetilde{A}_{X_{s}}$ be the Bott--Chern secondary class 
\cite[I e)]{BGS88}, \cite[Sect.\,1.2.4]{GilletSoule90} such that
$$
dd^{c} \widetilde{\rm Td}( N_{s}; a_{N_{s}}, h_{N_{s}} ) = {\rm Td}( N_{s}, a_{N_{s}} ) - {\rm Td}( N_{s}, h_{N_{s}} ).
$$
Since $G$ acts trivially on $N_{s}|_{X_{s}^{g}}$, we deduce from \cite[I, Props.\,1.3.2 and 1.3.4]{GilletSoule90} and 
\cite[Th.\,3.4, especially the formula for $\widetilde{\phi_{\zeta}\phi_{\eta}}$]{KohlerRoessler01} that
\begin{equation}
\label{eqn:eq:Todd}
\begin{aligned}
\,&
\widetilde{\rm Td}_{g}({\mathcal E}_{s};\,h_{X_{s}},h_{\mathcal X},a_{N_{s}})
=
\widetilde{\rm Td}_{g}({\mathcal E}_{s};\,h_{X_{s}},h_{\mathcal X},h_{N_{s}}) 
+ {\rm Td}_{g}(TX_{s},h_{X_{s}})\,\widetilde{\rm Td}(N_{s};\,a_{N_{s}},h_{N_{s}})
\\
&=
\widetilde{\rm Td}_{g}({\mathcal E}_{s};\,h_{X_{s}},h_{\mathcal X},h_{N_{s}})
+
\gamma^{*}{\rm Td}_{g}(\overline{\mathcal U})\, 
\nu^{*} \left\{ \frac{1-{\rm Td}(-c_{1}(\overline{\mathcal L}))}{-c_{1}(\overline{\mathcal L})} \right\} \log\|d\pi\|^{2}
|_{(X^{g}_{H})_{s}}
\\
&=
\left[
\gamma^{*}\widetilde{\rm Td}_{g} ( {\mathcal S}^{\lor};\, h_{\mathcal U},(\varPi^{\lor})^{*}h_{\mathcal X},h_{\mathcal H} )
+
\gamma^{*}{\rm Td}_{g}( \overline{\mathcal U} ) 
\nu^{*} \left\{ \frac{1-{\rm Td}(-c_{1}(\overline{\mathcal L}))}{-c_{1}(\overline{\mathcal L})} \right\} \log\|d\pi\|^{2}
\right]
|_{(X^{g}_{H})_{s}}.
\end{aligned}
\end{equation}
To get the second equality, we used \cite[Eq.\,(13)]{Yoshikawa07} and $(TX_{s},h_{X_{s}})=\gamma^{*}\overline{\mathcal U}|_{X_{s}}$; 
to get the third equality, we used the functoriality of the Bott--Chern class \cite[Th.\,3.4]{KohlerRoessler01} and 
$({\mathcal E}_{s},h_{X_{s}},h_{\mathcal X},h_{N_{s}}) = 
\gamma^{*}({\mathcal S}^{\lor},h_{\mathcal U},(\varPi^{\lor})^{*}h_{\mathcal X},h_{\mathcal H})|_{X_{s}}$.
Set $\widetilde{\pi}:=\pi\circ q$.
Substituting \eqref{eqn:eq:Todd} into \eqref{eqn:KM} and using $[(\pi|_{X^{g}_{H}})_{*}\omega]^{(0)}\equiv_{\mathfrak B} 0$ 
for any smooth differential form $\omega$ on $X^{g}_{H}$ (cf. \cite{Barlet82}), 
we deduce from the same argument as in \cite[p.74 l.1-l.13]{Yoshikawa07} that
\begin{equation}
\label{eqn:sing:KM}
\begin{aligned}
\,&
\log \left\| \sigma_{KM} \right\|^{2}_{Q,\lambda_{G}}(g) 
\equiv_{\mathcal B}
-\left\{ \int_{{\mathcal X}^{g}_{V}\cap X_{0}}{\rm Td}_{g}(T{\mathcal X})\,{\rm ch}_{g}(\xi) \right\} \log|s|^{2}
\\
&
\qquad+
(\widetilde{\pi}|_{{\mathcal X}^{g}_{H}})_{*}
\left[ 
\widetilde{\gamma}^{*} {\rm Td}_{g}( \overline{\mathcal U} ) 
\widetilde{\nu}^{*} \left\{ 
\frac{{\rm Td}(-c_{1}(\overline{\mathcal L}))-1}{-c_{1}(\overline{\mathcal L})} \right\} q^{*}{\rm ch}_{g}(\xi,h_{\xi})\,(q^{*}\log\|d\pi\|^{2})
\right]^{(0)}.
\end{aligned}
\end{equation}
By \cite[Cor.\,4.6]{Yoshikawa07} applied to the second term of the right hand side of \eqref{eqn:sing:KM}, we get
$$
\begin{aligned}
\log \left\| \sigma_{KM} \right\|^{2}_{Q,\lambda_{G}}(g)
&\equiv_{\mathcal B}
\left[
\int_{\widetilde{\mathcal X}^{g}_{H}\cap E_{0}}
\widetilde{\gamma}^{*} \left\{ {\rm Td}_{g}({\mathcal U}) \frac{{\rm Td}({\mathcal H})-1}{c_{1}({\mathcal H})} \right\} q^{*}{\rm ch}_{g}(\xi)
\right]
\log|s|^{2}
\\
&\qquad
-\left\{ \int_{{\mathcal X}^{g}_{V}\cap X_{0}} {\rm Td}_{g}( T{\mathcal X} ) {\rm ch}_{g}( \xi ) \right\} \log|s|^{2}
=
{\mathfrak a}_{\pi, \xi}(g)\,\log | s |^{2},
\end{aligned}
$$
where the last equality follows from the identity ${\rm Td}_{g}({\mathcal U}) {\rm Td}({\mathcal H}) = (\varPi^{\lor})^{*}{\rm Td}_{g}(T{\mathcal X})$.
This completes the proof.
\end{pf}

\begin{theorem}
\label{Theorem1.1}
For any $g\in G$, the following identity holds:
$$
\log \left\| \sigma \right\|^{2}_{Q,\lambda_{G}(\xi)}(g) \equiv_{{\mathcal B}} {\mathfrak a}_{\pi, \xi}(g) \log | s |^{2}.
$$
\end{theorem}

\begin{pf}
We follow \cite[Th.\,5.9]{Bismut97}. There exist admissible holomorphic sections
$$
\alpha=(\alpha_{W})_{W\in\widehat{G}}\in\Gamma(U,\lambda_{G}(p_{1}^{*}\xi)),
\qquad
\beta=(\beta_{W})_{W\in\widehat{G}}\in\Gamma(U,\lambda_{G}([\Gamma]^{-1}\otimes p_{1}^{*}\xi))
$$
such that $\sigma_{KM}=\beta\otimes\alpha^{-1}\otimes\sigma$ on $S$, i.e., 
$(\sigma_{KM})_{W}=\beta_{W}\otimes\alpha_{W}^{-1}\otimes\sigma_{W}$ for all $W\in\widehat{G}$. Then
$$
\begin{aligned}
\log \|\sigma\|^{2}_{Q,\lambda_{G}(\xi)}(g)
&=
\log \|\sigma_{KM}\|^{2}_{Q,\lambda_{G}}(g) + \log \|\alpha\|^{2}_{Q,\lambda_{G}(p_{1}^{*}\xi)}(g)
-
\log \|\beta\|^{2}_{Q,\lambda_{G}([\Gamma]\otimes p_{1}^{*}\xi)}(g)
\\
&\equiv_{\mathcal B}
{\mathfrak a}_{\pi,\xi}(g)\,\log | s |^{2}
\end{aligned}
$$
by Theorems~\ref{thm:smooth:eq:Quillen} and \ref{Theorem4.1}. This completes the proof.
\end{pf}

\begin{corollary}
\label{cor:curv:current}
For any $g \in G$, the following equation of $(1,1)$-currents on $S$ holds:
$$
-dd^{c} \log \left\| \sigma \right\|_{Q,\lambda_{G}(\xi)}^{2}(g)
=
\pi_{*} \left[ {\rm Td}_{g}(T{\mathcal X}/C,h_{{\mathcal X}/C})\,{\rm ch}_{g}(\xi,h_{\xi}) \right]^{(1,1)} - {\mathfrak a}_{\pi,\xi}(g)\, \delta_{0},
$$
where $\delta_{0}$ denotes the Dirac $\delta$-current supported at the origin.
\end{corollary}

\begin{pf}
The result follows from the curvature formula \eqref{eqn:curv:eq:Q} and Theorem~\ref{Theorem1.1}.
\end{pf}

\subsection
{The case of twisted relative canonical bundle}
\label{sect:3.4}
\par
In this subsection, we study the behavior of the equivariant Quillen metric for the vector bundle $K_{{\mathcal X}/C}(\xi)$.
We set $\Omega_{{\mathcal X}/C}^{1} := \Omega_{\mathcal X}^{1}/\pi^{*}\Omega_{C}^{1}$ and
$\Omega_{{\mathcal X}/C}^{q}:=\Lambda^{q}\Omega_{{\mathcal X}/C}^{1}$. Recall that
$K_{{\mathcal X}/C} = K_{\mathcal X} \otimes (\pi^{*}K_{C})^{-1} = K_{\mathcal X} \otimes (\pi^{*}\Omega^{1}_{C})^{-1}$.
On ${\mathcal X}\setminus\Sigma_{\pi}$, there is a canonical isomorphism $\Omega^{n}_{{\mathcal X}/C} \cong K_{{\mathcal X}/C}$ 
induced by the short exact sequence $0 \to \pi^{*}\Omega_{C}^{1} \to \Omega_{\mathcal X}^{1} \to \Omega_{{\mathcal X}/C}^{1} \to 0$.
The holomorphic vector bundle $\Omega^{q}_{{\mathcal X}/C}$ on ${\mathcal X}\setminus\Sigma_{\pi}$ is endowed with 
the Hermitian metric induced from $h_{{\mathcal X}/C}$.
Since $\pi^{*}\Omega_{C}^{1}\subset\Omega_{\mathcal X}^{1}$, $K_{\mathcal X}$ and $K_{{\mathcal X}/C}$ are 
endowed with the Hermitian metrics $h_{K_{\mathcal X}}$ and 
$h_{K_{{\mathcal X}/C}} := h_{K_{\mathcal X}}\otimes h_{\pi^{*}(\Omega_{C}^{1})^{\lor}}$ induced from $h_{\mathcal X}$, respectively. 
Then the canonical isomorphism $\Omega^{n}_{{\mathcal X}/C} \cong K_{{\mathcal X}/C}$ is an isometry. 
Note that $h_{K_{{\mathcal X}/C}}$ is divergent on $\Sigma_{\pi}$.
\par
For $g\in G$, let $\|\cdot\|_{Q,\lambda_{G}(K_{{\mathcal X}/C}(\xi))}(g)$ be the equivariant Quillen metric on 
$\lambda_{G}(K_{{\mathcal X}/C}(\xi))$ with respect to $h_{{\mathcal X}/C}$, $h_{\xi}$, $h_{K_{{\mathcal X}/C}}$.
Let $\varsigma = (\varsigma_{W})_{W\in\widehat{G}}$ be an admissible holomorphic section of 
$\lambda_{G}(K_{{\mathcal X}/C}(\xi))$ on $S$.
Define a complex number $\alpha_{\pi, K_{X/S}(\xi)}(g)$ by 
\begin{equation}
\label{eqn:log:coeff:tw}
\begin{aligned}
\alpha_{\pi,K_{X/S}(\xi)}(g)
&:=
\int_{E_{0}\cap\widetilde{\mathcal X}^{g}_{H}}
\widetilde{\gamma}^{*}
\left\{
\frac{{\rm Td}({\mathcal H}^{\lor})^{-1}-1}{c_{1}({\mathcal H}^{\lor})}
\right\}\,
q^{*}\{{\rm Td}_{g}(T{\mathcal X}){\rm ch}_{g}(K_{\mathcal X}(\xi))\}
\\
&\quad
-\int_{{\mathcal X}^{g}_{V}\cap X_{0}}{\rm Td}_{g}(T{\mathcal X}){\rm ch}_{g}(K_{\mathcal X}(\xi)).
\end{aligned}
\end{equation}

\begin{theorem}
\label{Thm:Sing:Q:adj}
For any $g\in G$, the following identity of functions on $S^{o}$ holds:
$$
\log \left\| \varsigma \right\|^{2}_{Q,\lambda_{G}(K_{{\mathcal X}/C}(\xi))}(g) \equiv_{{\mathcal B}} \alpha_{\pi,K_{X/S}(\xi)}(g)\, \log | s |^{2}.
$$
\end{theorem}

\begin{pf}
By Theorem~\ref{Theorem1.1} applied to $K_{\mathcal X}(\xi)$, we have
\begin{equation}
\label{eqn:sing:eq:Q}
\log \left\| \varsigma \right\|^{2}_{Q,\lambda_{G}(K_{\mathcal X}(\xi))}(g) 
\equiv_{{\mathcal B}} {\mathfrak a}_{\pi, K_{\mathcal X}(\xi)}(g)\,\log | s |^{2}.
\end{equation}
Since $\pi^{*}\Omega^{1}_{C}$ is generated by the $G$-invariant section $d\pi$ on $X\setminus\Sigma_{\pi}$, 
we have an isomorphism of $G$-equivariant holomorphic Hermitian line bundles
$(K_{{\mathcal X}/C},h_{K_{{\mathcal X}/C}}) \cong (K_{\mathcal X},\|d\pi\|^{-2}h_{K_{\mathcal X}})$ on $X$. 
Set $h_{K_{\mathcal X}(\xi)} := h_{K_{\mathcal X}}\otimes h_{\xi}$ and $h_{K_{{\mathcal X}/C}(\xi)} := h_{K_{{\mathcal X}/C}}\otimes h_{\xi}$.
By the anomaly formula for (equivariant) Quillen metrics \cite[Th.\,2.5]{Bismut95}, \cite[I, Th.\,0.3]{BGS88}, 
\begin{equation}
\label{eqn:anomaly:eq:Q}
\begin{aligned}
\,
&
\log\|\varsigma\|^{2}_{Q,\lambda_{G}(K_{{\mathcal X}/C}(\xi))}(g)
=
\log\|\varsigma\|^{2}_{Q,\lambda_{G}(K_{\mathcal X}(\xi))}(g) 
+ 
\log\frac{\|\cdot\|^{2}_{Q,\lambda_{G}(K_{{\mathcal X}/C}(\xi))}(g)}{\|\cdot\|^{2}_{Q,\lambda_{G}(K_{\mathcal X}(\xi))}(g)}
\\
&=
\log\|\varsigma\|^{2}_{Q,\lambda_{G}(K_{\mathcal X}(\xi))}(g)
+
\pi_{*}\left( {\rm Td}_{g}(T{\mathcal X}/C,h_{{\mathcal X}/C})
\widetilde{{\rm ch}}_{g}(K_{\mathcal X}(\xi);h_{K_{\mathcal X}(\xi)},h_{K_{{\mathcal X}/C}(\xi)}) \right)^{(0)}.
\end{aligned}
\end{equation}
Here $\widetilde{{\rm ch}}_{g}(K_{{\mathcal X}}(\xi);h_{K_{\mathcal X}(\xi)},h_{K_{{\mathcal X}/C}(\xi)})$ is the Bott-Chern class such that
$$
dd^{c} \widetilde{{\rm ch}}_{g}(K_{\mathcal X}(\xi);h_{K_{\mathcal X}(\xi)},h_{K_{{\mathcal X}/C}(\xi)})
=
{\rm ch}_{g}(K_{\mathcal X}(\xi),h_{K_{\mathcal X}(\xi)}) - {\rm ch}_{g}(K_{\mathcal X}(\xi),h_{K_{{\mathcal X}/C}(\xi)}).
$$
Since
$(K_{{\mathcal X}/C}(\xi),h_{K_{{\mathcal X}/C}(\xi)}) \cong (K_{\mathcal X}(\xi),\|d\pi\|^{-2}h_{K_{\mathcal X}(\xi)})$ 
and $-dd^{c}\log\|d\pi\|^{2}=\gamma^{*}c_{1}({\mathcal L},h_{\mathcal L})$, 
we deduce from \cite[I, (1.2.5.1), (1.3.1.2)]{GilletSoule90} (see also \cite[Eqs.\,(3.7), (5.5)]{FLY08}) that
\begin{equation}
\label{eqn:eq:Chern}
\begin{aligned}
\,&
\widetilde{{\rm ch}}_{g}(K_{\mathcal X}(\xi);h_{K_{\mathcal X}(\xi)},h_{K_{{\mathcal X}/C}(\xi)})
=
\widetilde{{\rm ch}}_{g}(K_{\mathcal X};h_{K_{\mathcal X}},\|d\pi\|^{-2}h_{K_{\mathcal X}}) {\rm ch}_{g}(\xi,h_{\xi})
\\
&=
{\rm ch}_{g}(K_{\mathcal X},h_{K_{\mathcal X}})
\frac{e^{-\gamma^{*}c_{1}({\mathcal L},h_{\mathcal L})}-1}{-\gamma^{*}c_{1}({\mathcal L},h_{\mathcal L})}\,
(-\log\|d\pi\|^{2}) \wedge {\rm ch}_{g}(\xi,h_{\xi})
\\
&=
-{\rm ch}_{g}(K_{\mathcal X}(\xi),h_{K_{\mathcal X}(\xi)}) 
\frac{1-e^{-\gamma^{*}c_{1}({\mathcal L},h_{\mathcal L})}}{\gamma^{*}c_{1}({\mathcal L},h_{\mathcal L})}\,
\log\|d\pi\|^{2}.
\end{aligned}
\end{equation}
Since ${\rm Td}_{g}(T{\mathcal X}/C, h_{{\mathcal X}/C})=\gamma^{*}{\rm Td}_{g}({\mathcal U})$, 
it follows from \eqref{eqn:anomaly:eq:Q}, \eqref{eqn:eq:Chern}, \cite[Cor.\,4.6]{Yoshikawa07} that
\begin{equation}
\label{eqn:ratio:eq:Q}
\begin{aligned}
\,&
\log \left[ \|\cdot\|^{2}_{Q,\lambda_{G}(K_{{\mathcal X}/C}(\xi))}(g) / \|\cdot\|^{2}_{Q,\lambda_{G}(K_{\mathcal X}(\xi))}(g) \right]
\\
&=
-\pi_{*}\left( 
\gamma^{*}{\rm Td}_{g}({\mathcal U}, h_{\mathcal U}) {\rm ch}_{g}(K_{\mathcal X}(\xi), h_{K_{\mathcal X}(\xi)}) 
\frac{1-e^{-\gamma^{*}c_{1}({\mathcal L},h_{\mathcal L})}}{\gamma^{*}c_{1}({\mathcal L},h_{\mathcal L})} \log \| d\pi \|^{2}
\right)
\\
&\equiv_{\mathcal B}
-\left\{
\int_{E_{0}\cap\widetilde{\mathcal X}_{H}^{g}}
\widetilde{\gamma}^{*}{\rm Td}_{g}({\mathcal U})q^{*}{\rm ch}_{g}(K_{\mathcal X}(\xi))
\frac{1-e^{-\widetilde{\gamma}^{*}c_{1}({\mathcal L})}}{\widetilde{\gamma}^{*}c_{1}({\mathcal L})}\,
\right\}
\log|s|^{2}
\\
&\equiv_{\mathcal B}
-\left[
\int_{E_{0}\cap\widetilde{\mathcal X}_{H}^{g}} 
\widetilde{\gamma}^{*}\left\{ {\rm Td}_{g}({\mathcal U})
\frac{e^{c_{1}({\mathcal H})}-1}{c_{1}({\mathcal H})} \right\} q^{*}{\rm ch}_{g}(K_{\mathcal X}(\xi))
\right]
\log|s|^{2}.
\end{aligned}
\end{equation}
Here, to obtain the last equality, we used the relation $c_{1}({\mathcal L}) = -c_{1}({\mathcal H}) + \pi^{*}c_{1}(C)$.
Since ${\rm Td}_{g}({\mathcal U}){\rm Td}({\mathcal H})=(\varPi^{\lor})^{*}{\rm Td}_{g}(T{\mathcal X})$ and
$$
\frac{1-{\rm Td}(x)^{-1}}{x}-e^{x}{\rm Td}(x)^{-2}=\frac{{\rm Td}(-x)^{-1}-1}{-x}
\qquad\text{with}\qquad
{\rm Td}(x) := \frac{x}{1-e^{-x}},
$$
we get
\begin{equation}
\label{eqn:dominant:term}
{\mathfrak a}_{\pi, K_{\mathcal X}(\xi)}(g) 
- 
\int_{E_{0}\cap\widetilde{\mathcal X}_{H}^{g}} \widetilde{\gamma}^{*}
\left\{ {\rm Td}_{g}({\mathcal U})\frac{e^{c_{1}({\mathcal H})}-1}{c_{1}({\mathcal H})} \right\} q^{*}{\rm ch}_{g}(K_{\mathcal X}(\xi))
=
\alpha_{\pi, K_{X/S}(\xi)}(g).
\end{equation}
The result follows from \eqref{eqn:sing:eq:Q}, \eqref{eqn:anomaly:eq:Q}, \eqref{eqn:ratio:eq:Q}, \eqref{eqn:dominant:term}. 
This completes the proof.
\end{pf}

\section
{Asymptotic behavior of $L^{2}$-metrics under degenerations}
\label{sect:4}
\par
In the remainder of this article, we make the following:

\medskip
\par\noindent
{\bf Assumption}\quad
{\em $(\xi,h_{\xi})$ is Nakano semi-positive on $X$ and $(S,0)\cong(\varDelta,0)$.}

\medskip
In this section, we determine the behavior of the $L^{2}$-metric $h_{L^{2}}$ on the direct image sheaves $R^{q}\pi_{*}K_{X/S}(\xi)$
as $s \to 0$.
Once fixing a basis of $R^{q}\pi_{*}K_{X/S}(\xi)$ as a free ${\mathcal O}_{S}$-module (after shrinking $S$ if necessary), 
$h_{L^{2}}$ is viewed as a function on $S^{o}$ with values in $(h_{W}^{q}, h_{W}^{q})$-Hermitian matrices, 
where $h_{W}^{q}$ is the rank of $R^{q}\pi_{*}K_{X/S}(\xi)_{W}$. 
To describe the singularity of $h_{L^{2}}$, we fix a semi-stable reduction $f \colon Y \to T$
of $\pi \colon X \to S$ induced by a finite map $\mu \colon T \to S$ and we consider the function $\mu^{*}h_{L^{2}}$ on $T$.
In Theorem~\ref{thm:str:sing:L2}, we will prove that $\mu^{*}h_{L^{2}}$ admits a decomposition 
$\mu^{*}h_{L^{2}}(t) = D(t) \widetilde{H}(t) \overline{D(t)}$, where $D(t)$ is a diagonal matrix of the form 
${\rm diag}( t^{-e_{1}}, \ldots, t^{-e_{h_{W}^{q}}} )$, $e_{i} \in {\mathbf Z}_{\geq0}$, 
the entries of $\widetilde{H}(t)$ lie in $\bigoplus_{k=1}^{n} {\mathbf C} ( \log |t|^{-2} )^{k} \oplus {\mathcal B}(T)$, 
and $\widetilde{H}(t)$ is non-degenerate in the sense that $\widetilde{H}(t) \geq C I_{h_{W}^{q}}$ on $T\setminus\{0\}$ with a constant $C>0$.
This leads to the structure theorem for the singularity of the $L^{2}$-metric on $\lambda(K_{X/S}(\xi))_{W}$.
When $\xi|_{X} \cong {\mathcal O}_{X}$ and $G=\{1\}$, these results reduce to known results in Hodge theory (\cite{Schmid73}, \cite{EFM21}).  
Since the techniques of Hodge theory do not apply when $(\xi, h_{\xi})|_{X}$ is  non-trivial, 
we will employ the theory of harmonic integrals for open K\"ahler manifolds \cite{Takegoshi95}, 
the existence of the asymptotic expansion for fiber integrals \cite{Barlet82, Takayama21, Takayama22}
and the non-degeneracy of the $L^{2}$-metrics \cite{MourouganeTakayama09}.

\subsection
{Semi-stable reduction}
\label{sect:4.1}
\par
Let 
$$
\beta \colon X' \to X 
$$ 
be a log-resolution of $X_{0}$ such that $\beta^{*}X_{0}$ is a normal crossing divisor on $X'$, such that 
$\beta \colon X'\setminus X'_{0} \to X\setminus X_{0}$ is an isomorphism, and such that $\beta$ is a composite of blow-ups 
with non-singular centers. Let $\pi' \colon X' \to S$ be the projection induced from $\pi$.
Let $T$ be the unit disc in ${\mathbf C}$. For $0<\epsilon<1$, we set $T(\epsilon) := \{t\in T;\,|t|<\epsilon\}$ and $T^{o}:=T\setminus\{0\}$.
By the semi-stable reduction theorem \cite[Chap.\,II]{Mumford73}, there is a commutative diagram
\begin{equation}
\label{eqn:ss:red}
\begin{CD}
(Y,Y_{0})@> \rho >> (X'\times_{S}T,X'_{0}) @>{\rm pr}_{1}>>  (X',X'_{0}) @> \beta >> (X,X_{0})
\\
@V f VV @V{\rm pr}_{2}VV  @V \pi' VV @V \pi VV
\\
(T,0) @>{\rm id}_{T}>> (T,0) @>\mu>> (S,0) @>{\rm id}_{S}>> (S,0)
\end{CD}
\end{equation}
such that $Y_{0}=f^{-1}(0)$ is a {\em reduced} normal crossing divisor. Here $\mu\colon(T,0)\to(S,0)$ is given by 
$$
\mu(t)=t^{d}
$$ 
for some $d \in{\mathbf Z}_{>0}$ and 
$$
\rho \colon Y\to X'\times_{S}T
$$ 
is a resolution, which is a projective morphism. Set 
$$
F' := {\rm pr}_{1}\circ \rho \colon Y\to X',
\qquad
F := \beta \circ F' = \beta \circ {\rm pr}_{1}\circ \rho \colon Y \to X.
$$
Since $(F^{*}\xi,h_{F^{*}\xi}:=F^{*}h_{\xi})$ is Nakano semi-positive, $R^{q}f_{*}K_{Y}(F^{*}\xi)$ is a locally free ${\mathcal O}_{T}$-module 
by Lemma~\ref{lemma:local:freeness}. Since $f|_{Y\setminus Y_{0}}\colon Y\setminus Y_{0}\to T^{o}$ is $G$-equivariant, 
$R^{q}f_{*}K_{Y}(F^{*}\xi)|_{T^{o}}$ is viewed as a $G$-equivariant holomorphic vector bundle on $T^{o}$.

\begin{lemma}
\label{lemma:G:equiv:ext}
The $G$-action on $R^{q}f_{*}K_{Y}(F^{*}\xi)|_{T^{o}}$ extends to a holomorphic $G$-action on $R^{q}f_{*}K_{Y}(F^{*}\xi)$.
\end{lemma}

\begin{pf}
Let $\rho'' \colon X''\to X\times_{S}T$ be a $G$-equivariant resolution and set $\varpi'' := {\rm pr}_{2}\circ\rho''$ and 
$\varPi'' := {\rm pr}_{1}\circ\rho''$. 
By the $G$-equivariance of $\varpi'' \colon X'' \to T$, $R^{q}\varpi''_{*}K_{X''}(\varPi^{''*}\xi)$ is a $G$-equivariant holomorphic 
vector bundle on $T$. Since $Y$ is birational to $X''$, considering a resolution of $X\times_{S}T$ dominating both of $X''$ and $Y$,
we get an isomorphism $R^{q}f_{*}K_{Y}(F^{*}\xi)\cong R^{q}\varpi''_{*}K_{X''}(\varPi^{''*}\xi)$ of holomorphic vector bundles 
on $T$ by \cite[Th.\,6.9 (i)]{Takegoshi95}, which induces the desired $G$-action on $R^{q}f_{*}K_{Y}(F^{*}\xi)$.
\end{pf}

\subsection
{Set up}
\label{sect:4.2}
\par
Let $\kappa_{\mathcal X}$ be the K\"ahler form of $h_{\mathcal X}$ and set $\kappa_{X} := \kappa_{\mathcal X}|_{X}$.
We write $t$ for the coordinate of $T\cong\varDelta$ centered at $0$ and we set $\kappa_{T} := i\,dt\wedge d\bar{t}$. 
Then $\kappa_{T}$ is a K\"ahler form on $T$. Define
$$
\kappa_{Y} := F^{*}\kappa_{X} + f^{*}\kappa_{T}.
$$
Then $\kappa_{Y}$ is a smooth $(1,1)$-form on $Y$, which is a K\"ahler form only on $Y\setminus Y_{0}$.
To understand the asymptotic behavior of the $L^{2}$-metric on $R^{q}f_{*}K_{Y/T}(F^{*}\xi)$ near $t=0$ with respect to 
the {\em degenerate} K\"ahler form $\kappa_{Y}$, we need an analogue of Theorem~\ref{Thm:Takegoshi} for 
the Nakano semi-positive vector bundle $(F^{*}\xi,F^{*}h_{\xi})$ on $(Y,\kappa_{Y})$, as well as a bound from below of the $L^{2}$-metric. 
Such an extension of Theorem~\ref{Thm:Takegoshi} and an estimate for the $L^{2}$-metric were obtained by 
Mourougane-Takayama \cite[Prop.\,4.4]{MourouganeTakayama09}, \cite[Lemmas 4.7 and 4.8]{MourouganeTakayama09}. 
However, their results are not directly applicable to our setting for the following reasons. Set
$$
\varSigma' := {\rm Sing}(X'\times_{S}T),
\qquad
U' := (X'\times_{S}T)\setminus\varSigma'.
$$
We may assume by \cite[Chap.\,II]{Mumford73} that $\rho \colon Y\setminus \rho^{-1}(\varSigma)\to U'$ is an isomorphism and that 
$\rho^{-1}(\varSigma)\subset Y_{0}$ is a normal crossing divisor. However, it is not clear from the construction in \cite[Chap.\,II]{Mumford73}
if $\rho \colon Y\to X'\times_{S}T$ is a {\em composite of blowing-ups with non-singular centers} disjoint from $U'$. 
Since this condition is essential in the construction of a sequence of K\"ahler forms on $Y$ approximating $\kappa_{Y}$ 
(cf. \cite[Proof of Prop.\,4.4 Step 1]{MourouganeTakayama09}), we are not able to apply the arguments in \cite[Prop.\,4.4]{MourouganeTakayama09} to the bundle $(F^{*}\xi,h_{F^{*}\xi})$ on $(Y,\kappa_{Y})$ at once. 
Another point is that the non-vanishing \cite[Lemmas 4.7 and 4.8]{MourouganeTakayama09} was proved 
under the assumption that the metric on $Y$ is non-degenerate (when $\dim S=1$). 
Since $\kappa_{Y}$ is degenerate on some components of $Y_{0}$, this non-vanishing result is not applicable to our setting. 
We proceed as follows to resolve these problems.
\par
For the first point, in stead of applying the argument of Mourougane-Takayama to $(Y,\kappa_{Y})$, we apply it to a manifold $Z$ 
dominating $Y$, whose construction is as follows. By \cite[Th.\,13.2]{BierstoneMilman97}, there exists a resolution of singularity 
$r \colon W\to X'\times_{S}T$, which is a composite of blowing-ups with non-singular centers disjoint from $U'$. 
We apply \cite[Lemma 1.3.1]{AKMW02} by setting $X_{1}=W$ and $X_{2}=Y$. Consequently, there exist a resolution 
$$
\widetilde{\rho} \colon Z\to X'\times_{S}T
$$ 
and a projective birational morphism 
$$
\varphi\colon Z\to Y
$$ 
such that $\widetilde{\rho}$ is a composite of blowing-ups with non-singular centers disjoint from $U'$ and 
$$
\widetilde{\rho} = \rho \circ \varphi.
$$ 
In particular, 
$Z\setminus\widetilde{\rho}^{-1}(\varSigma')\cong Y\setminus \rho^{-1}(\varSigma')\cong U'$ and $Z_{0}={\varphi}^{-1}(Y_{0})$ is a normal crossing divisor.
We set 
\begin{equation}
\label{eqn:proj:Z}
\varpi := {\rm pr}_{2} \circ \widetilde{\rho} \colon Z\to T,
\qquad
\widetilde{F} := \beta \circ{\rm pr}_{1} \circ \widetilde{\rho} \colon Z\to X.
\end{equation}
Then 
$$
\widetilde{F} = F \circ \varphi 
\qquad\text{and}\qquad 
\varpi = f \circ \varphi.
$$
In Section~\ref{sect:4.3}, following Mourougane-Takayama \cite[Prop.\,4.4]{MourouganeTakayama09}, we prove a version of 
Theorem~\ref{Thm:Takegoshi} for the vector bundle $K_{Z}(\widetilde{F}^{*}\xi)$ on $Z$ with respect to a degenerate metric, 
which leads to a corresponding result for the vector bundle $K_{Y}(F^{*}\xi)$ on $Y$. 
Once we obtain a Takegoshi type theorem for the cohomology group $H^{q}(Y, K_{Y}(F^{*}\xi))$, we prove that the Hermitian matrix of 
the $L^{2}$-metric $(\cdot,\cdot)_{L^{2},Y_{t}}$ with respect to a holomorphic frame of $R^{q}f_{*}K_{Y/T}(F^{*}\xi)$ admits 
an asymptotic expansion of Barlet-Takayama type (Section~\ref{sect:4.4}). 
To prove the non-degeneracy of the $L^{2}$-metric on $R^{q}f_{*}K_{Y/T}(F^{*}\xi)$, which is exactly the second point mentioned above, 
we fix a K\"ahler form $\kappa'_{Y}$ on $Y'$ with $\kappa_{Y} \leq \kappa'_{Y}$, which yields 
the corresponding $L^{2}$-metric $( \cdot, \cdot)'_{L^{2},Y_{t}}$ on $H^{q}(Y_{t}, K_{Y_{t}}(F^{*}\xi))$. 
In Section~\ref{sect:4.5}, we compare the metrics $( \cdot, \cdot)_{L^{2},Y_{t}}$ and $( \cdot, \cdot)'_{L^{2},Y_{t}}$ to prove that the 
non-degeneracy result \cite[Lemmas~4.7, 4.8]{MourouganeTakayama09} for $( \cdot, \cdot)'_{L^{2},Y_{t}}$ implies a non-degeneracy result 
for the metric $( \cdot, \cdot)_{L^{2},Y_{t}}$.

\subsection
{Representation of cohomology classes by harmonic forms}
\label{sect:4.3}
\par

\begin{proposition}
\label{Proposition:Takegoshi}
For every cohomology class $\Theta\in H^{q}(Y, K_{Y}(F^{*}\xi))$, 
there exists a holomorphic differential form $\theta\in H^{0}(Y,\Omega_{Y}^{n+1-q}(F^{*}\xi))$ such that
$$
\Theta|_{Y\setminus Y_{0}}=[\theta\wedge\kappa_{Y}^{q}]|_{Y\setminus Y_{0}}
$$
as elements of $H^{q}(Y\setminus Y_{0}, K_{Y}(F^{*}\xi))$ and
$$
\theta\wedge f^{*}dt=0,
\qquad
D^{F^{*}\xi, F^{*}h_{\xi}}\theta =0.
$$
\end{proposition}

\begin{pf}
We follow \cite[Prop.\,4.4]{MourouganeTakayama09}.
Since the assertion is obvious for $q=0$, we suppose $q>0$. Then $\varphi^{*}\Theta\in H^{q}(Z, K_{Z}(\widetilde{F}^{*}\xi))$.
Define 
$$
\kappa_{Z} := \widetilde{\rho}^{*}( \beta^{*}\kappa_{X} + ({\rm pr}_{2})^{*}\kappa_{T} ),
$$ 
which is a degenerate K\"ahler form on $Z$ with $\kappa_{Z} = \varphi^{*}\kappa_{Y}$. 
Since $Z$ is obtained from $X'\times_{S}T$ by a composite of blowing-ups with non-singular centers,
we deduce from e.g. \cite[Prop.\,12.4]{Demailly12}, \cite[Proof of Prop.\,4.4 Step 1]{MourouganeTakayama09} that 
there exists a sequence of K\"ahler forms $\{\kappa_{Z,k}\}_{k\geq1}$ on $Z$ converging to $\kappa_{Z}$ uniformly on $Z$
such that $\kappa_{Z,k}=\kappa_{Z}$ on $Z\setminus\varpi^{-1}(T(\frac{1}{k}))$.
By Theorem~\ref{Thm:Takegoshi} applied to the Nakano semi-positive vector bundle $(\widetilde{F}^{*}\xi, \widetilde{F}^{*}h_{\xi})$ 
on the K\"ahler manifold $(Z,\kappa_{Z,k})$, 
there is a holomorphic differential form $\theta_{k}\in \Gamma(Z,\Omega_{Z}^{n+1-q}(\widetilde{F}^{*}\xi))$ with
\begin{equation}
\label{eqn:Takegoshi}
\varphi^{*}\Theta=[\theta_{k}\wedge\kappa_{Z,k}^{q}],
\qquad
(\varpi^{*}dt)\wedge\theta_{k}=0,
\qquad
D^{\widetilde{F}^{*}\xi, F^{*}h_{\xi}} \theta_{k} =0.
\end{equation}
Here $D^{\widetilde{F}^{*}\xi, F^{*}h_{\xi}}$ denotes the Chern connection of $(\widetilde{F}^{*}\xi, \widetilde{F}^{*}h_{\xi})$.
Let $O \subset T$ be an open subset  such that $O \subset T\setminus T(\frac{1}{k})$ for all $k>1$.
Since $\kappa_{Z,k}=\kappa_{Z}$ on $Z\setminus\varpi^{-1}(T(\frac{1}{k}))$, we get $\kappa_{Z,k}=\kappa_{Z}$ on $\varpi^{-1}(O)$.
By \cite[Th.\,5.2 (iv)]{Takegoshi95}, \cite[Proof of Prop.\,4.4 Step 3]{MourouganeTakayama09}, 
the equality $\kappa_{Z,k}|_{\varpi^{-1}(O)}=\kappa_{Z,l}|_{\varpi^{-1}(O)}$ yields that $\theta_{k}|_{\varpi^{-1}(O)}=\theta_{l}|_{\varpi^{-1}(O)}$. 
Hence $\theta_{k}=\theta_{l}$ for all $k,l>1$. We set $\theta_{\infty}:=\theta_{k}$. By \eqref{eqn:Takegoshi}, for all $k>1$, we have
\begin{equation}
\label{eqn:Takegoshi:2}
\varphi^{*}\Theta=[\theta_{\infty}\wedge\kappa_{Z,k}^{q}],
\qquad
(\varpi^{*}dt)\wedge\theta_{\infty}=0,
\qquad
D^{\widetilde{F}^{*}\xi, \widetilde{F}^{*}h_{\xi}} \theta_{\infty} =0.
\end{equation} 
Since $\kappa_{Z,k}=\kappa_{Z}$ on $Z\setminus\varpi^{-1}(T(\frac{1}{k}))$, we get the following equality of cohomology classes:
\begin{equation}
\label{eqn:equality:coh}
\varphi^{*}\Theta|_{Z\setminus\varpi^{-1}(T(\frac{1}{k}))} = [\theta_{\infty}\wedge\kappa_{Z}^{q}]|_{Z\setminus\varpi^{-1}(T(\frac{1}{k}))}
\qquad(k>1).
\end{equation}
Set $Z_{0} := \varpi^{-1}(0)$.
Let $\sigma_{\varphi^{*}\Theta}, \sigma_{ [\theta_{\infty}\wedge\kappa_{Z}^{q}]} \in \Gamma(T^{o}, R^{q}\varpi_{*}K_{Z}(\widetilde{F}^{*}\xi))$ 
be the holomorphic sections corresponding to $\varphi^{*}\Theta|_{Z\setminus Z_{0}}$ and 
$[\theta_{\infty}\wedge\kappa_{Z}^{q}]|_{Z\setminus Z_{0}}$, respectively.
By \eqref{eqn:equality:coh}, we get $\sigma_{\varphi^{*}\Theta} = \sigma_{ [\theta_{\infty}\wedge\kappa_{Z}^{q}]}$ on $T^{o}$.
Since $T^{o}$ is Stein, $H^{q}(Z\setminus Z_{0}, K_{Z}(\widetilde{F}^{*}\xi)) \cong \Gamma(T^{o}, R^{q}\varpi_{*}K_{Z}(\widetilde{F}^{*}\xi))$ 
by the Leray spectral sequence. 
Hence
\begin{equation}
\label{eqn:equality:coh:2}
\varphi^{*}\Theta|_{Z\setminus Z_{0}} = [ \theta_{\infty}\wedge\kappa_{Z}^{q} ] |_{Z\setminus Z_{0}}.
\end{equation}
(We remark that \eqref{eqn:equality:coh:2} follows directly from the fact that $\kappa_{Z,k}$ can be constructed in such a way that 
$[\kappa_{Z,k}]|_{Z\setminus Z_{0}} = [\kappa_{Z}]|_{Z\setminus Z_{0}}$.
Indeed, $\varphi^{*}\Theta=[\theta_{\infty}\wedge\kappa_{Z}^{q}]$. See Section~\ref{sect:9.2}.)
\par
Since $\varphi\colon Z\to Y$ is a proper modification,
there exists $\theta\in H^{0}(Y,\Omega_{Y}^{n+1-q}(F^{*}\xi))$ with $\varphi^{*}\theta = \theta_{\infty}$.
Since $\kappa_{Z}=\varphi^{*}\kappa_{Y}$, we get 
$$
\Theta|_{Y\setminus Y_{0}}
=
(\varphi^{-1})^{*}\varphi^{*}\Theta|_{Z\setminus Z_{0}}
=
(\varphi^{-1})^{*}[\theta_{\infty}\wedge\kappa_{Z}^{q}]|_{Z\setminus Z_{0}}
=
[\theta\wedge\kappa_{Y}^{q}]|_{Y\setminus Y_{0}}.
$$
Since $\varpi = f \circ \varphi$ and hence $(\varphi^{-1})^{*}\varpi^{*} = f^{*}$, we get $(f^{*}dt)\wedge\theta=0$ by the relation 
$(\varpi^{*}dt)\wedge\theta_{\infty}=0$. 
Since $(\varphi^{-1})^{*}\widetilde{F}^{*} = F^{*}$, $\theta|_{Y\setminus Y_{0}} = (\varphi^{-1})^{*}\theta_{\infty}$ and 
$Y\setminus Y_{0}$ is a dense open subset of $Y$, we get $D^{F^{*}\xi}\theta = 0$ by $D^{\widetilde{F}^{*}\xi} \theta_{\infty} =0$. 
This completes the proof.
\end{pf}

\begin{remark}
In Proposition~\ref{Proposition:Takegoshi}, the equality $\Theta=[\theta\wedge\kappa_{Y}^{q}]$ holds. See Section~\ref{sect:9.2}.
\end{remark}

\subsection
{Asymptotic expansion of the $L^{2}$-metric}
\label{sect:4.4}
\par
By the $G$-equivariance of the vector bundle $R^{q}f_{*}K_{Y}(F^{*}\xi)$, 
its direct summand $R^{q}f_{*}K_{Y}(F^{*}\xi)_{W}$ is a $G$-equivariant holomorphic vector bundle on $T$. 
Let $h_{W}^{q} \in{\mathbf Z}_{\geq0}$ be its rank. Shrinking $T$ if necessary, we assume that $R^{q}f_{*}K_{Y}(F^{*}\xi)$ is a free
${\mathcal O}_{T}$-module on $T$. 
Let $\{\Theta_{1},\ldots,\Theta_{h_{W}^{q}}\} \subset \Gamma(T, R^{q}f_{*}K_{Y}(F^{*}\xi)_{W})$ be a basis of 
$R^{q}f_{*}K_{Y}(F^{*}\xi)_{W}$ as a free ${\mathcal O}_{T}$-module. 
Shrinking $T$ if necessary, we regard $\Theta_{\alpha} \in \Gamma(T, R^{q}f_{*}K_{Y}(F^{*}\xi)) = H^{q}(Y, K_{Y}(F^{*}\xi))$ 
via the isomorphism \eqref{eqn:dir:sum:dir:im}.
By Proposition~\ref{Proposition:Takegoshi}, there exist a holomorphic differential form 
$\theta_{\alpha} \in \Gamma(Y,\Omega_{Y}^{n-q+1}(F^{*}\xi))$ 
and a holomorphic relative differential form $\widetilde{\theta}_{\alpha} \in \Gamma(Y\setminus Y_{0},\Omega_{Y/T}^{n-q}(F^{*}\xi))$ 
such that
\begin{equation}
\label{eqn:theta:1}
\Theta_{\alpha}|_{Y\setminus Y_{0}}=[\theta_{\alpha}\wedge\kappa_{Y}^{q}]|_{Y\setminus Y_{0}} \in H^{q}(Y\setminus Y_{0}, K_{Y}(F^{*}\xi)),
\end{equation}
\begin{equation}
\label{eqn:theta:2}
\theta_{\alpha}|_{Y\setminus Y_{0}} = \widetilde{\theta}_{\alpha}\wedge f^{*}dt,
\end{equation}
\begin{equation}
\label{eqn:flatness}
D^{F^{*}\xi, F^{*}h_{\xi}} \theta_{\alpha} =0.
\end{equation}
For $t\in T^{o}$, set $\kappa_{Y_{t}}:=\kappa_{Y}|_{Y_{t}}$. 
By Theorem~\ref{Thm:Takegoshi} (3), $\widetilde{\theta}_{\alpha}|_{Y_{t}}\wedge\kappa_{Y_{t}}^{q}$ is the harmonic representative of 
the class $[\widetilde{\theta}_{\alpha}\wedge\kappa_{Y}^{q}]|_{Y_{t}}$ with respect to the metrics $\kappa_{Y_{t}}$, $h_{F^{*}\xi}|_{Y_{t}}$. 
\par
Let $\{e_{1},\ldots,e_{r}\}$ be a local holomorphic frame of $F^{*}\xi$ and 
let $\{e_{1}^{\lor},\ldots,e_{r}^{\lor}\}$ be its dual frame of $F^{*}\xi^{\lor}$. 
We can express locally $\widetilde{\theta}_{\alpha}=\sum_{i} \widetilde{\theta}_{\alpha,i}\otimes e_{i}$ and 
$h_{F^{*}\xi}=\sum_{i,j}h_{i\bar{j}}e_{i}^{\lor}\otimes\bar{e}_{j}^{\lor}$,
where the $\widetilde{\theta}_{\alpha,i}$ are local holomorphic sections of $\Omega_{Y/T}^{q}$ and 
the $h_{i\bar{j}}$ are local smooth functions. We define a smooth relative $(q,q)$-form by
$$
h_{F^{*}\xi}(\widetilde{\theta}_{\alpha}\wedge\overline{\widetilde{\theta}}_{\beta})
:=
\sum_{i,j}h_{i\bar{j}} \widetilde{\theta}_{\alpha,i}\wedge\overline{\widetilde{\theta}}_{\beta,j}.
$$
For any $t\in T^{o}$, we have
$h_{F^{*}\xi}(\widetilde{\theta}_{\alpha}\wedge\overline{\widetilde{\theta}}_{\beta})|_{Y_{t}} = 
h_{F^{*}\xi}(\widetilde{\theta}_{\alpha}|_{Y_{t}}\wedge\overline{{\widetilde{\theta}}}_{\beta}|_{Y_{t}}) 
\in A^{q,q}(Y_{t})$.
Since $\Theta_{\alpha} \otimes (f^{*}dt)^{-1} \in H^{q}(Y, K_{Y/T}(F^{*}\xi))$ and $K_{Y/T}|_{Y_{t}} \cong K_{Y_{t}}$ in the canonical way, 
its restriction $\Theta_{\alpha} \otimes (f^{*}dt)^{-1}|_{Y_{t}}$ is viewed as an element of $H^{q}(Y_{t}, K_{Y_{t}}(F^{*}\xi))$. 
By \eqref{eqn:theta:1}, \eqref{eqn:theta:2}, under this identification, we have
$$
\Theta_{\alpha} \otimes (f^{*}dt)^{-1}|_{Y_{t}} = [ \widetilde{\theta}_{\alpha}|_{Y_{t}} \wedge \kappa_{Y_{t}}^{q} ].
$$
\par
We set
\begin{equation}
\label{eqn:L2:met:f}
\widetilde{H}_{\alpha\bar{\beta}}(t) 
:= \left( \Theta_{\alpha}\otimes(f^{*}dt)^{-1}|_{Y_{t}},\Theta_{\beta}\otimes(f^{*}dt)^{-1}|_{Y_{t}} \right)_{L^{2}}
\qquad(t\in T^{o}).
\end{equation}
Then $\widetilde{H}(t)=( \widetilde{H}_{\alpha\bar{\beta}}(t) )$ is a positive-definite $(h_{W}^{q}, h_{W}^{q})$-Hermitian matrix.

\begin{lemma}
\label{lemma:fiber:integral}
For any $t\in T^{o}$, the following identity holds:
$$
\widetilde{H}_{\alpha\bar{\beta}}(t)
=
q! \int_{Y_{t}} i^{(n-q)^{2}} h_{F^{*}\xi}(\widetilde{\theta}_{\alpha}\wedge\overline{\widetilde{\theta}}_{\beta})|_{Y_{t}}
\wedge\kappa_{Y_{t}}^{q}.
$$
\end{lemma}

\begin{pf}
Let $\star_{t}$ be the Hodge star operator with respect to $\kappa_{Y_{t}}$.
Since $\widetilde{\theta}_{\alpha}|_{Y_{t}} \wedge \kappa_{Y_{t}}^{q}$ is the harmonic representative of the cohomology class 
$[\widetilde{\theta}_{\alpha} \wedge \kappa_{Y}^{q}]|_{Y_{t}}$ by Theorem~\ref{Thm:Takegoshi} (3), we have
$H_{\alpha\bar{\beta}}(t) 
= \int_{Y_{t}} \langle \widetilde{\theta}_{\alpha}|_{Y_{t}} \wedge \kappa_{Y_{t}}^{q} , \widetilde{\theta}_{\beta}|_{Y_{t}} 
\wedge\kappa_{Y_{t}}^{q} \rangle_{t} dv_{t}$ with $dv_{t}=\kappa_{Y_{t}}^{n}/n!$.
Since $\widetilde{\theta}_{\alpha}|_{Y_{t}} \wedge \kappa_{Y_{t}}^{q} = (q!/(-i)^{(n-q)^{2}}) (\star_{t} \widetilde{\theta}_{\alpha}|_{Y_{t}})$ 
(cf. \cite[Cor.\,2.2 (1)]{MourouganeTakayama08}), we get 
$$
\langle \widetilde{\theta}_{\alpha}|_{Y_{t}}\wedge\kappa_{Y_{t}}^{q}, \widetilde{\theta}_{\beta}|_{Y_{t}}\wedge\kappa_{Y_{t}}^{q} \rangle_{t}
=
(q!)^{2} \langle \star_{t}( \widetilde{\theta}_{\alpha}|_{Y_{t}}), \star_{t}(\widetilde{\theta}_{\beta}|_{Y_{t}}) \rangle_{t}
=
(q!)^{2} \langle \widetilde{\theta}_{\alpha}|_{Y_{t}}, \widetilde{\theta}_{\beta}|_{Y_{t}} \rangle_{t}.
$$
Since
$\langle \widetilde{\theta}_{\alpha}|_{Y_{t}}, \widetilde{\theta}_{\beta}|_{Y_{t}} \rangle_{t} dv_{t} 
= h_{F^{*}\xi}( \widetilde{\theta}_{\alpha}|_{Y_{t}} \wedge \overline{\star}_{t} ( \widetilde{\theta}_{\beta}|_{Y_{t}}) )
= C_{n,q} h_{F^{*}\xi}( \widetilde{\theta}_{\alpha}|_{Y_{t}} \wedge \overline{\widetilde{\theta}_{\beta}}|_{Y_{t}} ) 
\wedge \kappa_{Y_{t}}^{q}$ with $C_{n,q} = i^{(n-q)^{2}}/q!$, we get the result.
\end{pf}

\begin{remark}
As is mentioned after Theorem~\ref{Thm:Takegoshi}, $H_{\alpha\bar{\beta}}(t)$ is independent of the choice of K\"ahler form
within the K\"ahler class $[\kappa_{Y_{t}}]$. Indeed, for a $(q, q-1)$-form $\chi$ on $Y_{t}$, we have 
$h_{F^{*}\xi}( \widetilde{\theta}_{\alpha}|_{Y_{t}} \wedge \overline{\widetilde{\theta}_{\beta}}|_{Y_{t}} ) \wedge \bar{\partial}\chi
= \bar{\partial}_{t}\{ h_{F^{*}\xi}( \widetilde{\theta}_{\alpha}|_{Y_{t}} \wedge \overline{\widetilde{\theta}_{\beta}}|_{Y_{t}} ) \wedge \chi \}$
by the equality $D^{F^{*}\xi}_{t} (\widetilde{\theta}_{\alpha}|_{Y_{t}}) =0$ derived from the equality $D^{F^{*}\xi, F^{*}h_{\xi}}\theta_{\alpha} =0$
in Proposition~\ref{Proposition:Takegoshi}, where $D^{F^{*}\xi}_{t}$ is the Chern connection of $(F^{*}\xi, F^{*}h_{\xi})|_{Y_{t}}$.
\end{remark}

\begin{proposition}
\label{prop:asym:ex}
There exist smooth functions $A_{\alpha\bar{\beta};m}(t) \in C^{\infty}(T)$ $(1\leq\alpha,\beta\leq h_{W}^{q}$, $0\leq m\leq n)$ such that
\begin{equation}
\label{eqn:asym:exp:H}
\widetilde{H}_{\alpha\bar{\beta}}(t) = \sum_{0 \leq m \leq n} A_{\alpha\bar{\beta};m}(t)\,(\log|t|^{-2})^{m}.
\end{equation}
In particular, there exist smooth functions $a_{m}(t) \in C^{\infty}(T)$ $( 0\leq m\leq n h_{W}^{q} )$ with
\begin{equation}
\label{eqn:asym:exp:det:H}
\det \widetilde{H}(t) = \sum_{0 \leq m \leq n h_{W}^{q}} a_{m}(t) (\log|t|^{-2})^{m}.
\end{equation}
\end{proposition}

\begin{pf}
On $Y\setminus Y_{0}$, by \eqref{eqn:theta:2}, we have
\begin{equation}
\label{eqn:rel:diff:form}
f^{*}(i\,dt\wedge d\bar{t}) \wedge \{i^{(n-q)^{2}}h_{F^{*}{\xi}}(\widetilde{\theta}_{\alpha}\wedge\overline{\widetilde{\theta}}_{\beta}) 
\wedge \kappa_{Y}^{q}\}
=
i^{(n-q+1)^{2}}h_{F^{*}{\xi}}(\theta_{\alpha}\wedge\overline{\theta}_{\beta}) \wedge \kappa_{Y}^{q}.
\end{equation}
\par
Let $\{{\mathcal V}_{\lambda}\}_{\lambda\in\Lambda}$ be a finite open covering of $Y$ such that there is a system of coordinates 
$(z_{0},\ldots,z_{n})$ on ${\mathcal V}_{\lambda}$ with $f|_{{\mathcal V}_{\lambda}}(z)=z_{0}\cdots z_{k}$. 
Here $k$ depends on $\lambda\in\Lambda$.
We may suppose that ${\mathcal V}_{\lambda}$ is biholomorphic to a polydisc of dimension $n+1$ for all $\lambda \in \Lambda$. 
Let us write $u=(u_{0}, \ldots, u_{k})$ and $v=(v_{k+1}, \ldots, v_{n})$.
Then 
$$
Y_{t}\cap V_{\lambda} = W_{t} \times \varDelta^{n-k},
\qquad
W_{t} := \{ u \in \varDelta^{k+1} ;\, u_{0} \cdots u_{k} =t \} .
$$
Let $\{\varrho_{\lambda}\}_{\lambda\in\Lambda}$ be a partition of unity subject to $\{{\mathcal V}_{\lambda}\}_{\lambda\in\Lambda}$.
We set $du := du_{0}\wedge \cdots \wedge du_{k}$ and $dv := dv_{k+1}\wedge \cdots \wedge dv_{n}$. 
Since $\theta_{\alpha}$ and $\theta_{\beta}$ are holomorphic $n-q+1$-forms on $Y$ and $\kappa_{Y}\in A^{1,1}(Y)$,
there exists $B_{\alpha\bar{\beta}}^{(\lambda)}(u,v)\in C^{\infty}_{0}({\mathcal V}_{\lambda})$ such that
\begin{equation}
\label{eqn:rel:diff:form:2}
i^{(n-q+1)^{2}}
\varrho_{\lambda}(z) h_{F^{*}{\xi}}( \theta_{\alpha} \wedge \overline{\theta}_{\beta} ) \wedge \kappa_{Y}^{q} |_{{\mathcal V}_{\lambda}}
=
B_{\alpha\bar{\beta}}^{(\lambda)}(u,v)\,du\wedge \overline{du} \wedge dv \wedge \overline{dv}.
\end{equation}
Define $C_{\alpha\bar{\beta}}^{(\lambda)}(u) \in C^{\infty}_{0}( \varDelta^{k+1})$ by 
$C_{\alpha\bar{\beta}}^{(\lambda)}(u) := \int_{\varDelta^{n-k}} B_{\alpha\bar{\beta}}^{(\lambda)}(u,v)\,dv\wedge \overline{dv}$.
By Takayama~\cite[Proof of Prop.\,6.1]{Takayama22} applied to the fiber integral 
$\int_{W_{t}} C_{\alpha\bar{\beta}}^{(\lambda)}(u) du\wedge \overline{du}$, there exist smooth functions 
$A_{\alpha\bar{\beta};m}^{(\lambda)}(t) \in C^{\infty}(T)$ $(0 \leq m \leq k)$ such that
\begin{equation}
\label{eqn:asym:exp:fib:int1}
\begin{aligned}
\int_{Y_{t}\cap {\mathcal V}_{\lambda}} B_{\alpha\bar{\beta}}^{(\lambda)}(u, v) du \wedge d\overline{u} \wedge dv \wedge d\overline{v} 
&= 
\int_{W_{t}} C_{\alpha\bar{\beta}}^{(\lambda)}(u) du\wedge \overline{du}
\\
&=
\{ \sum_{0 \leq m \leq k} A_{\alpha\bar{\beta};m}^{(\lambda)}(t) ( \log |t| )^{m} \}\, i \, dt \wedge d \overline{t}.
\end{aligned}
\end{equation}
By \eqref{eqn:rel:diff:form}, \eqref{eqn:rel:diff:form:2}, \eqref{eqn:asym:exp:fib:int1}, as functions on $T^{o}$, we have
\begin{equation}
\label{eqn:asym:fib:int}
\int_{Y_{t}\cap{\mathcal V}_{\lambda}} i^{(n-q)^{2}}\,\varrho_{\lambda}(z)
h_{F^{*}{\xi}}(\widetilde{\theta}_{\alpha}\wedge\overline{\widetilde{\theta}}_{\beta})|_{Y_{t}}\wedge\kappa_{Y_{t}}^{q}
=
\sum_{0 \leq m \leq k} A_{\alpha\bar{\beta};m}^{(\lambda)}(t) (\log|t|^{2})^{m}.
\end{equation}
Since
$\widetilde{H}_{\alpha\bar{\beta}} =\sum_{\lambda\in\Lambda}i^{(n-q)^{2}}
f_{*}(\varrho_{\lambda}\,h_{F^{*}{\xi}}(\widetilde{\theta}_{\alpha}\wedge\overline{\widetilde{\theta}}_{\beta})\wedge\kappa_{Y}^{q})$,
the result follows from \eqref{eqn:asym:fib:int}. This completes the proof.
\end{pf}

\begin{remark}
For $q=0$, Proposition~\ref{prop:asym:ex} is a special case of \cite[Prop.\,6.1]{Takayama22}. 
In fact, for $q=0$, this proposition holds even if $(\xi, h_{\xi})$ is not Nakano semi-positive.
\end{remark}

\subsection
{Non-degeneracy of the $L^{2}$-metric}
\label{sect:4.5}
\par
To prove the non-vanishing of $a_{m}(0)$ for some $0 \leq m \leq n h_{W}^{q}$ in Proposition~\ref{prop:asym:ex}, 
we need the following:

\begin{lemma}
\label{lemma:comparison:L2}
Let $M$ be a compact complex manifold and let $\omega$, $\omega'$ be K\"ahler metrics on $M$. 
Let $(E,h)$ be a holomorphic Hermitian vector bundle on $M$.
Let $h_{L^{2},\omega,h}$ (resp. $h_{L^{2},\omega',h}$) be the $L^{2}$-metric on $H^{q}(M, K_{M}(E))$ with respect to 
the metrics $\omega$, $h$ (resp. $\omega'$, $h$). If $\omega \leq \omega'$ as Hermitian forms on $T^{(1,0)}M$, then
$$
h_{L^{2},\omega',h} \leq h_{L^{2},\omega,h}.
$$
\end{lemma}

\begin{pf}
Let $\theta \in H^{q}(M, K_{M}(E))$. Let ${\mathcal H}_{\omega,h}(\theta)$ (resp. ${\mathcal H}_{\omega',h}(\theta)$) 
be the harmonic representative of $\theta$ with respect to the metrics $\omega$, $h$ (resp. $\omega'$, $h$). 
Let $dv_{\omega}$ and $dv_{\omega'}$ be the volume forms of $\omega$ and $\omega'$, respectively. Then we have
$$
\int_{M} | {\mathcal H}_{\omega,h}(\theta) |_{\omega,h}^{2} dv_{\omega}
\geq 
\int_{M} | {\mathcal H}_{\omega,h}(\theta) |_{\omega',h}^{2} dv_{\omega'}
\geq
\int_{M} | {\mathcal H}_{\omega',h}(\theta) |_{\omega',h}^{2} dv_{\omega'},
$$
where the first inequality follows from \cite[Chap.VIII, Lemma~6.3]{Demailly12} and the second inequality follows from the fact 
that the harmonic form has the minimal $L^{2}$-norm among all $\bar{\partial}$-closed differential forms in the given cohomology class.
The result follows from this inequality.
\end{pf}

Let $\kappa'_{Y}$ be a K\"ahler form on $Y$. Multiplying a positive constant and shrinking $T$ if necessary, we may assume
$\kappa_{Y} \leq \kappa'_{Y}$. Let $( \cdot, \cdot)_{L^{2}, Y_{t}}$ (resp. $( \cdot, \cdot)'_{L^{2}, Y_{t}}$) be the $L^{2}$-metric on
$H^{q}(Y_{t}, K_{Y_{t}}(F^{*}\xi))$ with respect to $\kappa_{Y}|_{Y_{t}}$, $F^{*}h_{\xi}|_{Y_{t}}$ 
(resp. $\kappa'_{Y}|_{Y_{t}}$, $F^{*}h_{\xi}|_{Y_{t}}$).
The corresponding $L^{2}$-norms are denoted by $\| \cdot \|_{L^{2},Y_{t}}$ and $\| \cdot \|'_{L^{2},Y_{t}}$, respectively.
Since $\kappa'_{Y}$ is a non-degenerate K\"ahler form on $Y$ and $Y_{0}$ is a  reduced normal crossing divisor on $Y$, 
the non-degeneracy result of Mourougane-Takayama \cite[Lemmas~4.7, 4.8]{MourouganeTakayama09} is applicable to
the $L^{2}$-metric on $R^{q}f_{*}K_{Y/T}(F^{*}\xi)$ induced from $(\cdot,\cdot)'_{L^{2},t}$.

\begin{proposition}
\label{prop:nonvanishing} 
There exists a constant $C>0$ such that for all $t \in T^{o}$,
\begin{equation}
\label{eqn:nonvanishing:0}
\widetilde{H}(t) \geq C\, I_{h_{W}^{q}}
\end{equation}
as positive definite Hermitian matrices.
In particular, $\det \widetilde{H}(t)\geq C^{h_{W}^{q}}$ on $T^{o}$.
\end{proposition}

\begin{pf}
$\{ \Theta_{1}\otimes(f^{*}dt)^{-1},\ldots, \Theta_{h_{W}^{q}} \otimes(f^{*}dt)^{-1}\}$ is a basis of the free ${\mathcal O}_{T}$-module 
$R^{q}f_{*}K_{Y/T}(F^{*}\xi)_{W}$.
Since $Y_{0}$ is a reduced normal crossing divisor on $Y$ and $\kappa'_{Y}$ is a non-degenerate K\"ahler from on $Y$, 
by \cite[Lemmas 4.7 and 4.8]{MourouganeTakayama09} applied to $f \colon Y \to T$ and $F^{*}\xi$ with respect to the metrics
$\kappa'_{Y}$, $F^{*}h_{\xi}$, there is a constant $C_{0}>0$ such that 
for all $u=(u_{\alpha})\in{\mathbf C}^{h_{W}^{q}}$ and $t\in T^{o}$,
\begin{equation}
\label{eqn:nonvanishing}
\| \sum_{\alpha} u_{\alpha} \Theta_{\alpha}\otimes (f^{*}dt)^{-1} |_{Y_{t}} \|'_{L^{2},Y_{t}} \geq C_{0}\,|u|.
\end{equation}
(Note that since $Y_{0}$ is a reduced normal crossing divisor, we can take $\tau^{\circ}=\nu=\mu={\rm id}_{X}$ and $\tau={\rm id}_{Y}$
in \cite[Sect.\,3.1]{MourouganeTakayama09} and hence $X''=X$ and $Y'=Y$ in \cite[Sect.\,4.3]{MourouganeTakayama09}.)
Since $\kappa_{Y}|_{Y_{t}} \leq \kappa'_{Y}|_{Y_{t}}$, 
we deduce from Lemma~\ref{lemma:comparison:L2} that for all $\sigma \in H^{q}(Y_{t}, K_{Y_{t}}(F^{*}\xi))$,
\begin{equation}
\label{eqn:comparison:L2}
\| \sigma \|_{L^{2}, Y_{t}} \geq \| \sigma \|'_{L^{2}, Y_{t}}.
\end{equation}
By \eqref{eqn:nonvanishing}, \eqref{eqn:comparison:L2}, setting $\sigma = \sum_{\alpha}u_{\alpha}\Theta_{\alpha}\otimes(f^{*}dt)^{-1}$,
we get \eqref{eqn:nonvanishing:0} with $C = C_{0}^{2}$.
\end{pf}

\begin{remark}
\label{remark:Nk:pos}
Suppose that $(\xi, h_{\xi})|_{X}$ is {\em Nakano positive}. By the Nakano vanishing theorem \cite[Cor.\,7.5]{Demailly12}, 
$R^{q}\pi_{*}K_{X/S}(\xi)=0$ for $q>0$. 
Let $h'_{\xi}$ be a Hermitian metric on $\xi$ such that $(\xi, h'_{\xi})|_{X}$ is {\em not} necessarily Nakano semi-positive. 
Even in this setting, Propositions~\ref{Proposition:Takegoshi}, ~\ref{prop:asym:ex}, ~\ref{prop:nonvanishing}
for $h'_{\xi}$ remain valid. 
For Proposition~\ref{Proposition:Takegoshi}, this is obvious. Since the proof of Proposition~\ref{prop:asym:ex} works 
if Proposition~\ref{Proposition:Takegoshi} holds, Proposition~\ref{prop:asym:ex} remains valid for $h'_{\xi}$. 
Since there exists a constant $C>0$ such that $C^{-1}h_{\xi} \leq h'_{\xi} \leq C h_{\xi}$ on $X$, 
Proposition~\ref{prop:nonvanishing} for $h'_{\xi}$ follows from that for $h_{\xi}$. 
\end{remark}

\subsection
{Singularity of the $L^{2}$-metric on $R^{q}\pi_{*}K_{X/S}(\xi)$}
\label{sect:4.6}
\par
In this subsection, we describe the structure of the singularity of the $L^{2}$-metric on $R^{q}\pi_{*}K_{X/S}(\xi)$.
By Lemma~\ref{lemma:local:freeness} or \cite[Th.\,6.5 (i)]{Takegoshi95}, $R^{q}\pi_{*}K_{X/S}(\xi)$ is a locally free sheaf on $S$.

\begin{proposition}[\cite{MourouganeTakayama09}]
\label{Proposition:MourouganeTakayama}
There is a $G$-equivariant injective homomorphism of sheaves with the following properties:
$$
\varphi \colon R^{q}f_{*} K_{Y/T}(F^{*}\xi) \hookrightarrow \mu^{*} R^{q}\pi_{*} K_{X/S}(\xi).
$$
\par\noindent
{\rm (1) }
$\mu^{*}R^{q}\pi_{*} K_{X/S}(\xi)/\varphi(R^{q}f_{*} K_{Y/T}(F^{*}\xi))$ is a torsion sheaf on $T$ supported at $0$.
\par\noindent
{\rm (2) }
$\varphi^{*}\mu^{*}h_{R^{q}\pi_{*} K_{X/S}(\xi)} = h_{R^{q}f_{*} K_{Y/T}(F^{*}\xi)}$.
Here $h_{R^{q}\pi_{*} K_{X/S}(\xi)}$ (resp. $h_{R^{q}f_{*} K_{Y/T}(F^{*}\xi)}$) is the $L^{2}$-metric on $R^{q}\pi_{*} K_{X/S}(\xi)$ 
(resp. $R^{q}f_{*} K_{Y/T}(F^{*}\xi)$) with respect to $\kappa_{X}$, $h_{\xi}$ (resp. $\kappa_{Y}$, $F^{*}h_{\xi}$).
\end{proposition}

\begin{pf}
Since $R^{q}\pi_{*} K_{X/S}(\xi) \cong R^{q}\pi'_{*} K_{X'/S}(\beta^{*}\xi)$ in the canonical way by \cite[Th.\,6.5 (i) ($\alpha$)]{Takegoshi95}
and $\beta_{*}K_{X'}(\beta^{*}\xi) = K_{X}(\xi)$,
except for the $G$-equivariance of $\varphi$, the assertions were proved in \cite[Lemmas 3.3 and 4.2]{MourouganeTakayama09}.
Since $\varphi|_{T^{o}}$ is $G$-equivariant and $\varphi$ is defined on $T$, $\varphi$ is $G$-equivariant on $T$.
\end{pf}

For $q\geq0$ and $W \in \widehat{G}$, we set
\begin{equation}
\label{eqn:length}
\delta_{W}^{q}
:=
\frac{1}{\deg \mu} \dim_{\mathbf C}\left( \frac{\mu^{*}R^{q}\pi_{*} K_{X/S}(\xi)_{W}}{R^{q}f_{*} K_{Y/T}(F^{*}\xi)_{W}} \right)_{0}
\in{\mathbf Q}_{\geq0}.
\end{equation}
\par
We are now in a position to state the main theorem of this section, which extends the structure theorem of the singularity of 
the $L^{2}$-metrics on the direct image sheaves of the relative canonical bundles \cite[Prop.\,2.10]{EFM21} 
to the case of those twisted by Nakano semi-positive vector bundles. Write $M_{r}({\mathbf C})$ for the complex $(r,r)$-matrices.
Let ${\mathfrak m}^{\infty}_{0}(T)$ be the smooth functions on $T$ vanishing at $t=0$.

\begin{theorem}
\label{thm:str:sing:L2}
By choosing suitable bases $\{ {\mathbf e}_{1},\ldots,{\mathbf e}_{h_{W}^{q}} \}$ of $R^{q}\pi_{*}K_{X/S}(\xi)_{W}$ and 
$\{ \widetilde{\mathbf e}_{1}, \ldots, \widetilde{\mathbf e}_{h_{W}^{q}}\}$ of $R^{q}f_{*}K_{Y/T}(F^{*}\xi)_{W}$,
there exist integers $e^{q}_{1},\ldots, e^{q}_{h_{W}^{q}} \geq0$ with the following properties:
\par{\rm (1) }
The $(h_{W}^{q}, h_{W}^{q})$-Hermitian matrices $H(s) := ( H_{\alpha\beta}(s) )$, 
$H_{\alpha\beta}(s) := ( {\mathbf e}_{\alpha} , {\mathbf e}_{\beta} )_{L^{2},X_{s}}$, and
$\widetilde{H}(t) = ( \widetilde{H}_{\alpha\beta}(t) )$, 
$\widetilde{H}_{\alpha\beta}(t) := ( \widetilde{\mathbf e}_{\alpha} , \widetilde{\mathbf e}_{\beta} )_{L^{2},Y_{t}}$, 
are related via the following identity
\begin{equation}
\label{eqn:decomp:H}
H(\mu(t)) = D(t)\cdot \widetilde{H}(t)\cdot\overline{D(t)},
\qquad
D(t) = {\rm diag}( t^{-e^{q}_{1}},\ldots,t^{-e^{q}_{h_{W}^{q}}} ).
\end{equation}
Moreover, $\widetilde{H}(t)$ admits an asymptotic expansion of Barlet-Takayama type
\begin{equation}
\label{eqn:asym:exp:H:2}
\widetilde{H}(t) \equiv \sum_{0\leq m \leq n} (\log |t|^{-2})^{m} A_{m} 
\mod \bigoplus_{0\leq k \leq n} (\log |t|^{-2})^{k} {\mathfrak m}^{\infty}_{0}(T) \otimes M_{h_{W}^{q}}( {\mathbf C} )
\end{equation}
with constant $(h_{W}^{q}, h_{W}^{q})$-Hermitian matrices $A_{m}$ $(1\leq m \leq n)$. 
In particular, there exist constants $a_{m}^{q} \in {\mathbf R}$ $( 1 \leq m \leq n h_{W}^{q})$ such that 
\begin{equation}
\label{eqn:asym:exp:det:H:2}
\det \widetilde{H}(t) = \sum_{0 \leq m \leq n h_{W}^{q}} a_{m}^{q} (\log |t|^{-2})^{m} 
\mod \bigoplus_{0\leq k \leq n h_{W}^{q}} (\log |t|^{-2})^{k} {\mathfrak m}^{\infty}_{0}(T).
\end{equation}
\par{\rm (2) }
There exists a constant $C>0$ such that $\widetilde{H}(t) \geq C\,I_{h_{W}^{q}}$ for all $t\in T^{o}$ as positive definite Hermitian matrices.
Especially, $a_{m}^{q} \not=0$ for some $1 \leq m \leq n h_{W}^{q}$.
\par{\rm (3) }
There exist real-valued smooth functions $\psi_{j,k}^{q}(s)$ on $S$ such that
\begin{equation}
\label{eqn:L2:asym:eqv}
\| {\mathbf e}_{1}\wedge \cdots\wedge {\mathbf e}_{h_{W}^{q}} \|_{L^{2}}^{2} =
|s|^{-2\delta_{W}^{q}} \sum_{0\leq m \leq n h_{W}^{q}} \{ a_{m}^{q} 
+ \sum_{1 \leq j \leq d} |s|^{\frac{2j}{d}} \psi_{j,m}^{q}(s) \} (\log |s|^{-2})^{m}. 
\end{equation}
Let
\begin{equation}
\label{eqn:rho:q}
\varrho_{W}^{q} := \max\{ 0 \leq m \leq n h_{W}^{q};\, a_{m}^{q} \not=0 \}.
\end{equation}
Then the following equality holds as $s\to 0$:
\begin{equation}
\label{eqn:asymptotoc:L2:X}
\log\| {\mathbf e}_{1}\wedge \cdots\wedge {\mathbf e}_{h_{W}^{q}} \|_{L^{2}}^{2} 
= -\delta_{W}^{q}\,\log |s|^{2} + \varrho_{W}^{q}\log\log(|s|^{-2}) + c_{W}^{q} + O\left(1/\log |s|^{-1} \right).
\end{equation}
\par{\rm (4) }
The rational numbers $e^{q}_{\alpha}/\deg \mu$ ($1 \leq \alpha \leq h_{W}^{q}$) are independent of the choice of semi-stable reduction.
\end{theorem}

\begin{pf}
By choosing bases $\{ {\mathbf e}_{\alpha} \}$, $\{ \widetilde{\mathbf e}_{\alpha} \}$ suitably, 
it follows from Proposition~\ref{Proposition:MourouganeTakayama} (1) that there exist integers 
$e_{\alpha} = e^{q}_{\alpha} \in {\mathbf Z}_{\geq0}$ ($1\leq\alpha\leq h_{W}^{q}$) satisfying
\begin{equation}
\label{eqn:rel:bases}
\mu^{*}{\mathbf e}_{\alpha} = t^{-e^{q}_{\alpha}}\varphi( \widetilde{\mathbf e}_{\alpha} )
\qquad
(t \in T).
\end{equation}
By \eqref{eqn:rel:bases} and Proposition~\ref{Proposition:MourouganeTakayama} (2), we get
$H_{\alpha\bar{\beta}}(\mu(t)) = t^{-e_{\alpha}}\bar{t}^{-e_{\beta}} \widetilde{H}_{\alpha\bar{\beta}}(t)$.
This proves \eqref{eqn:decomp:H}. The expansions \eqref{eqn:asym:exp:H:2}, \eqref{eqn:asym:exp:det:H:2} follow from
\eqref{eqn:asym:exp:H}, \eqref{eqn:asym:exp:det:H}, respectively. This proves (1).
The inequality $\widetilde{H}(t) \geq C I_{h_{W}^{q}}$ follows from Proposition~\ref{prop:nonvanishing}. 
By this inequality, $a_{m}^{q}\not=0$ for some $0 \leq m \leq n h_{W}^{q}$.
This proves (2). Since
\begin{equation}
\label{eqn:elm:exp:length}
\frac{1}{\deg \mu} \sum_{\alpha} e^{q}_{\alpha} = \delta_{W}^{q}
\end{equation}
by \eqref{eqn:length}, \eqref{eqn:rel:bases}, we deduce from \eqref{eqn:decomp:H}, \eqref{eqn:elm:exp:length} that
\begin{equation}
\label{eqn:comp:det:L2}
\det H(\mu(t)) = |t|^{-2\deg \mu\cdot \delta_{W}^{q}} \det \widetilde{H}(t).
\end{equation}
Since $|t|=|s|^{1/\deg \mu}$ and $\| {\mathbf e}_{1}\wedge \cdots\wedge {\mathbf e}_{h_{W}^{q}} \|_{L^{2}}^{2} = \det H(s)$, 
\eqref{eqn:L2:asym:eqv} follows from \eqref{eqn:asym:exp:det:H}, \eqref{eqn:asym:exp:det:H:2}, \eqref{eqn:comp:det:L2} 
and Lemma~\ref{lemma:Taylor:ram:cov} below. 
We deduce \eqref{eqn:asymptotoc:L2:X} from \eqref{eqn:L2:asym:eqv}. This proves (3).
Since any two semi-stable reductions of $\pi \colon X \to S$ are dominated by the third one and since the integers $e^{q}_{\alpha}$ 
($1\leq \alpha \leq h_{W}^{q}$) depend only on the semi-stable reduction, we get (4). This completes the proof.
\end{pf}

\begin{remark}
\label{remark:inaccuracy}
The proof of Theorem~\ref{thm:str:sing:L2} relies on Propositions~\ref{Proposition:Takegoshi}, \ref{prop:asym:ex} 
and \ref{prop:nonvanishing}.
In the previous version \cite{Yoshikawa10a}, the proof of \cite[Lemma~6.3]{Yoshikawa10a}
(corresponding to Proposition~\ref{prop:asym:ex} here) relied on a theorem of Barlet \cite[Th.\,4 bis]{Barlet82}. 
Since that theorem can no longer be relied upon in its original form, as shown by Takayama \cite[Th.\,7.1.3, Remark\,7.1.4]{Takayama21},
we now provide an alternative proof based on \cite[Prop.\,6.1]{Takayama22} (the calculations in \cite[p.140 Cas 1]{Barlet82}).
Furthermore, the proof of \cite[Lemma~6.4]{Yoshikawa10a} (now Proposition~\ref{prop:nonvanishing}) previously utilized 
a non-degeneracy result from Mourougane-Takayama \cite[Lemma 4.8]{MourouganeTakayama09} for the family $\varpi \colon Z \to T$. 
However, as the K\"ahler form $\kappa_{Z}$ on $Z$ may degenerate on certain non-$\widetilde{\rho}$-exceptional components of $Z_{0}$, 
\cite[Lemma\,4.1]{MourouganeTakayama09} does not necessarily hold for $\kappa_{Z}$. 
Consequently, \cite[Lemma 4.8]{MourouganeTakayama09} cannot be directly applied to the pair $\varpi \colon Z \to T$ and $\kappa_{Z}$. 
This difficulty is circumvented by employing Lemma~\ref{lemma:comparison:L2}.
\end{remark}

\begin{lemma}
\label{lemma:Taylor:ram:cov}
Let $\varphi(t) \in C^{\infty}(T)$. Set $s = t^{d}$. If $\varphi(\zeta t) = \varphi(t)$, $\zeta=e^{2\pi i/d}$,
then there exist $\psi_{j}(s) \in C^{\infty}(S)$ $(0\leq j \leq d-1)$ such that $\varphi(t) = \sum_{j=0}^{d-1} |s|^{\frac{2j}{d}} \psi_{j}(s)$.
\end{lemma}

\begin{pf}
Since the map $C^{\infty}(T) \to {\mathbf C}[[ t, \bar{t} ]]$ assigning $F(t) \in C^{\infty}(T)$ to its Taylor series at the origin is surjective
by Borel,
there exists $\psi_{j}(s) \in C^{\infty}(S)$ $(0\leq j \leq d-1)$ such that $\chi(t) := \varphi(t) - \sum_{j=0}^{d-1} |s|^{\frac{2j}{d}} \psi_{j}(s)$
has the vanishing Taylor series at $t=0$. Since $\chi(\zeta t ) = \chi(t)$, $\chi(s^{\frac{1}{d}})$ is well-defined and lies in 
$C^{0}(S) \cap C^{\infty}(S^{o})$. Write $s=x+iy=r(\cos\theta + \sqrt{-1} \sin\theta)$. 
Since $\partial_{x}^{i}\partial_{y}^{j} \chi(s^{1/d})$ for $s\not=0$ is a linear combination of the various derivatives of $\chi(t)$ 
up to the order $i+j$ and some polynomials in the variables $r^{-1}$, $r^{1/d}$, $\cos(\theta/d)$, $\sin(\theta/d)$, 
the condition that $\chi(t)$ has the vanishing Taylor series at $t=0$ implies that
$\partial_{x}^{i}\partial_{y}^{j} \chi(s^{1/d}) \to 0$ as $s \to 0$ for all $i,j \geq0$. Hence $\chi(s^{1/d}) \in C^{\infty}(S)$.
Replacing $\psi_{0}(s)$ with $\psi_{0}(s)+\chi(s^{1/d})$, we get the result.
\end{pf}

After \cite[Def.\,2.5]{EFM21} and Theorem~\ref{thm:str:sing:L2}, we make the following:

\begin{definition}
\label{def:elem:exponent}
The rational numbers $\epsilon^{q}_{\alpha} := e^{q}_{\alpha}/\deg \mu$ ($1\leq \alpha \leq h_{W}^{q}$) are called 
the elementary exponents of the vector bundle $R^{q}\pi_{*}K_{X/S}(\xi)_{W}$. 
The elementary exponents of $R^{q}\pi_{*}K_{X/S}(\xi)$ are defined to be those for the trivial group $G=\{1\}$.
\end{definition}

By \eqref{eqn:elm:exp:length}, we have
\begin{equation}
\label{eqn:tr:elem:exp}
\delta_{W}^{q} = \sum_{\alpha} \epsilon^{q}_{\alpha}.
\end{equation}

We define
\begin{equation}
\label{eqn:elem:exp}
\delta_{\pi, K_{X/S}(\xi)}(g) 
:= \frac{1}{\deg \mu} \sum_{q\geq0}(-1)^{q}{\rm Tr}[ g|_{(\mu^{*}R^{q}\pi_{*} K_{X/S}(\xi) / R^{q}f_{*} K_{Y/T}(F^{*}\xi))_{0}}].
\end{equation}
By \eqref{eqn:rel:bases}, \eqref{eqn:tr:elem:exp}, we have
\begin{equation}
\label{eqn:elem:exp}
\delta_{\pi, K_{X/S}(\xi)}(g) 
= 
\sum_{W \in \widehat{G}} \sum_{q\geq0} (-1)^{q}\frac{\chi_{W}(g)}{\dim W} \delta_{W}^{q}
=
\sum_{W \in \widehat{G}} \sum_{q\geq0} \sum_{\alpha} (-1)^{q}\frac{\chi_{W}(g)}{\dim W} \epsilon_{\alpha}^{q}.
\end{equation}
We often write $\delta_{g}$ for $\delta_{\pi, K_{X/S}(\xi)}(g)$ when there is no possibility of confusion.
Similarly, we write $\delta_{\pi, K_{X/S}(\xi)}$ for $\delta_{\pi, K_{X/S}(\xi)}(1)$. In Corollary~\ref{cor:sing:L2} below, 
we will give an expression of $\delta_{\pi, K_{X/S}(\xi)}$ in terms of various characteristic classes associated to 
the semi-stable reduction $f \colon Y \to T$.

\begin{remark}
\label{remark:str:sing:L2}
If $(\xi, h_{\xi})|_{X}$ is Nakano positive, then Theorem~\ref{thm:str:sing:L2} holds for any Hermitian metric $h'_{\xi}$ on $\xi$, 
whose curvature is not necessarily semi-positive. Indeed, by Remark~\ref{remark:Nk:pos}, 
the proof of Theorem~\ref{thm:str:sing:L2} works in this case. 
\end{remark}

After Theorem~\ref{thm:str:sing:L2} and Remark~\ref{remark:str:sing:L2}, we ask the following:

\begin{question}
In the situation of Theorem~\ref{thm:str:sing:L2}, let $h'_{\xi}$ be any Hermitian metric on $\xi$, 
whose curvature form is not necessarily semi-positive definite.
Then does Theorem~\ref{thm:str:sing:L2} hold for $(\xi, h'_{\xi})|_{X}$? Remark~\ref{remark:str:sing:L2} implies that this is the case 
if $(\xi, h_{\xi})|_{X}$ is Nakano positive.
\end{question}

If we drop the semi-positivity of the curvature of the Hermitian metric on $\xi$,
we still have the following weak form of Theorem~\ref{thm:str:sing:L2} (3).

\begin{theorem}
\label{thm:asym:L2:weak}
Let $h'_{\xi}$ be another Hermitian metric on $\xi$ such that $(\xi, h'_{\xi})|_{X}$ is not necessarily Nakano semi-positive. 
Then, as $s \to 0$,
$$
\log\|\sigma_{W}^{q}(s)\|_{L^{2}, h'_{\xi}}^{2} = -\delta_{W}^{q}\,\log |s|^{2} + \varrho_{W}^{q} \log\log(|s|^{-2}) + O\left(1 \right),
$$
where $\delta_{W}^{q}$ and $\varrho_{W}^{q}$ are the same constants as in Theorem~\ref{thm:str:sing:L2} (3).
\end{theorem}

\begin{pf}
We write $h^{q}_{L^{2}}$ and $h^{\prime q}_{L^{2}}$ for the $L^{2}$-metrics on $R^{q}\pi_{*}K_{X/S}(\xi)$ with respect to 
$h_{\xi}$ and $h'_{\xi}$, respectively. For a cohomology class $\theta \in H^{q}(X_{s}, K_{X_{s}}(\xi_{s}))$, 
we write ${\mathcal H}(\theta)$ and ${\mathcal H}'(\theta)$ for the harmonic representatives of $\theta$ with respect to 
$h_{\xi}$ and $h'_{\xi}$, respectively. Let $dv_{s}$ be the volume form of $h_{X_{s}}$. 
Let $C>0$ be a constant such that $C^{-1} h'_{\xi} \leq h_{\xi} \leq C h'_{\xi}$ as Hermitian metrics on $\xi$. Then 
$$
\int_{X_{s}} \langle {\mathcal H}(\theta), {\mathcal H}(\theta) \rangle_{h_{\xi}} dv_{s}
\leq
\int_{X_{s}} \langle {\mathcal H}'(\theta), {\mathcal H}'(\theta) \rangle_{h_{\xi}} dv_{s}
\leq
C \int_{X_{s}} \langle {\mathcal H}'(\theta), {\mathcal H}'(\theta) \rangle_{h'_{\xi}} dv_{s},
$$
where the first inequality follows from the fact that the harmonic form minimizes the $L^{2}$-norm in the given Dolbeault cohomology class
and the second inequality follows from the inequality $h_{\xi} \leq C h'_{\xi}$. From this inequality, we obtain the inequality of 
Hermitian metrics $h^{q}_{L^{2}} \leq C h^{\prime q}_{L^{2}}$ on $R^{q}\pi_{*}K_{X/S}(\xi)$. 
Similarly, we get $C^{-1} h^{\prime q}_{L^{2}} \leq h^{q}_{L^{2}}$.
By these inequalities of the $L^{2}$-metrics, we obtain $\log( \| \cdot \|_{L^{2}, h'_{\xi}} / \| \cdot \|_{L^{2}, h_{\xi}} ) = O(1)$ as $s \to 0$.
This, together with Theorem~\ref{thm:str:sing:L2} (3), yields the result.
\end{pf}

\section
{The singularity of analytic torsion near the boundary}
\label{sect:5}
In this section, we prove the existence of the asymptotic expansion of the logarithm of (equivariant) analytic torsion 
viewed as a function on the punctured disc, whose leading term is a logarithmic singularity and whose subdominant term is 
a $\log\log$-type singularity. We determine the coefficient of the logarithmic singularity 
when $X_{0}$ has only canonical singularities or $\xi$ is the trivial Hermitian line bundle.

\subsection
{The existence of asymptotic expansion}
\label{sect:5.1}
\par
Recall that $\alpha_{\pi, K_{X/S}(\xi)}(g)$ and $\varrho^{q}_{W}$ were defined by \eqref{eqn:log:coeff:tw} and \eqref{eqn:rho:q}, respectively. 
Define
\begin{equation}
\label{eqn:def:kappa}
\kappa_{\pi, K_{X/S}(\xi)}(g) = \kappa_{g} := \alpha_{\pi, K_{X/S}(\xi)}(g) + \delta_{\pi, K_{X/S}(\xi)}(g),
\end{equation}
\begin{equation}
\label{eqn:def:rho}
\varrho_{\pi, K_{X/S}(\xi)}(g) = \varrho_{g} := \sum_{q\geq0,\,W\in\widehat{G}} (-1)^{q} \chi_{W}(g) \varrho_{W}^{q}/\dim W.
\end{equation}
When $g=1$, we write $\kappa_{K_{X/S}(\xi)}$ and $\varrho_{K_{X/S}(\xi)}$ in place of $\kappa_{\pi, K_{X/S}(\xi)}(1)$ and 
$\varrho_{\pi, K_{X/S}(\xi)}(1)$, respectively.

\begin{theorem}
\label{thm:main:1}
If $(\xi,h_{\xi})$ is Nakano semi-positive on $X=\pi^{-1}(S)$, then the equivariant analytic torsion $\tau_{G}(X_{s},K_{X_{s}}(\xi_{s}))(g)$ 
is expressed as follows on $S^{o}$:
\begin{equation}
\label{eqn:reg:tors}
\begin{aligned}
\,&
\log \tau_{G}(X_{s},K_{X_{s}}(\xi_{s}))(g) = 
\kappa_{g} \log |s|^{2} + c_{g} + \sum_{0\leq m \leq n}\sum_{i\in I} |s|^{2r_{i}} (\log |s|^{-2})^{m} \phi_{i,m,g}(s)
\\
&\quad - \sum_{q\geq0, W \in \widehat{G}} \frac{(-1)^{q}\chi_{W}(g)}{\dim W} 
\log \left( \sum_{0 \leq k \leq n h_{W}^{q}} \{ a_{k,W}^{q} + \sum_{1 \leq j \leq d} |s|^{\frac{2j}{d}} \psi_{j,k,W}^{q}(s) \} (\log |s|^{-2})^{k} \right),
\end{aligned}
\end{equation}
where $c_{g} \in {\mathbf C}$, $\{ r_{i} \}_{i \in I} \subset {\mathbf Q}\cap (0,1]$ is a finite set of positive rational numbers,
$(a_{0,W}^{q}, \ldots, a_{n h_{W}^{q},W}^{q}) \not= (0,\ldots,0)$ is the non-zero real constant vector 
appearing in \eqref{eqn:L2:asym:eqv}, the $\phi_{i,m,W}(s)$ are complex-valued smooth functions on $S$, 
and the $\psi_{j,k,W}^{q}(s)$ are real-valued smooth functions on $S$. 
In particular, there exists a constant $\gamma_{g}\in{\mathbf C}$ such that as $s\to0$
\begin{equation}
\label{eqn:asymp:eq:tors}
\log \tau_{G}(X_{s},K_{X_{s}}(\xi_{s}))(g) = \kappa_{g} \log |s|^{2} - \varrho_{g}\,\log\log( |s|^{-2}) + \gamma_{g} + O\left(1/\log |s|^{-1} \right).
\end{equation}
\end{theorem}

\begin{pf}
Let $\sigma:=(\sigma_{W})_{W\in\widehat{G}}$ be an admissible section of 
$\lambda_{G}(K_{X/S}(\xi)) = \lambda_{G}(K_{{\mathcal X}/C}(\xi))|_{S}$. 
By Theorem~\ref{Thm:Sing:Q:adj}, there exist a finite set of rational numbers $\{ r_{i} \}_{i \in I} \subset {\mathbf Q}\cap (0,1]$,
complex-valued smooth functions $\phi_{i,m,g}(s)$ on $S$ and a constant $c_{g} \in {\mathbf C}$ such that
\begin{equation}
\label{eqn:asymp:adm:sect}
\begin{aligned}
\log\| \sigma(s) \|_{\lambda_{G}(K_{{\mathcal X}/C}(\xi)),Q}^{2}(g) 
&= 
\alpha_{\pi, K_{X/S}(\xi)}(g)\,\log | s |^{2} + c_{g} 
\\
&\quad + \sum_{0\leq m \leq n}\sum_{i\in I} |s|^{2r_{i}} (\log |s|^{-2})^{m} \phi_{i,m,g}(s)
\end{aligned}
\end{equation}
as functions on $S^{o}$. By Theorem~\ref{thm:str:sing:L2} (2), we can express
\begin{equation}
\label{eqn:asym:exp:L2}
\begin{aligned}
\,&
\log \left( \frac{\|\sigma(s)\|_{\lambda_{G}(K_{{\mathcal X}/C}(\xi)),Q}^{2}(g)}{\tau_{G}(X_{s},K_{X_{s}}(\xi_{s}))(g)} \right)
=
\sum_{q\geq0,\,W\in\widehat{G}} \frac{(-1)^{q}\chi_{W}(g)}{\dim W} \log \|\sigma^{q}_{W}(s)\|_{L^{2}}^{2}
\\
&=
\sum_{q\geq0,\,W\in\widehat{G}} \frac{(-1)^{q}\chi_{W}(g)}{\dim W} 
\log \left( \sum_{0\leq k \leq n h_{W}^{q}} \{ a_{k,W}^{q} + \sum_{1\leq j \leq d} |s|^{\frac{2j}{d}} \psi_{j,k,W}^{q}(s) \} (\log |s|^{-2})^{k} \right).
\end{aligned}
\end{equation}
The expansion \eqref{eqn:reg:tors} follows from \eqref{eqn:asymp:adm:sect} and \eqref{eqn:asym:exp:L2}.
By \eqref{eqn:asymptotoc:L2:X}, \eqref{eqn:elem:exp}, \eqref{eqn:def:rho}, we find
\begin{equation}
\label{eqn:asym:eq:L2}
\begin{aligned}
\,&
\sum_{q\geq0,\,W\in\widehat{G}} \frac{(-1)^{q}\chi_{W}(g)}{\dim W} \log \|\sigma^{q}_{W}(s)\|_{L^{2}}^{2}
\\
&=
\sum_{q,\,W}(-1)^{q}\frac{\chi_{W}(g)}{\dim W} \left\{ -\delta_{W}^{q}\,\log |s|^{2} 
+ \varrho_{W}^{q} \log\log (|s|^{-2}) + c_{W}^{q} + O\left(1/\log |s|^{-1}\right)
\right\}
\\
&=
-\delta_{\pi, K_{X/S}(\xi)}(g) \log |s|^{2} + \varrho_{g} \log\log (|s|^{-2}) + c'_{g} + O\left( 1/\log |s|^{-1} \right),
\end{aligned}
\end{equation}
where we set $c'_{g} := \sum_{q\geq0,\,W\in\widehat{G}} (-1)^{q} \chi_{W}(g) c_{W}^{q}/\dim W$. 
Then \eqref{eqn:asymp:eq:tors} is derived from \eqref{eqn:asymp:adm:sect}, \eqref{eqn:asym:eq:L2} with $\gamma_{g} := c_{g} - c'_{g}$.
This completes the proof.
\end{pf}

\begin{corollary}
\label{cor:HessianTorsion}
As $s\to0$, the following equality holds:
$$
\partial_{s\bar{s}}\log\tau(X_{s}, K_{X_{s}}(\xi_{s}))(g) 
= \frac{\varrho_{\pi,K_{X/S}(\xi)}(g)}{|s|^{2}(\log|s|)^{2}} + O\left(\frac{1}{|s|^{2}(\log|s|)^{3}}\right).
$$
\end{corollary}

\begin{pf}
The result follows from \eqref{eqn:reg:tors} and Lemma~\ref{lemma:Mumf:good} below.
\end{pf}

\begin{corollary}
\label{cor:Nk:neg}
Let $(E,h_{E})$ be a holomorphic Hermitian vector bundle on ${\mathcal X}$. 
If $(E,h_{E})$ is semi-negative in the dual Nakano sense on $X$, then as $s\to0$,
$$
(-1)^{n+1}\log\tau_{G}(X_{s},E_{s})(g) = \kappa_{g} \log|s|^{2} - \varrho_{g} \log\log (|s|^{-2}) + c_{g} + O\left( 1/\log |s|^{-1} \right).
$$
Here $\kappa_{g},\varrho_{g},c_{g}$ are the constants given by \eqref{eqn:def:kappa}, \eqref{eqn:def:rho} with $\xi=E^{\lor}$.
\end{corollary}

\begin{pf}
Let $\square_{E_{s}}^{p,q}$ be the Laplacian acting on $A^{p,q}(X_{s},E_{s})$ and let 
$\bar{\star}_{E_{s}} \colon A^{p,q}(X_{s}, E_{s}) \to A^{n-p,n-q}(X_{s}, E_{s}^{\lor})$ be the Hodge star operator (cf. \cite[p.166]{Wells08}). 
Since 
$\bar{\star}_{E_{s}}\square_{E_{s}}^{0,q}\bar{\star}_{E_{s}}^{-1} = \square_{E_{s}^{\lor}}^{n,n-q} = \square_{K_{X_{s}}(E_{s}^{\lor})}^{0,n-q}$
by \cite[Chap.\,V Prop.\,2.4 (b)]{Wells08}, we have
\begin{equation}
\label{eqn:serre:duality}
\log\tau_{G}(X_{s},E_{s})(g) = (-1)^{n+1}\log\tau_{G}(X_{s}, K_{X_{s}}(E_{s}^{\lor}))(g).
\end{equation}
Since $(E^{\lor},h_{E^{\lor}})$ is Nakano semi-positive on $X$, the result follows from Theorem~\ref{thm:main:1} and \eqref{eqn:serre:duality}.
\end{pf}

\begin{corollary}
\label{cor:ss:degn}
If $(\xi,h_{\xi})$ is Nakano semi-positive on $X$ and $X_{0}$ is a reduced normal crossing divisor of $X$, then as $s\to0$
$$
\log\tau_{G}(X_{s}, K_{X_{s}}(\xi_{s}))(g) 
= \alpha_{\pi, K_{X/S}(\xi)}(g) \log|s|^{2} - \varrho_{g} \log\log (|s|^{-2}) + c_{g} + O\left(1/\log |s|^{-1} \right).
$$
\end{corollary}

\begin{pf}
We have $\delta_{\pi, K_{X/S}(\xi)}(g) = 0$ in Theorem~\ref{thm:main:1}, since $\pi\colon X\to S$ is a semi-stable degeneration. 
The result follows from Theorem~\ref{thm:main:1}.
\end{pf}

\begin{remark}
Suppose that $(\xi, h_{\xi})|_{X}$ is Nakano positive. Let $h'_{\xi}$ be an arbitrary Hermitian metric on $\xi$. 
In particular, $(\xi, h'_{\xi})$ is not necessarily Nakano semi-positive on $X$. 
Even in this case, the asymptotic expansion~\eqref{eqn:asymp:eq:tors} still holds. In fact, since Theorem~\ref{thm:str:sing:L2} holds 
in this case by Remark~\ref{remark:str:sing:L2}, the proof of Theorem~\ref{thm:main:1} works.
\end{remark}

\begin{theorem}
\label{thm:asym:tors:weak}
Let $h'_{\xi}$ be another Hermitian metric on $\xi$ such that $(\xi, h'_{\xi})$ is not necessarily Nakano semi-positive. 
Let $\tau'_{G}(X_{s},K_{X_{s}}(\xi_{s}))(g)$ be the equivariant analytic torsion of $\xi_{s}$ with respect to the metrics 
$\kappa_{X_{s}}$, $h'_{\xi}$. Then, as $s \to 0$,
$$
\log \tau'_{G}(X_{s},K_{X_{s}}(\xi_{s}))(g) = \kappa_{g} \log |s|^{2} - \varrho_{g}\,\log\log( |s|^{-2}) + O\left( 1 \right).
$$
\end{theorem}

\begin{pf}
By Theorem~\ref{thm:asym:L2:weak}, we have the following weak form of \eqref{eqn:asym:eq:L2}:
$$
\sum_{q\geq0,\,W\in\widehat{G}} \frac{(-1)^{q}\chi_{W}(g)}{\dim W} \log \|\sigma^{q}_{W}(s)\|_{L^{2}}^{2}
=
-\delta_{\pi, K_{X/S}(\xi)}(g) \log |s|^{2} + \varrho_{g} \log\log (|s|^{-2}) + O\left( 1 \right).
$$
Substituting this into the first equality of \eqref{eqn:asym:exp:L2} and using \eqref{eqn:asymp:adm:sect}, we get the result.
\end{pf}

\subsection
{Hodge theory and the singularity of analytic torsion}
\label{sect:5.2}

Throughout this subsection, we assume that $\xi|_{X} = {\mathcal O}_{X}$ is the trivial line bundle and that $h_{\xi}|_{X}$ is the trivial metric.
Then the singularity of the $L^{2}$-metric on $\lambda(K_{X/S})$ was determined by 
Eriksson-Freixas i Montplet-Mourougane \cite[Prop.\,2.10]{EFM21}. 
When $X_{0}$ has at most isolated singularities, the singularity of the analytic torsion $\tau(X_{s}, {\mathcal O}_{X_{s}})$ was determined by 
Eriksson-Freixas i Montplet \cite[Cor.\,3.2]{ErikssonFreixas26}. 
In this subsection, we recall the formula for $\kappa_{\pi, K_{X/S}}$.
In Section~\ref{sect:7}, these results will be used to express $\kappa_{\pi, K_{X/S}(\xi)}$
when $\xi$ is topologically trivial near ${\rm Sing}\,X_{0}$.

\subsubsection
{Hodge theory and the singularity of the $L^{2}$-metrics}
\label{sect:5.2.1}
To describe the structure of the singularity of the $L^{2}$-metric on $R^{q}\pi_{*}K_{X/S}$ (\cite[Prop.\,2.10]{EFM21}), 
we introduce some notation.
Since $S \cong \varDelta$ is endowed with the coordinate $s$, fixing a reference point of $S^{o}$ and the corresponding fiber $X_{\infty}$, 
we have the limiting mixed Hodge structure on $H^{r}(X_{\infty})$ with limiting Hodge filtration $F^{\bullet}_{\infty}H^{r}(X_{\infty})$ 
(\cite{Schmid73}, \cite{Steenbrink77}).
Let $M \in {\rm GL}(H(X_{\infty}))$ be the monodromy for the family $\pi^{o} \colon X^{o} \to S^{o}$. 
Let $M_{s}$ be the semi-simple part of $M$. Then $M_{s}$ preserves $F^{k}_{\infty}H^{r}(X_{\infty})$ (\cite[Th.\,2.13]{Steenbrink77}). 
Since the $G$-action on $R^{r}\pi_{*}{\mathbf C} \otimes {\mathcal O}_{S^{o}}$ is parallel with respect to the Gauss-Manin connection, 
$G$ acts on the multi-valued flat sections of $R^{r}\pi_{*}{\mathbf C} \otimes {\mathcal O}_{S^{o}}$ over $S^{o}$
and the $G$-action on the flat sections of $R^{r}\pi_{*}{\mathbf C} \otimes {\mathcal O}_{S^{o}}$ commutes with $M$, hence $M_{s}$. 
Then $G$ acts on the (upper) canonical extension ${\mathcal H}^{r} \cong H^{r}(X_{\infty})\otimes{\mathcal O}_{S}$ of 
$R^{r}\pi_{*}{\mathbf C} \otimes {\mathcal O}_{S^{o}}$. 
Since the $G$-action on $R^{r}\pi_{*}{\mathbf C} \otimes {\mathcal O}_{S^{o}}$ preserves the Hodge bundles 
$F^{k}(R^{r}\pi_{*}{\mathbf C} \otimes {\mathcal O}_{S^{o}})$ $(k\geq 0)$, $G$ acts on the $F^{k}{\mathcal H}^{r}$ $(k\geq 0)$. 
In particular, $G$ acts on $F^{k}_{\infty}H^{r}(X_{\infty})$. 
Since the $M_{s}$-action commutes with the $G$-action, $M_{s}$ acts on 
$$
F^{k}_{\infty}H^{r}(X_{\infty})_{W} := {\rm Hom}_{G}(W, F^{k}_{\infty}H^{r}(X_{\infty})) \otimes W
$$ 
for all $W \in \widehat{G}$, where $M_{s}$ acts trivially on $W$, and the $M_{s}$-action on $F^{k}_{\infty}H^{r}(X_{\infty})$ 
is identified with the direct sum of these $M_{s}$-actions on $F^{k}_{\infty}H^{r}(X_{\infty})_{W}$. 
Let $\log z$ be the logarithm whose imaginary part lies in the interval $[0, 2\pi)$ and 
let $\log M_{s}$ be the corresponding logarithm of $M_{s}$. Then $\log M_{s}$ acts on $F^{k}H^{r}(X_{\infty})_{W}$.

\begin{proposition}
\label{prop:EFM21}
The elementary exponents of $(R^{q}\pi_{*}K_{X/S})_{W}$ are the eigenvalues of $\frac{1}{2\pi i}\log M_{s}$ 
acting on $F^{n}_{\infty}H^{n+q}(X_{\infty})_{W}$. In particular,
$$
\delta_{W}^{q} = \frac{1}{2\pi i} {\rm Tr} \left[ \left. \log M_{s} \right|_{F^{n}_{\infty}H^{n+q}(X_{\infty})_{W}} \right],
$$
$$
\delta_{\pi, K_{X/S}}(g) =
\frac{1}{2\pi i} \sum_{q\geq 0,\,W \in \widehat{G}} (-1)^{q} \frac{\chi_{W}(g)}{\dim W}
{\rm Tr} \left[ \left. \log M_{s} \right|_{F^{n}_{\infty}H^{n+q}(X_{\infty})_{W}} \right].
$$
\end{proposition}

\begin{pf}
Recall that $\beta \colon (X', X'_{0}) \to (X, X_{0})$ is a log-resolution of $X_{0}$ and $\pi' = \pi \circ \beta$.
Hence $X'_{0}$ is a normal crossing divisor. 
Since $R^{q}\beta_{*}K_{X'}=0$ ($q>0$) (\cite[Th.\,6.5 (i)]{Takegoshi95}) and $\beta_{*}K_{X'} = K_{X}$, 
we have the canonical isomorphism $R^{q}\pi'_{*}K_{X'/S} \cong R^{q}\pi_{*}K_{X/S}$.
Consequently, the elementary exponents of $R^{q}\pi_{*}K_{X/S}$ and $R^{q}\pi'_{*}K_{X'/S}$ coincide.
To prove the assertion, we can suppose that $X_{0}$ is a normal crossing divisor on $X$.
Since $F^{n}_{\infty}H^{n+q}(X_{\infty}) \cong \bigoplus_{W\in\widehat{G}} F^{n}_{\infty}H^{n+q}(X_{\infty})_{W}$ as $G$-modules
and $X_{0}$ is a normal crossing divisor, the result follows from \cite[Prop.\,2.10]{EFM21}.
\end{pf}

\begin{corollary}
\label{cor:sing:tors:triv}
If $(\xi, h_{\xi})|_{X}$ is the trivial line bundle endowed with the trivial metric, the coefficient of the logarithmic singularity of
$\log \tau_{G}(X_{s}, K_{X_{s}}(\xi_{s}))(g)$ is given by
$$
\kappa_{g} = 
\alpha_{\pi, K_{X/S}}(g)
+
\frac{1}{2\pi i} \sum_{q\geq 0,\,W \in \widehat{G}} (-1)^{q} \frac{\chi_{W}(g)}{\dim W}
{\rm Tr} \left[ \left. \log M_{s} \right|_{F^{n}_{\infty}H^{n+q}(X_{\infty})_{W}} \right].
$$
\end{corollary}

\begin{pf}
The result follows from \eqref{eqn:def:kappa} and Proposition~\ref{prop:EFM21}
\end{pf}

\subsubsection
{The case of isolated hypersurface singularities}
\label{sect:5.2.2}
\par
In Subsection~\ref{sect:5.2.2}, we assume that ${\rm Sing}\,X_{0}$ consists of isolated points. 
We recall two invariants of isolated hypersurface singularities.
In what follows, we identify an isolated hypersurface singularity germ $(X_{0},0) \subset ({\mathbf C}^{n+1},0)$ 
with its defining equation $f(z)=0$, where $f(z) \in {\mathcal O}_{{\mathbf C}^{n+1},0}$ has an isolated critical point at the origin. 
We denote by $\mu(f)$ the Milnor number of $f$.
\par
Very recently, another invariant of $f$ called {\em spectral genus} is introduced by Eriksson-Freixas i Montplet \cite{ErikssonFreixas26}. 
Let ${\rm Mil}_{f}$ be the Milnor fiber of $f$. By Steenbrink \cite[Sect.\,3]{Steenbrink77}, 
$H^{n}({\rm Mil}_{f})$ carries a mixed Hodge structure. Let $F^{\bullet}H^{n}({\rm Mil}_{f})$ be the Hodge filtration on $H^{n}({\rm Mil}_{f})$. 
As before, let $M_{s} \in {\rm GL}( H^{n}({\rm Mil}_{f}) )$ be the semi-simple part of the monodromy acting on $H^{n}({\rm Mil}_{f})$. 
Recall that $\log z$ is the branch of the logarithm, whose imaginary part lies in $[0, 2\pi)$. 
Let $\log M_{s}$ be the corresponding logarithm of $M_{s}$. Since $M_{s}$ preserves $F^{\bullet}H^{n}({\rm Mil}_{f})$,
$\log M_{s}$ acts on ${\rm Gr}_{F}^{n} H^{n}({\rm Mil}_{f})$. By \cite{ErikssonFreixas26}, 
the spectral genus of $f$ is defined as the rational number 
$$
\widetilde{p}_{g}(f) := (2\pi i)^{-1} {\rm Tr} \left[ \left. \log M_{s} \right|_{{\rm Gr}_{F}^{n} H^{n}({\rm Mil}_{f})} \right].
$$
We refer the reader to \cite{ErikssonFreixas26} for further details about the spectral genus.
\par
The following formula for the asymptotic of the analytic torsion was obtained by Eriksson-Freixas i Montplet \cite[Cor.\,3.2]{ErikssonFreixas26}.

\begin{proposition}[\cite{ErikssonFreixas26}]
\label{prop:EF24}
Let $\pi_{x} \in {\mathcal O}_{X,x}$ is the germ defined by $\pi$ at $x \in X$. As $s \to 0$,
$$
\log \tau(X_{s}, K_{X_{s}}) 
= - \{ \sum_{x\in {\rm Sing}\,X_{0}} \frac{\mu(\pi_{x})}{(n+2)!} - \widetilde{p}_{g}(\pi_{x}) \} \log |s|^{2} +O\left( \log\log (|s|^{-2})\right).
$$
\end{proposition}

\begin{pf}
The result follows from \cite[Cor.\,3.2]{ErikssonFreixas26} and \eqref{eqn:serre:duality}.
\end{pf}

\subsection
{Canonical singularities and $L^{2}$-metrics}
\label{sect:5.3}
In Section~\ref{sect:5.3}, we prove the continuity of the $L^{2}$-metric on $R^{q}\pi_{*}K_{X/S}(\xi)$ when $X_{0}$ has canonical singularities. 
For this sake, we recall some notions concerning singularities and introduce some notations used throughout Section~\ref{sect:5.3}. 
Let $V$ be a normal variety with canonical divisor $K_{V}$ and let $\phi \colon \widetilde{V} \to V$ be any resolution. 
Then $V$ has canonical singularities if $K_{V}$ is ${\mathbf Q}$-Cartier and $K_{\widetilde{V}} - \phi^{*}K_{V}$ is effective. 
\par
Consider the pair $(X, X_{0})$. Recall that $\beta \colon X' \to X$ is an embedded log resolution of $X_{0}$. 
Let $E = \sum_{i\in I} a_{i} D_{i} \subset X'$ be a $\beta$-exceptional normal crossing divisor with
\begin{equation}
\label{eqn:discrepancy}
K_{X'}+\widetilde{X}_{0} = \beta^{*}( K_{X} + X_{0} ) + E.
\end{equation}
Here $\widetilde{X}_{0}\subset X'$ is the proper transform of $X_{0}$, which is a smooth hypersurfaces of $X'$, 
$D_{i}$ is a prime divisor of $X'$, $a_{i}\in {\mathbf Z}$ ($i\in I$), 
and $X'_{0} := \beta^{-1}(X_{0})$ is a normal crossing divisor of $X'$ with $\widetilde{X}_{0}\cup E\subset X'_{0}$.
The pair $(X, X_{0})$ is called canonical (resp. log canonical) if $a_{i} \geq 0$ (resp. $a_{i}\geq -1$) for all $i \in I$ (\cite[Def.\,3.5]{Kollar97}). 
If $X_{0}$ is reduced and has only canonical singularities, then $(X, X_{0})$ is canonical (\cite{Stevens88}, \cite[Th.\,7.9]{Kollar97}).

\subsubsection
{Fiber integrals and canonical singularities}
\label{sect:5.3.1}
\par
Let $\omega \in A^{n+1,n+1}(X)$. We can express $\omega = \pi^{*}(ds\wedge d\bar{s})\wedge R$, where
$R\in C^{\infty}(X\setminus{\rm Sing}\,X_{0},\Omega_{X/S}^{n}\wedge\overline{\Omega_{X/S}^{n}})$.
Define ${\mathcal R}(s)\in C^{\infty}(S^{o})$ as the fiber integral $\displaystyle {\mathcal R}(s):=\int_{X_{s}}R|_{X_{s}}$. 
Then $\pi_{*}(\omega) = {\mathcal R}(s)\,ds\wedge d\bar{s}$.

\begin{lemma}
\label{lemma:fib:int:can:sing}
Suppose that $X_{0}$ is a reduced divisor with only canonical singularities. Then ${\mathcal R}(s)$ lies in ${\mathcal B}(S)$ 
by setting $\displaystyle {\mathcal R}(0) := \int_{(X_{0})_{\rm reg}}R|_{(X_{0})_{\rm reg}}$.
\end{lemma}

\begin{pf}
{\em (Step 1) }
Set $\pi':=\pi \circ \beta$. Then $X'_{0}=(\pi')^{-1}(0)$ and $X'_{s}:=(\pi')^{-1}(s)\cong X_{s}$ for $s\in S^{o}$.
Let ${\mathcal V} \subset X$ be a relatively compact open subset carrying a nowhere vanishing holomorphic $n+1$-form $\eta$. 
By an argument using partition of unity, it suffices to prove the assertion when $\omega$ has support in ${\mathcal V}$.
So we suppose ${\rm supp}\,\omega \subset {\mathcal V}$. Let $p\in X'_{0} \cap \beta^{-1}({\mathcal V})$. 
Since $X_{0}$ is a reduced divisor of $X$, there is a coordinate neighborhood $(U,(z_{0},\cdots,z_{n}))$ centered at $p$ satisfying
$\beta(U)\subset{\mathcal V}$ and one of the following (a), (b): 
Write $\pi'|_{U}(z)=z_{0}^{e_{0}}\cdots z_{n}^{e_{n}}$, $e_{i}\in{\bf Z}_{\geq0}$.
\begin{itemize}
\item[(a)]
If $p\in\widetilde{X}_{0}$, then $X'_{s}\cap U = \{z\in U;\,z_{0}z_{1}^{e_{1}}\cdots z_{n}^{e_{n}}=s\}$ and $\widetilde{X}_{0}\cap U=\{z_{0}=0\}$.
\item[(b)]
If $p\not\in\widetilde{X}_{0}$, then
$X'_{s}\cap U = \{z\in U;\,z_{0}^{e_{0}}z_{1}^{e_{1}}\cdots z_{n}^{e_{n}}=s\}$, $e_{0}>0$ and $\widetilde{X}_{0}\cap U=\emptyset$.
\end{itemize}
Then, on $U$, we can express
\begin{equation}
\label{eqn:pullback}
(\pi')^{*}\left(\frac{ds}{s}\right) = \beta^{*}\pi^{*}\left(\frac{ds}{s}\right) =
\begin{cases}
\begin{array}{lll}
\frac{dz_{0}}{z_{0}}+\sum_{i\geq1}e_{i}\frac{dz_{i}}{z_{i}}
&{\rm if }&p\in\widetilde{X}_{0},
\\
\sum_{i\geq0}e_{i}\frac{dz_{i}}{z_{i}}
&{\rm if }&p\not\in\widetilde{X}_{0}.
\end{array}
\end{cases}
\end{equation}
\par{\em (Step 2) }
Since $X_{0}$ has only canonical singularities, the pair $(X, X_{0})$ is canonical. 
Namely, $a_{i} \geq 0$ for all $i \in I$ in \eqref{eqn:discrepancy}.
Since $\eta/\pi^{*}s$ is a meromorphic $(n+1)$-form with pole of order $1$ along $X_{0}\cap {\mathcal V}$,  
$\beta^{*}(\eta/\pi^{*}s)$ is a meromorphic $(n+1)$-form on $X' \cap \beta^{-1}({\mathcal V})$ with at most logarithmic pole along 
$\widetilde{X}_{0}\cap\beta^{-1}({\mathcal V})$, 
because $K_{X'}+\widetilde{X}_{0}=\beta^{*}(K_{X}+X_{0})+\sum_{i\in I}a_{i}D_{i}$ with $a_{i}\geq0$. 
Hence, on $U$, we have
\begin{equation}
\label{eqn:log:can:form}
\beta^{*}\left(\frac{\eta}{\pi^{*}s}\right)
=
\begin{cases}
\begin{array}{lll}
a(z)\,\frac{dz_{0}}{z_{0}}\wedge dz_{1}\wedge\cdots\wedge dz_{n} & \text{if} & p\in\widetilde{X}_{0}\cap U,
\\
b(z)\,dz_{0}\wedge\cdots\wedge dz_{n} & \text{if} & p\not\in\widetilde{X}_{0}\cap U,
\end{array}
\end{cases}
\end{equation}
where $a(z),b(z)\in{\mathcal O}(U)$. Let $\widetilde{\eta}\in\Gamma({\mathcal V}\setminus X_{0},\Omega^{n}_{X/S})$ be such that 
$\eta=\widetilde{\eta}\wedge\pi^{*}ds$. Since
\begin{equation}
\label{eqn:Theta:Xi}
\beta^{*}(\eta/\pi^{*}s) = \beta^{*}\widetilde{\eta} \wedge (\pi')^{*}(ds/s),
\end{equation}
$\beta^{*}\widetilde{\eta}$ is expressed as follows on $U\setminus X'_{0}$ by 
\eqref{eqn:pullback}, \eqref{eqn:log:can:form}, \eqref{eqn:Theta:Xi}:
\begin{equation}
\label{eqn:pullback:rel:can:form}
\beta^{*}\widetilde{\eta}
=
\begin{cases}
\begin{array}{lll}
a(z)\,dz_{1}\wedge\cdots\wedge dz_{n}\mod(\pi')^{*}ds & \text{if} & p\in\widetilde{X}_{0}\cap U,
\\
\frac{z_{0}}{e_{0}}b(z)\,dz_{1}\wedge\cdots\wedge dz_{n} \mod(\pi')^{*}ds & \text{if} & p\not\in\widetilde{X}_{0}\cap U.
\end{array}
\end{cases}
\end{equation}
In \eqref{eqn:log:can:form}, \eqref{eqn:pullback:rel:can:form}, set $c(z):=a(z)$ in case (a) and $c(z):=z_{0}b(z)/e_{0}$ in case (b). 
Then $c(z)\in{\mathcal O}(U)$. 
Set $\tau:=dz_{1}\wedge\cdots\wedge dz_{n}$. By \eqref{eqn:pullback:rel:can:form}, for all $\chi\in C_{0}^{\infty}(U)$ and $s\in S^{o}$, 
we have
$\displaystyle \int_{X'_{s}\cap U} \chi\,\beta^{*}(\widetilde{\eta}\wedge\overline{\widetilde{\eta}})|_{X'_{s}\cap U} 
= \int_{X'_{s}\cap U} \chi\,|c|^{2}\,\tau\wedge\overline{\tau}$. It follows from \cite[Th.\,1]{Barlet82} that
\begin{equation}
\label{eqn:lim:fib:int}
\lim_{s\to0}\int_{X'_{s}\cap U} \chi\,\beta^{*}(\widetilde{\eta}\wedge\overline{\widetilde{\eta}})
=
\int_{X'_{0}\cap U} \chi\,|c|^{2}\,\tau\wedge\overline{\tau}
=
\int_{\widetilde{X}_{0}\cap U} \chi\,|c|^{2}\,\tau\wedge\overline{\tau}
=
\int_{\widetilde{X}_{0}\cap U} \chi\,\beta^{*}(\widetilde{\eta}\wedge\overline{\widetilde{\eta}}).
\end{equation}
Here we get the second equality as follows. 
In case (a), $(X'_{0}\setminus\widetilde{X}_{0})\cap U$ is defined locally by the equation $z_{1}^{e_{1}}\cdots z_{n}^{e_{n}}=0$.
Hence one of $dz_{1},\ldots,dz_{n}$ vanishes on $(X'_{0}\setminus\widetilde{X}_{0})\cap U$, 
which implies the second equality of \eqref{eqn:lim:fib:int}. 
In case (b), let $x \in (X'_{0}\setminus\widetilde{X}_{0})\cap U$. 
Then one of $z_{0},\ldots,z_{n}$ vanishes on a neighborhood $W$ of $x$ in $(X'_{0} \setminus \widetilde{X}_{0})\cap U$. 
If $z_{0}|_{W}=0$, then $c|_{W}=(z_{0}b/e_{0})|_{W}=0$. If $z_{j}|_{W}=0$ for some $j>0$, then $dz_{j}|_{W}=0$. 
Since $(X'_{0}\setminus\widetilde{X}_{0})\cap U$ is covered by these $W$, 
we get $\int_{(X'_{0}\setminus\widetilde{X}_{0})\cap U} F\,|c|^{2}\,\tau\wedge\overline{\tau}=0$.
\par{\em (Step 3) }
Since $\omega\in A^{n+1,n+1}(X)$ has compact support in ${\mathcal V}$, 
there exists a smooth function $m\in C_{0}^{\infty}({\mathcal V})$ such that 
$\omega = (-1)^{n}m\,\eta\wedge\overline{\eta} = \pi^{*}(ds\wedge d\bar{s})\wedge m\,\widetilde{\eta}\wedge\overline{\widetilde{\eta}}$.
Then ${\mathcal R}(s) = \pi_{*}(m\,\widetilde{\eta}\wedge\overline{\widetilde{\eta}})$. 
Since $\pi_{*}\omega = (\pi\circ\beta)_{*}(\beta^{*}m \cdot \beta^{*}(\widetilde{\eta}\wedge\overline{\widetilde{\eta}}))\,ds\wedge d\bar{s}$,
it follows from \cite[Th.\,1]{Barlet82} that ${\mathcal R}(s)$ lies in ${\mathcal B}(S)$ 
by setting ${\mathcal R}(0) = \lim_{s\to0} {\mathcal R}(s)$,
because the integrand of this fiber integral is a smooth $(n,n)$-form on $\beta^{-1}({\mathcal V})$ by \eqref{eqn:pullback:rel:can:form}.
By \eqref{eqn:lim:fib:int}, we get
$\displaystyle \lim_{s\to0}{\mathcal R}(s) 
= \int_{(X_{0})_{\rm reg}} m\,\widetilde{\eta}\wedge\overline{\widetilde{\eta}}|_{(X_{0})_{\rm reg}} 
= \int_{(X_{0})_{\rm reg}}R|_{(X_{0})_{\rm reg}}$.
\end{pf}

\begin{remark}
For $q=0$, the assertion of Lemma~\ref{lemma:fib:int:can:sing} holds true even when $\dim T >1$ by Takayama \cite[Prop.\,3.2.3 (2)]{Takayama19}.
\end{remark}

\subsubsection
{Canonical singularities and the continuity of the $L^{2}$-metric}
\label{sect:5.3.2}
\par

\begin{theorem}
\label{cont:L2:can:sing}
Suppose that $X_{0}$ is reduced and has only canonical singularities. 
Then, for all $q\geq0$, the elementary exponents of $R^{q}\pi_{*}K_{X/S}(\xi)$ vanish,
and the $L^{2}$ metric on $R^{q}\pi_{*}K_{X/S}(\xi)|_{S^{o}}$ extends to a continuous Hermitian metric on $R^{q}\pi_{*}K_{X/S}(\xi)|_{S}$ 
of class ${\mathcal B}(S)$. 
\end{theorem}

\begin{pf}
After shrinking $S$ if necessary, we have sections $\Psi_{1},\ldots,\Psi_{h^{q}}\in H^{q}(X, K_{X}(\xi))$ such that
$R^{q}\pi_{*}K_{X}(\xi)|_{S}={\mathcal O}_{S}\Psi_{1}\oplus\cdots\oplus{\mathcal O}_{S}\Psi_{h^{q}}$. 
Set
\begin{equation}
\label{eqn:metric:tensor}
g_{\alpha\bar{\beta}}(s) := ( \Psi_{\alpha}\otimes(\pi^{*}ds)^{-1}|_{X_{s}}, \Psi_{\beta}\otimes(\pi^{*}ds)^{-1}|_{X_{s}} )_{L^{2}},
\qquad
s\in S^{o}.
\end{equation}
Then $g_{\alpha\bar{\beta}}\in C^{\infty}(S^{o})$ and $\det(g_{\alpha\bar{\beta}})>0$ on $S^{o}$.
It suffices to prove that $g_{\alpha\bar{\beta}}\in{\mathcal B}(S)$ and $\det(g_{\alpha\bar{\beta}})>0$ on $S$. 
By Theorem~\ref{Thm:Takegoshi} (1), (2), there exist $\psi_{\alpha}\in \Gamma(X,\Omega_{X}^{n-q+1}(\xi))$ and 
$\widetilde{\psi}_{\alpha}\in \Gamma(X\setminus{\rm Sing}\,X_{0},\Omega_{X/S}^{n-q}(\xi))$ such that 
$\Psi_{\alpha}=[\psi_{\alpha}\wedge\kappa_{X}^{q}]$ and 
$\psi_{\alpha}|_{X\setminus{\rm Sing}\,X_{0}}=\widetilde{\psi}_{\alpha}\wedge\pi^{*}ds$. 
Under the canonical identification $K_{X/S}(\xi)|_{X\setminus X_{0}} = \Omega_{X/S}^{n}(\xi)|_{X\setminus X_{0}}$, 
in $H^{q}(X\setminus X_{0}, K_{X/S}(\xi))$, we have
\begin{equation}
\label{eqn:express:Psi}
\Psi_{\alpha}\otimes(\pi^{*}ds)^{-1} = [\widetilde{\psi}_{\alpha}\wedge\kappa_{X}^{q}],
\qquad
\alpha=1,\ldots, h^{q}.
\end{equation}
Since $(\widetilde{\psi}_{\alpha}\wedge\kappa_{X}^{q})|_{X_{s}}$ is the harmonic representative of the cohomology class 
$\Psi_{\alpha}\otimes(\pi^{*}ds)^{-1}|_{X_{s}}$ $(s\in S^{o})$ by Theorem~\ref{Thm:Takegoshi} (3), 
we deduce from \eqref{eqn:metric:tensor}, \eqref{eqn:express:Psi} that
\begin{equation}
\label{eqn:met:ten:fib:int}
\pi_{*}(i^{(n-q+1)^{2}}h_{\xi}(\psi_{\alpha}\wedge\overline{\psi}_{\beta})\wedge\kappa_{X}^{q}) = i\,g_{\alpha\bar{\beta}}(s)\,ds\wedge d\bar{s}.
\end{equation}
Since $\psi_{\alpha}\in \Gamma(X,\Omega_{X}^{n-q+1}(\xi))$, we see that 
$h_{\xi}(\psi_{\alpha}\wedge\overline{\psi_{\beta}})\wedge\kappa_{X}^{q}$ is a smooth top form on $X$. 
By Lemma~\ref{lemma:fib:int:can:sing} and \eqref{eqn:met:ten:fib:int}, we get $g_{\alpha\bar{\beta}}(s)\in{\mathcal B}(S)$.
\par
Consider the semi-stable reduction \eqref{eqn:ss:red}. By \eqref{eqn:comp:det:L2} and Proposition~\ref{prop:nonvanishing}, 
there exist $\delta^{q}\in{\mathbf Z}_{\geq0}$ and $C>0$ such that
\begin{equation}
\label{eqn:nonvanish:det:1}
\det( g_{\alpha\bar{\beta}})( \mu(t) ) = |t|^{-2\delta^{q}} \det \widetilde{H}(t) \geq C\,|t|^{-2\delta^{q}}.
\end{equation}
Since $\delta^{q}\geq0$, there exists $\epsilon>0$ such that for all $s\in S^{o}$, $\det(g_{\alpha\bar{\beta}}(s)) \geq \epsilon>0$.
Since $g_{\alpha\bar{\beta}}(s)\in{\mathcal B}(S)$, this implies $\det(g_{\alpha\bar{\beta}}(0))\geq\epsilon>0$.
Since the left hand side of \eqref{eqn:nonvanish:det:1} extends to a continuous function on $T$, we get $\delta^{q}=0$. 
By \eqref{eqn:elm:exp:length}, this implies the vanishing of the elementary exponents of $R^{q}\pi_{*}K_{X/S}(\xi)$.
\end{pf}

\begin{corollary}
\label{cor:C0:L2}
Let $\sigma_{W}$ be a nowhere vanishing holomorphic section of $\lambda_{W}(K_{X/S}(\xi))$ (cf. Sect.~\ref{sect:1.3.2}).
Then, under the same assumption as in Theorem~\ref{cont:L2:can:sing}, $\delta_{W}^{q}=\varrho_{W}^{q}=0$ 
for all $q\geq 0$ and $W \in \widehat{G}$. In particular, 
there exist $c\in{\mathbf R}$, $r\in{\mathbf Q}_{>0}$ and $\nu\in{\mathbf Z}_{\geq0}$ such that as $s\to0$
$$
\log\|\sigma_{W}^{q}(s)\|_{L^{2}}^{2} = c + O\left(|s|^{r}(\log|s|)^{\nu}\right).
$$
\end{corollary}

\begin{pf}
Since $R^{q}\pi_{*}K_{X/S}(\xi)_{W} \subset R^{q}\pi_{*}K_{X/S}(\xi)$ is a holomorphic subbundle, 
the result follows from Theorem~\ref{cont:L2:can:sing}. 
\end{pf}

\begin{theorem}
\label{thm:sing:tor:can}
Suppose that $X_{0}$ is reduced and has only canonical singularities. 
Then there exist $\gamma_{g} \in {\mathbf C}$, $r\in{\mathbf Q}_{>0}$ and $\nu\in{\mathbf Z}_{\geq0}$ such that as $s \to 0$,
$$
\log\tau_{G}(X_{s},K_{X_{s}}(\xi_{s}))(g) = \alpha_{\pi,K_{X/S}(\xi)}(g) \log |s|^{2} + \gamma_{g} + O\left( |s|^{r} ( \log (|s|^{-1}) )^{\nu} \right).
$$
\end{theorem}

\begin{pf}
Since $\delta_{\pi,K_{X/S}(\xi)}(g) = \varrho_{\pi, K_{X/S}(\xi)} =0$ by \eqref{eqn:elem:exp}, \eqref{eqn:def:rho} and Corollary~\ref{cor:C0:L2}, 
the result follows from Theorem~\ref{thm:main:1}.
\end{pf}

For an example where $\alpha_{\pi, K_{X/S}(\xi)}(g)$ is explicitly evaluated, we refer the reader to \cite[Th.\,6.3]{Yoshikawa04}. 
In \cite{KawaguchiYoshikawa26, Yoshikawa26}, we will use Theorem~\ref{thm:sing:tor:can} to determine certain 
holomorphic torsion invariants of Calabi-Yau manifolds with group action.

\subsubsection
{Log-canonical singularities and $L^{2}$-metrics}
\label{sect:5.3.3}
\par
We keep the notation in the proof of Theorem~\ref{cont:L2:can:sing}.

\begin{proposition}
\label{prop:L2:log:can}
Suppose that $X_{0}$ is reduced and the pair $(X,X_{0})$ is log-canonical. 
Then the elementary exponents of $R^{q}\pi_{*}K_{X/S}(\xi)$ vanish.
Moreover, there exist constants $\varrho^{q} \in {\mathbf Z}_{\geq 0}$ and $c^{q} \in {\mathbf R}$ such that as $s\to0$,
\begin{equation}
\label{eqn:asymptotoc:L2:X:lc}
\log\det ( g_{\alpha\bar{\beta}}(s) )
= \varrho^{q} \log\log(|s|^{-2}) + c^{q} + O(1/\log |s|^{-1}).
\end{equation}
\end{proposition}

\begin{pf}
Recall that $(U, (z_{0}, \ldots, z_{n}))$ is a coordinate neighborhood as in (Step 1) in the proof of Lemma~\ref{lemma:fib:int:can:sing}, 
that $\eta$ is a nowhere vanishing holomorphic $n+1$-form on $U$, and that $\widetilde{\eta}$ is a relative $n$-form such that 
$\eta = \widetilde{\eta} \wedge \pi^{*}ds$.
Since the pair $(X,X_{0})$ is log-canonical, by \eqref{eqn:log:can:form}, $\beta^{*}(\eta/\pi^{*}s)$ is expressed as follows 
on $U\setminus X'_{0}$:
\begin{equation}
\label{eqn:residue}
\beta^{*}\left(\frac{\eta}{\pi^{*}s}\right) 
= 
c(z) \frac{dz_{0}\wedge\cdots\wedge dz_{n}}{z_{0}\cdots z_{n}}
=
\pi^{\prime *}(ds) \wedge c(z) \frac{dz_{1}\wedge\cdots\wedge dz_{n}}{z_{1}\cdots z_{n}},
\end{equation}
where $c(z) \in {\mathcal O}(U)$. By \eqref{eqn:Theta:Xi}, \eqref{eqn:residue}, for any smooth function $\chi \in C_{0}^{\infty}(U)$ 
with compact support in $U$, there is a constant $C>0$ such that
\begin{equation}
\label{eqn:fib:int:lc}
\begin{aligned}
\left|\int_{X'_{s}\cap U} \chi\,\beta^{*}(\widetilde{\eta} \wedge \overline{\widetilde{\eta}})\right|
&\leq
C\,\left|
\int_{z\in U,\,z_{0}^{e_{1}}\cdots z_{n}^{e_{n}}=s}
\frac{dz_{1}}{z_{1}}\wedge\cdots\wedge\frac{dz_{n}}{z_{n}}
\wedge
\overline{\frac{dz_{1}}{z_{1}}\wedge\cdots\wedge\frac{dz_{n}}{z_{n}}}
\right|
\\
&\leq
C\,(\log |s|^{-1})^{n}.
\end{aligned}
\end{equation}
\par
By \eqref{eqn:met:ten:fib:int} and the smoothness of the top form 
$h_{\xi}(\psi_{\alpha}\wedge\overline{\psi}_{\beta})\wedge\kappa_{X}^{q}$,
we deduce from \eqref{eqn:fib:int:lc} that $g_{\alpha\bar{\beta}}(s) = O\left( (\log|s|^{-1})^{n} \right)$.
This implies that $D(t)$ in Theorem~\ref{thm:str:sing:L2} (1) is the identity matrix. Namely, $e_{1}^{q} = \cdots = e_{h_{W}^{q}}^{q} =0$.
Since $W$ is arbitrary, this implies the vanishing of the elementary exponents of $R^{q}\pi_{*}K_{X/S}(\xi)$.
By \eqref{eqn:elm:exp:length}, we get $\delta_{W}^{q} =0$ in \eqref{eqn:asymptotoc:L2:X}. 
Then \eqref{eqn:asymptotoc:L2:X:lc} follows from \eqref{eqn:asymptotoc:L2:X}. This completes the proof.
\end{pf}

\begin{theorem}
\label{thm:sing:tor:log:can}
Suppose that $X_{0}$ is reduced and the pair $(X, X_{0})$ has only log-canonical singularities. 
Then there exist $\gamma_{g} \in {\mathbf C}$ such that as $s \to 0$,
$$
\log\tau_{G}(X_{s},K_{X_{s}}(\xi_{s}))(g) 
= \alpha_{\pi,K_{X/S}(\xi)}(g) \log |s|^{2} - \varrho_{g}\,\log\log( |s|^{-2}) + \gamma_{g} + O\left(1/\log |s|^{-1} \right).
$$
In particular, if $X_{0}$ is a reduced normal crossing divisor of $X$, this formula holds.
\end{theorem}

\begin{pf}
By Proposition~\ref{prop:L2:log:can}, the elementary exponents of $R\pi_{*}K_{X/S}(\xi)$ vanish. The result follows from Theorem~\ref{thm:main:1} and \eqref{eqn:def:kappa}.
\end{pf}

\subsubsection
{Example}
\label{sect:5.3.4}
\par
Consider the case $g=1$.
Suppose that $\pi(z)$ is locally given by a quadratic polynomial of rank two near the critical locus $\Sigma=\Sigma_{\pi}$, i.e.,
for any $p \in \Sigma$,
$$
\pi(z) = z_{0}z_{1}
$$
in a suitable local coordinates $z=(z_{0}, \ldots, z_{n})$ centered at $p$. 
(Typical examples of such degenerations are provided by the deformation to the normal cone e.g. \cite[Sect.\,8 a)]{Bismut97}.) 
Then the pair $(X, X_{0})$ is log-canonical and Theorem~\ref{thm:sing:tor:log:can} is applicable to this family.
To give an explicit formula for $\kappa_{\pi, K_{X/S}(\xi)} = \alpha_{\pi, K_{X/S}(\xi)}$ in this case, recall that the Bismut $E$-genus
\cite[Def.\,5.1, Prop.\,5.2]{Bismut97} is the additive characteristic class associated to the power series
$$
{\rm E}(x) := \frac{x - {\rm sinh}(x)}{2x(1-{\rm cosh}(x))}
= \frac{{\rm Td}(x){\rm Td}(-x)}{2x} \left( \frac{{\rm Td}^{-1}(x)-1}{x} - \frac{{\rm Td}^{-1}(-x)-1}{-x} \right).
$$
Then ${\rm E}(0) = 1/6$.

\begin{proposition}[\cite{Bismut97}]
\label{prop:kappa:q:rk2}
Under the assumption as above, 
$$
\kappa_{\pi, K_{X/S}(\xi)} = \alpha_{\pi, K_{X/S}(\xi)} 
= -\frac{1}{2} \int_{\Sigma} {\rm Td}(T\Sigma) {\rm E}(N_{\Sigma/X}) {\rm ch}(K_{\Sigma}(\xi)).
$$
\end{proposition}

\begin{pf}
We use the notation in Section~\ref{sect:3.4}. Let $q \colon \widetilde{X} \to X$ be the blowing-up of $\Sigma$ and
let $E = q^{-1}(\Sigma) \cong {\mathbf P}(N_{\Sigma/X})$ be the exceptional divisor. Then the indeterminacy of the Gauss map for $\pi$ 
can be resolved by considering $\widetilde{X}$. 
Let $p \colon {\mathbf P}(N_{\Sigma/X}) \to \Sigma$ be the projection. Since $q|_{E}=p$, setting $g=1$ in \eqref{eqn:log:coeff:tw}, we get
$$
\begin{aligned}
\alpha_{\pi, K_{X/S}(\xi)} 
&= 
\int_{E} \widetilde{\gamma}^{*} 
\left\{ \frac{{\rm Td}^{-1}({\mathcal H}^{\lor})-1}{c_{1}({\mathcal H}^{\lor})} \right\} p^{*}\left\{ {\rm Td}(TX) e^{-c_{1}(X)} {\rm ch}(\xi) \right\}
\\
&=
\int_{\Sigma} p_{*} \widetilde{\gamma}^{*} \left\{ \frac{{\rm Td}^{-1}({\mathcal H}^{\lor})-1}{c_{1}({\mathcal H}^{\lor})} \right\} 
{\rm Td}(N_{\Sigma/X}) e^{-c_{1}(N_{\Sigma/X})} {\rm Td}(T\Sigma) {\rm ch}(K_{\Sigma}(\xi)).
\end{aligned}
$$ 
Let $F:= {\mathcal O}_{{\mathbf P}(N_{\Sigma/X})}(1)$. 
By \cite[(23)]{Yoshikawa07}, $F = \widetilde{\gamma}^{*}{\mathcal H}|_{E}$. Set $f(x) := ({\rm Td}^{-1}(x) - 1)/x$ and
$f_{-}(x) := ( f(x) - f(-x) )/2x$. Then we have
$$
p_{*} \widetilde{\gamma}^{*} \left\{ \frac{{\rm Td}^{-1}({\mathcal H}^{\lor})-1}{c_{1}({\mathcal H}^{\lor})} \right\} 
=
p_{*}( f( -c_{1}(F) ))
=
-p_{*}( f_{*}( c_{1}(F) ))
=
-\frac{1}{2} f_{-}(N_{\Sigma/X}),
$$
where the second equality follows from \cite[p.80 last line]{Yoshikawa07} and the third equality follows from \cite[p.81]{Yoshikawa07}.
Since $c_{1}(N_{\Sigma/X})=0$ by \cite[(2.9)]{Bismut97}, this implies that 
\begin{equation}
\label{eqn:alpha:q:s}
\alpha_{\pi, K_{X/S}(\xi)} 
= 
- \frac{1}{2} \int_{\Sigma} f_{-}(N_{\Sigma/X}) {\rm Td}(N_{\Sigma/X}) {\rm Td}(T\Sigma) {\rm ch}(K_{\Sigma}(\xi)).
\end{equation}
Since $f_{-}(N_{\Sigma/X}) {\rm Td}(N_{\Sigma/X}) = {\rm E}(N_{\Sigma/X})$ by \cite[(29)]{Yoshikawa07},
the desired formula follows from \eqref{eqn:alpha:q:s}.
\end{pf}

\section
{Asymptotic behavior of Bott-Chern forms under degenerations}
\label{sect:6}
\par
In Sections~\ref{sect:6} and \ref{sect:7}, we provide a topological formula for the coefficient $\kappa_{\pi,K_{X/S}(\xi)}$ 
in Theorem~\ref{thm:main:1}. To this end, in this section, 
we study the asymptotic behavior of the fiber integrals of Bott-Chern forms under certain degeneracy assumption. 
In the remainder of this paper, we assume 
$$
G = \{ 1 \}.
$$
\subsection
{Bott-Chern forms}
\label{sect:6.1}
\par
We introduce the notation used throughout Section~\ref{sect:6}.
Let $M$ be a complex manifold. 
Let $E$, $E'$ be holomorphic vector bundles on $M$ of rank $r$. Let 
$$
\varphi \colon E \to E'
$$ 
be a homomorphism of vector bundles. Let $N$ be the degeneracy locus of $\varphi$. 
Namely, $N:= \{ x \in M; \, {\rm rank}(\varphi_{x}) < r \}$. We assume that $N$ is a compact hypersurface of $M$. 
We set $M^{0} := M\setminus N$. 
Let $h_{E}$ and $h_{E'}$ be Hermitian metrics on $E$ and $E'$, respectively. 
We set $\overline{E}:=(E,h_{E})$ and $\overline{E'}:=(E',h_{E'})$.
Let $R_{\overline{E}}$ and $R_{\overline{E'}}$ be the curvature forms of $\overline{E}$ and $\overline{E'}$ 
with respect to the Chern connection, respectively.
Then $\varphi$ induces an isomorphism of holomorphic Hermitian vector bundles on $M^{0}$
$$
\overline{E'} \cong (E, h_{E'}^{\varphi}),
\qquad
h_{E'}^{\varphi}(u,v) := h_{E'}(\varphi(u),\varphi(v)).
$$
\par
Throughout this section, we assume that
$P(\cdot)\in\frak{gl}({\bf C}^{r})^{{\rm GL}({\bf C}^{r})}$ is a ${\rm GL}({\mathbf C}^{r})$-invariant polynomial 
on the Lie algebra $\frak{gl}({\bf C}^{r})$ 
with respect to the adjoint action. We set 
$$
P(\overline{E}) = P(E,h_{E}) := P\left(-R_{\overline{E}}/2\pi i\right).
$$
By \cite[I, e), f)]{BGS88}, \cite[Sect.1.2.4]{GilletSoule90}, we have the Bott-Chern class 
$$
\widetilde{P}(\overline{E},\overline{E'}) := \widetilde{P}(E ; h_{E}, \varphi^{*}h_{E'})\in\widetilde{A}(M^{0})
$$ 
such that 
$$
dd^{c}\widetilde{P}(\overline{E},\overline{E'}) = P(\overline{E}) - P(\overline{E'}).
$$
In this section, we construct a natural extension of $\widetilde{P}(\overline{E},\overline{E'})$ to a current on $M$. 
For later use in the next section, when $M$ admits a proper holomorphic function $f \colon M \to \varDelta$ with $N \subset f^{-1}(0)$, 
we study the asymptotic behavior of the fiber integral
$f_{*}\widetilde{P}(\overline{E},\overline{E'})$ as a function on the punctured disc $\varDelta^{*}$.
We remark that, when $N$ is a smooth hypersurface and $\varphi \colon {\mathcal O}_{M}(E) \to {\mathcal O}_{M}(E')$ provides 
a resolution of $i_{*}(E'|_{N})$ with $i \colon N \hookrightarrow M$ being an embedding, 
a natural extension of $\widetilde{P}(\overline{E},\overline{E'})$ to a current on $M$ was constructed by 
Bismut, Gillet, and Soul\'e \cite{BGS90b}. While their construction extends to general immersions of higher codimension,
it is not applicable to our setting because $N$ is not necessarily smooth. 
\par
We introduce a particular representative of the Bott-Chern class $\widetilde{P}(\overline{E},\overline{E'})$, 
which we use in Section~\ref{sect:6}. To this end, we recall the construction of the Bott-Chern class given in \cite[I, f)]{BGS88},
\cite[Sect.\,1.2.3]{GilletSoule90}. (See \cite[I, Cor.\,1.30]{BGS88} for another construction.)
Let $v \colon F \to F'$ be an isomorphism of holomorphic vector bundles on $M$.
Let $h_{F}$, $h_{F'}$ be Hermitian metrics on $F$, $F'$, respectively.
Let $\pi_{1}\colon M\times{\bf P}^{1}\to M$ and $\pi_{2}\colon M\times{\bf P}^{1}\to{\bf P}^{1}$ be the natural projections.
Set $\widetilde{F} := \pi_{1}^{*}F$. Define a Hermitian metric $h_{\widetilde{F}}$ on $\widetilde{F}$ as
$$
h_{\widetilde{F}} := \frac{1}{1+|\zeta|^{2}} h_{F} + \frac{|\zeta|^{2}}{1+|\zeta|^{2}} h_{F'}^{v},
$$
where $\zeta$ is the inhomogeneous coordinate of ${\mathbf P}^{1}$.
Then $h_{\widetilde{F}}|_{M \times\{0\}} = h_{F}$ and $\widetilde{h}_{\widetilde{F}}|_{M \times\{\infty\}} = v^{*}h_{F'} = h_{F'}^{v}$.
Setting $S=0$, $E = F$, $Q = F'$ in \cite[Sect.\,1.2.3]{GilletSoule90}, 
we deduce from \cite[(1.2.3.2)]{GilletSoule90} that in $\widetilde{A}_{M}$,
\begin{equation}
\label{eqn:Bott:Chern}
\widetilde{P}( \overline{F}, \overline{F'} ) 
= \int_{{\mathbf P}^{1}} P\left( \frac{i}{2\pi} R(\widetilde{F}, h_{\widetilde{F}}) \right)\, \log |\zeta|^{2}.
\end{equation}
(We remark that we set $S=0$, $E=F$, $Q=F'$ here, while $S=F$, $E=F'$, $Q=0$ in \cite[Sect.\,1.2.4]{GilletSoule90}. 
This explains the difference of the sign between \eqref{eqn:Bott:Chern} and \cite[(1.2.3.2)]{GilletSoule90}.)
Under the identification $(F', h_{F'}) \cong (F, h_{F'}^{v})$,
$\widetilde{P}( \overline{F}, \overline{F'} )$ is also denoted by $\widetilde{P}( F ; h_{F}, h_{F'} )$. 
In what follows, the representative of $\widetilde{P}( \overline{F}, \overline{F'} )$ given by \eqref{eqn:Bott:Chern} 
is called the Bott-Chern form and is still denoted by the same symbol $\widetilde{P}( \overline{F}, \overline{F'} )$.

\subsection
{The Grassmann bundle and the section associated to $\varphi$}
\label{sect:6.2}
\par
Let $\sigma_{0},\sigma_{\infty}\in H^{0}({\bf P}^{1},{\mathcal O}_{{\bf P}^{1}}(1))$ be the canonical sections such that
${\rm div}(\sigma_{0})=0$, ${\rm div}(\sigma_{\infty})=\infty$. Then $\zeta=\sigma_{0}/\sigma_{\infty}$. 
Let $h_{{\mathcal O}_{{\bf P}^{1}}(1)}$ be the standard Hermitian metric on ${\mathcal O}_{{\bf P}^{1}}(1)$ with
$$
\|\sigma_{0}(\zeta)\|^{2}=|\zeta|^{2}/(1+|\zeta|^{2}),
\qquad
\|\sigma_{\infty}(\zeta)\|^{2}=1/(1+|\zeta|^{2}).
$$
\par
Set 
$$
E(1) := \pi_{1}^{*}E\otimes\pi_{2}^{*}{\mathcal O}_{{\bf P}^{1}}(1),
\qquad
E'(1) := \pi_{1}^{*}E'\otimes\pi_{2}^{*}{\mathcal O}_{{\bf P}^{1}}(1).
$$
Then $E(1)$ and $E'(1)$ are endowed with the product Hermitian metrics 
$h_{E(1)} = \pi_{1}^{*}h_{E}\otimes\pi_{2}^{*}h_{{\mathcal O}_{{\bf P}^{1}}(1)}$ and 
$h_{E'(1)} = \pi_{1}^{*}h_{E'}\otimes\pi_{2}^{*}h_{{\mathcal O}_{{\bf P}^{1}}(1)}$, respectively. 
\par
Let $\varepsilon_{\varphi}\colon\pi_{1}^{*}E\to E(1)\oplus E'(1)$ be the homomorphism of vector bundles defined as
\begin{equation}
\label{eqn:hom:v.b.}
(\varepsilon_{\varphi})_{(x,\zeta)}(e) := (e\otimes\sigma_{\infty}(\zeta),\varphi_{x}(e)\otimes\sigma_{0}(\zeta)),
\qquad
(x,\zeta) \in M \times {\mathbf P}^{1},
\quad
e\in E_{x}.
\end{equation}
Define a holomorphic Hermitian vector bundle $\overline{\mathcal E}$ on $(M\times{\bf P}^{1})\setminus(N\times\{\infty\})$ as
$$
\overline{\mathcal E} := (\pi_{1}^{*}E,\|\sigma_{\infty}\|^{2}h_{E} + \|\sigma_{0}\|^{2}h_{E'}^{\varphi}).
$$
Then we have an isomorphism of holomorphic Hermitian vector bundles
$$
\overline{\mathcal E}\cong(\varepsilon_{\varphi}(\pi_{1}^{*}E), (h_{E(1)}\oplus h_{E'(1)})|_{\varepsilon_{\varphi}(E)})
$$
on $(M\times{\bf P}^{1})\setminus(N\times\{\infty\})$. In particular, for $\zeta\not=\infty$, we have
\begin{equation}
\label{eqn:vb}
\left(E,\frac{1}{1+|\zeta|^{2}}h_{E}+\frac{|\zeta|^{2}}{1+|\zeta|^{2}}h_{E'}^{\varphi}\right)
\cong
(\varepsilon_{\varphi}(\pi_{1}^{*}E),(h_{E(1)}\oplus h_{E'(1)})|_{\varepsilon_{\varphi}(E)})|_{M\times\{\zeta\}}
\end{equation}
and for $\zeta=\infty$, we have
$$
(E|_{M^{0}},h_{E',\varphi}) 
\cong 
(\varepsilon_{\varphi}(\pi_{1}^{*}E),(h_{E(1)}\oplus h_{E'(1)})|_{\varepsilon_{\varphi}(E)})|_{M^{0}\times\{\infty\}} 
= \overline{E'}|_{M^{0}}.
$$
\par
Let $\varPi\colon{\rm Gr}(r, E(1)\oplus E'(1)) \to M\times{\bf P}^{1}$ be the Grassmann bundle such that
$$
{\rm Gr}(r, E(1)\oplus E'(1))_{(x,\zeta)} = \varPi^{-1}(x,\zeta) := {\rm Gr}(r, E(1)_{(x,\zeta)}\oplus E'(1)_{(x,\zeta)})
$$
for $(x,\zeta)\in M\times{\bf P}^{1}$. Hence ${\rm Gr}(r, E(1)\oplus E'(1))$ parametrizes linear subspacess of dimension $r$ 
of the fibers of $E(1)\oplus E'(1)$. 
\par
Let 
$$
{\mathscr U} := {\mathscr U}^{E(1)\oplus E'(1)} \to {\rm Gr}(r, E(1)\oplus E'(1))
$$ 
be the universal vector bundle of rank $r$. 
Since ${\mathscr U} \subset \varPi^{*}(E(1)\oplus E'(1))$, ${\mathscr U}$ is endowed with the Hermitian metric 
$h_{\mathscr U} := \varPi^{*}(h_{E(1)}\oplus h_{E'(1)})|_{\mathscr U}$. We set 
$$
\overline{\mathscr U} := ({\mathscr U}, h_{\mathscr U}).
$$
\par
Let $\sigma_{\varphi}\colon(M\times{\bf P}^{1})\setminus(N\times\{\infty\})\to{\rm Gr}(r, E(1)\oplus E'(1))$ be the holomorphic section 
associated to the homomorphism $\varepsilon_{\varphi}$ defined by
\begin{equation}
\label{eqn:taut:section}
\sigma_{\varphi}(x,\zeta) := [(\varepsilon_{\varphi})_{(x,\zeta)}(E_{x})].
\end{equation}
Here, for a subspace $V\subset E_{x}\oplus E'_{x}$ of dimension $r$, $[V]\in{\rm Gr}(r,E(1)\oplus E'(1))_{(x,\zeta)}$ denotes the point 
corresponding to $V$. Since
$$
( {\mathscr U}_{\sigma_{\varphi}(x,\zeta)}, h_{{\mathcal U}, \sigma_{\varphi}(x,\zeta)} )
=
( \varepsilon_{\varphi}(E_{x}\times\{\zeta\}), h_{E(1)}\oplus h_{E'(1)}|_{\varepsilon_{\varphi}(E_{x}\times\{\zeta\})} ),
$$
comparing this with \eqref{eqn:vb}, we have
$$
\overline{\mathcal E} = \sigma_{\varphi}^{*}\overline{\mathscr U}
$$
as holomorphic Hermitian vector bundles on $(M\times{\bf P}^{1})\setminus(N\times\{\infty\})$. 
Since $N\times\{\infty\}$ has codimension $2$ in $M \times {\mathbf P}^{1}$, 
the holomorphic section $\sigma_{\varphi}$ from $(M\times{\bf P}^{1})\setminus(N\times\{\infty\})$ to ${\rm Gr}(r, E(1)\oplus E'(1))$ 
extends to a meromorphic map (\cite{Siu75}).
\par
Fix a resolution of the indeterminacy 
$$
q\colon{\mathfrak M}\to M\times{\mathbf P}^{1}
$$
of the meromorphic map $\sigma_{\varphi}\colon M\times{\bf P}^{1}\dashrightarrow{\rm Gr}(r, E(1)\oplus E'(1))$. Set 
\begin{equation}
\label{eqn:divisor}
{\mathfrak D}_{\varphi} := q^{-1}(N\times\{\infty\}),
\end{equation}
which is endowed with the structure of a divisor on ${\mathfrak M}$ such that
\begin{equation}
\label{eqn:div:0:infty}
{\rm div}(q^{*}\zeta)^{-1} = \widetilde{M\times\{\infty\}} + {\mathfrak D}_{\varphi} - q^{-1}(M\times\{0\}).
\end{equation}
Here $\widetilde{M\times\{\infty\}}$ is the proper transform of $M\times\{\infty\}$. 
Since $N$ is compact and $q$ is proper, ${\mathfrak D}_{\varphi}$ is compact.
By definition, there is a holomorphic map 
$$
\widetilde{\sigma}_{\varphi} \colon {\mathfrak M} \to {\rm Gr}(r, E(1)\oplus E'(1))
$$ 
such that $\widetilde{\sigma}_{\varphi}=\sigma_{\varphi}\circ q$ on ${\frak M}\setminus {\mathfrak D}_{\varphi}$. 
We have the following commutative diagram:
$$
   \xymatrix{
    {\mathfrak D}_{\varphi} \ar@{^{(}-_>}[r] \ar[d]_{q} & {\mathfrak M} \ar[d]_{q} \ar[rd]^{\widetilde{\sigma}_{\varphi}}
    \\
    N\times\{\infty\} \ar@{^{(}-_>}[r] & M\times{\bf P}^{1}  \ar@{.>}[r]_{\sigma_{\varphi}} & {\rm Gr}(r, E(1)\oplus E'(1))
   }
$$
Then $\widetilde{\sigma}_{\varphi}^{*}\overline{\mathscr U}$ is a holomorphic Hermitian vector bundle on ${\mathfrak M}$ such that
on ${\mathfrak M} \setminus {\mathfrak D}_{\varphi}$,
\begin{equation}
\label{eqn:universal:property:2}
\widetilde{\sigma}_{\varphi}^{*}\overline{\mathscr U} = q^{*}\overline{\mathcal E}
\end{equation}
By \eqref{eqn:universal:property:2}, we get the following equality of differential forms 
on ${\mathfrak M} \setminus {\mathfrak D}_{\varphi}$
\begin{equation}
\label{eqn:char:form:1}
q^{*}P(\overline{\mathcal E}) = \widetilde{\sigma}_{\varphi}^{*}P(\overline{\mathscr U}),
\end{equation}
where $\widetilde{\sigma}_{\varphi}^{*} P(\overline{\mathscr U})\in \bigoplus_{p\geq0} A^{p,p}({\mathfrak M})$.

\subsection
{An extension of Bott-Chern forms to currents}
\label{sect:6.3}
\par
Following MacPherson's graph construction for vector bundle homomorphism  \cite[Example 18.1.6]{Fulton98}
(see also \cite[p.341 Eq.(14) and Def.\,18.1]{Fulton98}), for a polynomial $P=P(c_{1},c_{2},\ldots)$ in the Chern classes, we set
\begin{equation}
\label{eqn:local:Chern}
P^{M}_{N}(\varphi)
:=
(\pi_{1}\circ q)_{*}\left[c_{1}({\mathfrak D}_{\varphi}) \wedge \widetilde{\sigma}_{\varphi}^{*} P({\mathscr U})\right].
\end{equation}
Here $P^{M}_{N}(\varphi)$ is regarded as a linear form on $H^{*}_{\rm DR}(M, {\mathbf C})$ defined as
\begin{equation}
\label{eqn:def:local:Chern}
\begin{aligned}
\int_{M} P^{M}_{N}(\varphi) \wedge \chi
&=
\int_{\mathfrak M} 
c_{1}({\mathfrak D}_{\varphi}) \wedge \widetilde{\sigma}_{\varphi}^{*} P({\mathscr U}) \wedge (\pi_{1}\circ q)^{*}\chi
\\
&=
\int_{{\mathfrak D}_{\varphi}} \widetilde{\sigma}_{\varphi}^{*} P( {\mathscr U} ) \wedge (\pi_{1}\circ q)^{*}\chi
\end{aligned}
\end{equation}
for $\chi \in H^{*}_{\rm DR}(M, {\mathbf C})$. Hence $P^{M}_{N}(\varphi) \in H^{*}_{{\rm DR},c}(M, {\mathbf C})$.
By \cite[Example 18.1.6 (c)]{Fulton98}, we have the following equality of cohomology classes with compact support:
\begin{equation}
\label{eqn:difference:Chern}
P(E) - P(E') = P^{M}_{N}(\varphi).
\end{equation}
In this subsection, using the Bott-Chern forms, we provide a refinement of \eqref{eqn:difference:Chern} at the level of currents.
Recall that on $M^{0}$, $\widetilde{P}(\overline{E},\overline{E'})$ is given by the smooth form 
\begin{equation}
\label{eqn:def:Bott:Chern}
\begin{aligned}
\widetilde{P}(\overline{E},\overline{E'})
&=
(\pi_{1})_{*}\left[ P(\overline{\mathcal E})\,\log|\zeta|^{2} \right]
=
(\pi_{1}\circ q)_{*}\left[ P(q^{*}\overline{\mathcal E})\cdot q^{*}(\log|\zeta|^{2}) \right]
\\
&=
(\pi_{1}\circ q)_{*}\left[ \widetilde{\sigma}_{\varphi}^{*} P(\overline{\mathscr U})\cdot q^{*}(\log|\zeta|^{2}) \right],
\end{aligned}
\end{equation}
where we used \eqref{eqn:char:form:1} to get the last expression. 
Since $\widetilde{\sigma}_{\varphi}^{*}P(\overline{\mathscr U})\cdot q^{*}(\log|\zeta|^{2})$ is a differential form on ${\mathfrak M}$ 
with coefficients in $L^{1}_{\rm loc}$, $(\pi_{1}\circ q)_{*}[\widetilde{\sigma}_{\varphi}^{*}P(\overline{\mathscr U})\cdot q^{*}(\log|\zeta|^{2})]$ 
defines a current on $M$ by the properness of $\pi_{1}\circ q$.

\begin{definition}
\label{def:canonical:ext}
The Bott-Chern type current $\langle \widetilde{P}(\overline{E},\overline{E'}) \rangle$ on $M$ is defined as
$$
\langle \widetilde{P}(\overline{E},\overline{E'}) \rangle 
:= 
(\pi_{1}\circ q)_{*}[\widetilde{\sigma}_{\varphi}^{*} P(\overline{\mathscr U})\cdot q^{*}(\log|\zeta|^{2})].
$$ 
\end{definition}

Then \eqref{eqn:difference:Chern} is refined as follows.

\begin{theorem}
\label{thm:canonical:ext}
The following identity of currents on $M$ holds:
$$
dd^{c} \langle \widetilde{P}(\overline{E},\overline{E'}) \rangle
=
P(\overline{E}) - P(\overline{E'}) 
- (\pi_{1}\circ q)_{*}\left[ \delta_{{\mathfrak D}_{\varphi}}\wedge\widetilde{\sigma}_{\varphi}^{*} P(\overline{\mathscr U}) \right].
$$
Here $\delta_{{\mathfrak D}_{\varphi}}$ is the Dirac $\delta$-current associated to the divisor ${\mathfrak D}_{\varphi}$.
\end{theorem}

\begin{pf}
We remark that since $\widetilde{\sigma}_{\varphi}^{*}P(\overline{\mathscr U})$ is a smooth form on ${\mathfrak M}$, the product of currents
$\delta_{{\mathfrak D}_{\varphi}}\wedge\widetilde{\sigma}_{\varphi}^{*}P(\overline{\mathscr U})$ is well-defined.
Since $\partial$ and $\bar{\partial}$ commute with $(\pi_{1}\circ q)_{*}$ and $dd^{c} \log |\zeta|^{2} = \delta_{0} - \delta_{\infty}$, we get 
$$
\begin{aligned}
dd^{c} \langle \widetilde{P}(\overline{E},\overline{E'}) \rangle
&=(\pi_{1}\circ q)_{*}\left[\widetilde{\sigma}_{\varphi}^{*}P(\overline{\mathscr U})\wedge\{dd^{c}q^{*}(\log|\zeta|^{2})\}\right]
\\
&=
(\pi_{1}\circ q)_{*}\left[\widetilde{\sigma}_{\varphi}^{*}P(\overline{\mathscr U})\wedge
\{ \delta_{q^{-1}(M\times\{0\})} - \delta_{\widetilde{M\times\{\infty\}}} - \delta_{{\mathfrak D}_{\varphi}} \}  \right]
\\
&= P(\overline{E}) - P(\overline{E'}) 
- (\pi_{1}\circ q)_{*}\left[\widetilde{\sigma}_{\varphi}^{*}P(\overline{\mathscr U})\wedge\delta_{{\mathfrak D}_{\varphi}}\right],
\end{aligned}
$$
where the second equality follows from \eqref{eqn:div:0:infty}. This completes the proof.
\end{pf}

In certain special cases, an explicit formula for the current $\langle \widetilde{P}(\overline{E},\overline{E'}) \rangle$, 
along with a comparison to the Bott-Chern singular currents of Bismut-Gillet-Soul\'e \cite{BGS90}, \cite{BGS90b}, 
will be presented in Section~\ref{sect:9.3}.

\subsection
{The asymptotic behavior of the fiber integral $f_{*}[\chi\wedge \langle \widetilde{P}(\overline{E},\overline{E'})\rangle ]$}
\label{sect:6.4}
\par

\begin{theorem}
\label{thm:pusf:forward:Bott:Chern}
Let $\varDelta \subset {\mathbf C}$ be the unit disc. Let $f \colon M \to \varDelta$ be a surjective proper holomorphic function with 
$N \subset f^{-1}(0)$. Let $\chi$ be a $\partial$ and $\bar{\partial}$ closed smooth differential form on $M$. 
Then 
$$
f_{*}[\chi\wedge \langle \widetilde{P}(\overline{E},\overline{E'}) \rangle]^{(n,n)}
\equiv_{\mathcal B} -\left\{ \int_{M} P^{M}_{N}(\varphi) \wedge \chi \right\}\,\log | t |^{2}.
$$
\end{theorem}

\begin{pf}
Set $A := \int_{M} P^{M}_{N}( \varphi )\wedge\chi$. 
By \cite[Lemma 4.4]{Yoshikawa07} and Definition~\ref{def:canonical:ext},
there exist $\alpha \in {\mathbf R}$ and $g \in {\mathcal B}(\varDelta)$ such that
$f_{*}[\chi\wedge \langle \widetilde{P}(\overline{E},\overline{E'}) \rangle]^{(n,n)} = \alpha\,\log | t |^{2} + g$. 
It suffices to prove $\alpha=-A$.
Since $\partial\chi=\bar{\partial}\chi=0$, it follows from Theorem~\ref{thm:canonical:ext} and the commutativity of $f_{*}$ and $dd^{c}$ 
that as currents on $\varDelta$,
$$
\begin{aligned}
\,&
dd^{c}f_{*}[\chi\wedge \langle \widetilde{P}(\overline{E},\overline{E'}) \rangle]^{(n,n)}
=
f_{*}[\chi\wedge\{ dd^{c} \langle\widetilde{P}(\overline{E},\overline{E'}) \rangle\}]^{(n+1,n+1)}
\\
&=
f_{*}[\chi\wedge\{ P(\overline{E}) - P(\overline{E'})\}]^{(n+1,n+1)}
-
f_{*}\left[
\chi\wedge(\pi_{1}\circ q)_{*}\left\{\delta_{{\mathfrak D}_{\varphi}}\wedge\widetilde{\sigma}_{\varphi}^{*} P(\overline{\mathcal U})\right\}
\right]
\\
&\equiv
-f_{*}\left[
\chi\wedge(\pi_{1}\circ q)_{*}\left\{\delta_{{\mathfrak D}_{\varphi}}\wedge\widetilde{\sigma}_{\varphi}^{*} P(\overline{\mathcal U})\right\}
\right]
\mod
\bigcup_{p>1} L^{p}_{\rm loc}(\varDelta)\cdot dt\wedge d\bar{t}.
\end{aligned}
$$
Here we used $f_{*}( A^{n+1,n+1}(M)) \subset \bigcup_{p>1} L_{\rm loc}^{p}(\varDelta) \cdot dt\wedge d\bar{t}$
(\cite[Lemma 9.2]{Yoshikawa07}) to get the last equality.
Since $(\pi_{1}\circ q)({\mathfrak D}_{\varphi})\subset N$ and hence $(f\circ\pi_{1}\circ q)({\mathfrak D}_{\varphi})=\{0\}$ by assumption, 
we get for all $\varrho\in C^{\infty}_{0}(\varDelta)$
$$
\begin{aligned}
\,&
\int_{\varDelta} \varrho\cdot f_{*}
\left[
\chi \wedge (\pi_{1}\circ q)_{*} \left\{ \delta_{{\mathfrak D}_{\varphi}} \wedge \widetilde{\sigma}_{\varphi}^{*} P(\overline{\mathcal U}) \right\}
\right]
=
\int_{M} f^{*}\varrho\cdot\chi\wedge(\pi_{1}\circ q)_{*} \left\{ \delta_{{\mathfrak D}_{\varphi}} \wedge \widetilde{\sigma}_{\varphi}^{*}
P(\overline{\mathcal U}) \right\}
\\
&=
\int_{\frak M} 
\{ (\pi_{1}\circ q)^{*}(f^{*}\varrho\cdot\chi)\} \wedge \left\{ \delta_{{\mathfrak D}_{\varphi}} \wedge \widetilde{\sigma}_{\varphi}^{*}
P(\overline{\mathcal U}) \right\}
=
\int_{{\mathfrak D}_{\varphi}} \{ (\pi_{1}\circ q)^{*}(f^{*}\varrho\cdot\chi) \} \wedge \widetilde{\sigma}_{\varphi}^{*}P(\overline{\mathcal U})
\\
&=
\varrho(0)\int_{{\mathfrak D}_{\varphi}} (\pi_{1}\circ q)^{*}\chi \wedge \widetilde{\sigma}_{\varphi}^{*}P(\overline{\mathcal U})
=
A\,\varrho(0)
=
A \int_{\varDelta} \varrho\, dd^{c} \log | s |^{2}.
\end{aligned}
$$
Since $dd^{c}{\mathcal B}(\varDelta) \subset \bigcup_{p>1} L^{p}_{\rm loc}(\varDelta)\cdot dt\wedge d\bar{t}$ and hence
$dd^{c}(\alpha \log |t|^{2} + g) \equiv \alpha dd^{c} \log |t|^{2} \mod \bigcup_{p>1} L^{p}_{\rm loc}(\varDelta)\cdot dt\wedge d\bar{t}$, 
we get $\alpha=-A$. 
\end{pf}

\subsection
{Bott-Chern type currents for generically injective homomorphisms}
\label{sect:6.5}
\par
For later use in the next section, we determine the asymptotic behavior of the fiber integral of Bott-Chern forms in a more degenerate setting.
Let $(V, h_{V})$ be a holomorphic Hermitian vector bundle on $M$ of rank $l > r$. 
Let $\varPi \colon {\rm Gr}(r, V) \to M$ be the Grassmann bundle on $M$ such that ${\rm Gr}(r, V)_{x} = {\rm Gr}(r, V_{x})$ for $x \in M$. 
Let
$$
\psi \colon E \to V
$$
be a homomorphism of holomorphic vector bundles on $M$ such that $\psi$ is injective on $M^{0}$. In particular, it can be possible that 
${\rm rk}(\psi_{x}) < {\rm rk}(E)$ for some $x \in N$. Then 
$$
W := \psi(E)
$$ 
is a holomorphic subbundle of $V$ defined on $M^{0}$ endowed with the Hermitian metric $h_{W} := h_{V}|_{W}$. 
We set $\overline{W} := (W, h_{W})$. On $M^{0}$, we have the Bott-Chern form $\widetilde{P}(\overline{E}, \overline{W})$.
As before, we construct an extension of $\widetilde{P}(\overline{E}, \overline{W})$ to a current on $M$.

\begin{lemma}
\label{lemma:mer:extension}
Define holomorphic map ${\mathbf s}_{\psi} \colon M^{0} \to {\rm Gr}(r, V)$ by
$$
{\mathbf s}_{\psi}(x) := [ \psi(E_{x}) ] \in {\rm Gr}(r, V_{x})
\qquad
(x \in M^{0}).
$$
Then ${\mathbf s}_{\psi}$ extends to a meromorphic map from $M$ to ${\rm Gr}(r, V)$.
\end{lemma}

\begin{pf}
It suffices to prove the assertion when $M=\varDelta^{m}$, $E={\mathbf C}^{r}$, $V={\mathbf C}^{l}$, 
and ${\rm Gr}(r,V) = {\rm Gr}(r, {\mathbf C}^{l})$, where $m=\dim M$. 
Fix bases $\{ {\mathbf e}_{1}, \ldots, {\mathbf e}_{r} \}$ of ${\mathbf C}^{r}$ and $\{{\mathbf v}_{1},\ldots,{\mathbf v}_{l}\}$ of $V$. 
Then, for $I=\{i_{1}<\cdots<i_{r}\}$, there exists $\psi_{I} \in {\mathcal O}(M)$ such that 
$\Lambda^{r} \psi( x ) ( {\mathbf e}_{1}\wedge\cdots\wedge {\mathbf e}_{r}) = \sum_{ | I | = r} \psi_{I}(x)\,{\bf v}_{I}$, 
where $\Lambda^{r}\psi(x) \in {\rm Hom}( \Lambda^{r}E_{x}, \Lambda^{r}V_{x})$ is the homomorphism induced from $\psi$ and 
${\mathbf v}_{I}:={\mathbf v}_{i_{1}} \wedge \cdots \wedge {\mathbf v}_{i_{k}}$. Set $N:=\binom{l}{r}$. 
By the injectivity of $\psi$ on $M^{0}$, there exists an $I$ such that $\psi_{I}(x) \not=0$. 
By identifying ${\rm Gr}(r,{\bf C}^{l})$ with its image in ${\mathbf P}^{N-1}$ via the Pl\"ucker embedding, 
${\mathbf s}_{\psi}$ is expressed as the holomorphic map 
${\mathbf s}_{\psi} \colon M^{0} \ni x \to [\cdots : \psi_{I}(x) : \cdots ] \in {\mathbf P}^{N-1}$. 
Since ${\mathbf s}_{\psi}$ lifts to the non-constant holomorphic map $M \ni x \to (\ldots, \psi_{I}(x), \ldots) \in {\mathbf C}^{N}$ 
from $M$ to ${\mathbf C}^{N}$, ${\mathbf s}_{\psi}$ extends to a meromorphic map from $M$ to ${\mathbf P}^{N-1}$. 
Since the image of ${\mathbf s}_{\psi}$ is contained in ${\rm Gr}(r,{\mathbf C}^{l})$, 
${\mathbf s}_{\psi}$ is a meromorphic map from $M$ to ${\rm Gr}(r,{\mathbf C}^{l})$.
\end{pf}

Let 
$$
\nu \colon M^{\natural} \to M
$$ 
be a resolution of the indeterminacy of ${\mathbf s}_{\psi}$. By definition,
$$
{\mathbf s}_{\psi}^{\natural} := {\mathbf s}_{\psi}\circ \nu \colon M^{\natural} \setminus \nu^{-1}(N) \to {\rm Gr}(r, V)
$$
extends to a holomorphic map from $M^{\natural}$ to ${\rm Gr}(r, V)$. 
Let ${\mathcal U}^{V} \subset \varPi^{*}V$ be the universal rank $r$ bundle on ${\rm Gr}(r, V)$. 
Since ${\mathbf s}_{\psi}^{*}{\mathcal U}^{V} = \psi(E)$ on $M^{0}$, 
$({\mathbf s}_{\psi}^{\natural})^{*}{\mathcal U}^{V}$ is a holomorphic vector bundle on $M^{\natural}$ such that
$$
({\mathbf s}_{\psi}^{\natural})^{*}{\mathcal U}^{V} |_{\nu^{-1}(M^{0})} = \nu^{*}(\psi(E)) |_{\nu^{-1}(M^{0})} = \nu^{*}W. 
$$
Since $\nu^{*}E = M^{\natural} \times_{M} E = M^{\natural} \times_{{\rm Gr}(r, V)} \varPi^{*}E$ and 
$({\mathbf s}_{\psi}^{\natural})^{*}{\mathcal U}^{V} = M^{\natural} \times_{{\rm Gr}(r, V)} {\mathcal U}^{V}$, under this identification,
$\psi$ induces the following homomorphism of holomorphic vector bundles on $M^{\natural}$
$$
\psi^{\natural} \colon \nu^{*}E \to 
({\mathbf s}_{\psi}^{\natural})^{*} {\mathcal U}^{V} \subset ({\mathbf s}_{\psi}^{\natural})^{*}\varPi^{*}V,
\qquad
\psi^{\natural} := {\rm id}_{M^{\natural}} \times \psi.
$$
Let $N^{\natural} \subset M^{\natural}$ be the degeneracy locus of $\psi^{\natural}$, i.e., the locus where $\psi^{\natural}$ is not injective. 
Then $M^{\natural} \setminus N^{\natural} \supset \nu^{-1}(M^{0})$ by the injective of $\psi|_{M^{0}}$.
Since ${\rm Im}\,\psi^{\natural}|_{M^{0}} \subset ({\mathbf s}_{\psi}^{\natural})^{*} {\mathcal U}^{V}|_{M^{0}}$, 
we get $\psi^{\natural}( \nu^{*}E ) \subset ({\mathbf s}_{\psi}^{\natural})^{*} {\mathcal U}^{V}$ by the continuity of $\psi^{\natural}$.
Let $h_{{\mathcal U}^{V}}$ be the Hermitian metric on ${\mathcal U}^{V}$ induced from $\varPi^{*}h_{V}$ via the inclusion 
${\mathcal U}^{V} \subset \varPi^{*}V$. 
Set
$$
{E}^{\natural} := \nu^{*}E,
\qquad
{W}^{\natural} := ({\mathbf s}_{\psi}^{\natural})^{*} {\mathcal U}^{V},
\qquad
h_{E^{\natural}} := \nu^{*}h_{E},
\qquad
h_{W^{\natural}} := ({\mathbf s}_{\psi}^{\natural})^{*}h_{{\mathcal U}^{V}}.
$$
We define $\overline{E^{\natural}} := (E^{\natural}, h_{E^{\natural}})$ and $\overline{W^{\natural}} := (W^{\natural}, h_{W^{\natural}})$.
Then we have the following equality of differential forms on $M^{\natural} \setminus N^{\natural}$
\begin{equation}
\label{eqn:BC:1}
\widetilde{P}(\overline{E^{\natural}}, \overline{W^{\natural}}) 
= ({\mathbf s}_{\psi}^{\natural})^{*}\widetilde{P}(\varPi^{*}\overline{E}, \overline{\mathcal U}^{V}).
\end{equation}
Since $\psi^{\natural} \colon \overline{E^{\natural}} \to \overline{W^{\natural}}$ is a homomorphism of 
holomorphic Hermitian vector bundles on $M^{\natural}$ invertible on $M^{\natural} \setminus N^{\natural}$, 
we have the Bott-Chern type current $\langle \widetilde{P}(\overline{E^{\natural}},\overline{W^{\natural}}) \rangle$ on $M^{\natural}$ 
by Definition~\ref{def:canonical:ext}.
By \eqref{eqn:BC:1} and Theorem~\ref{thm:canonical:ext}, we have
\begin{equation}
\label{eqn:BC:2}
\langle \widetilde{P}(\overline{E^{\natural}}, \overline{W^{\natural}}) \rangle |_{\nu^{-1}(M^{0})} 
= \nu^{*}\widetilde{P}(\overline{E}, \overline{W})|_{M^{0}}.
\end{equation}

\begin{definition}
\label{def:canonical:ext:2}
Define the Bott-Chern type current $ \langle \widetilde{P}(\overline{E}, \overline{W}) \rangle$ on $M$ as
$$
\langle \widetilde{P}(\overline{E}, \overline{W}) \rangle := \nu_{*}\widetilde{P}(\overline{E^{\natural}}, \overline{W^{\natural}}).
$$
\end{definition}

By \eqref{eqn:BC:2}, $\langle \widetilde{P}(\overline{E}, \overline{W}) \rangle$ is an extension of 
$\widetilde{P}(\overline{E}, \overline{W})|_{M^{0}}$ to a current on $M$.

\begin{theorem}
\label{thm:pusf:forward:Bott:Chern:2}
Let $\chi$ be a $\partial$ and $\bar{\partial}$ closed smooth differential from on $M$. 
Let $f \colon M \to \varDelta$ be a proper surjective holomorphic function with $N \subset f^{-1}(0)$. Then 
$$
f_{*}[\chi\wedge\widetilde{P}(\overline{E}, \overline{W})]^{(n,n)}
\equiv_{\mathcal B}
-\left\{ \int_{M^{\natural}} \nu^{*}\chi \wedge P^{M^{\natural}}_{N^{\natural}}(\psi^{\natural}) \right\} 
\log | t |^{2}.
$$
In particular, $\int_{M^{\natural}} \nu^{*}\chi \wedge P^{M^{\natural}}_{N^{\natural}}(\psi^{\natural})$ 
is independent of the choice of $M^{\natural}$.
\end{theorem}

\begin{pf}
Since 
$f_{*}[\chi\wedge \langle \widetilde{P}(\overline{E},\overline{W}) \rangle]^{(n,n)}=
(f\circ \nu)_{*} [ \nu^{*}\chi\wedge \langle \widetilde{P}(\overline{E^{\natural}},\overline{W^{\natural}}) \rangle ]^{(n,n)}$,
the result follows from Theorem~\ref{thm:pusf:forward:Bott:Chern}.
\end{pf}

\section
{The leading term of the asymptotic expansion}
\label{sect:7}
\par
In this section, we consider the semi-stable reduction $f \colon Y \to T$ of $\pi \colon X \to S$ as in \eqref{eqn:ss:red}, 
induced by the map $\mu \colon T \to S$. Recall that $F \colon Y \to X$ is the corresponding morphism (cf. \eqref{eqn:ss:red}). 
Since $S$ is an open subset of the compact Riemann surface $C$ and $X = \pi^{-1}(S)$ is the corresponding open subset of 
a projective algebraic manifold ${\mathcal X}$ (cf. Section~\ref{sect:3.1}), we can assume the following: 
\begin{itemize}
\item
$T$ is an open subset of a compact Riemann surface $C'$.
\item
$Y$ is an open subset of a projective algebraic manifold ${\mathcal Y}$. 
\item
$f$ is induced by a surjective holomorphic map from ${\mathcal Y}$ to $C'$. 
\item
$F^{*}\xi$ extends to a holomorphic vector bundle over ${\mathcal Y}$. 
\end{itemize}
The goal of this section is a formula for $\kappa_{\pi,K_{X/S}(\xi)}$ in Theorem~\ref{thm:main:1} 
in terms of various characteristic classes associated to the semi-stable reduction $f \colon Y \to T$ and the map $F \colon Y \to X$
(Section~\ref{sect:7.1}). As its application, we will prove a certain locality of $\kappa_{\pi,K_{X/S}(\xi)}$, which, together with the results in
Section~\ref{sect:7.2}, enables us to give an explicit formula for $\kappa_{\pi,K_{X/S}(\xi)}$ when $\xi$ is topologically trivial 
near ${\rm Sing}\,X_{0}$ (Section~\ref{sect:7.2}). In particular, when ${\rm Sing}\,X_{0}$ is isolated, considering the two cases of $\xi$, 
one trivial and the other an ample line bundle, we relate the ratio of analytic torsions with period-type integrals (Section~\ref{sect:7.3}).

\subsection
{A topological expression of the leading term}
\label{sect:7.1}
\par
In this subsection, keeping the notation in Subsection~\ref{sect:6.5}, we write ${\rm Gr}(n, TY)$ for ${\mathbf P}(TY)^{\lor}$.
To simplify the notation, we denote by $T\pi$ and $Tf$ the relative tangent bundles of $\pi$ and $f$, respectively. 
Hence $T\pi = TX/S = \ker \pi_{*} |_{X\setminus\Sigma_{\pi}}$ and $Tf = TY/T = \ker f_{*} |_{Y\setminus\Sigma_{f}}$. 
\par
Let
$$
\gamma_{Y} \colon Y\setminus\Sigma_{f} \ni y \to [T_{y}Y_{f(y)}] \in {\rm Gr}(n, TY) 
$$ 
be the Gauss map for the family $f \colon Y \to T$. Define $M$ as a resolution of the indeterminacy of $\gamma_{Y}$. 
Namely, there exists a proper birational morphism
$$
b \colon M \to Y
$$ 
and a morphism $\gamma_{M} \colon M \to {\rm Gr}(n, TY)$ such that 
$\gamma_{M} = \gamma_{Y} \circ b$ on $M \setminus b^{-1}(\Sigma_{f})$.
We set $N := (F\circ b)^{-1}(\Sigma_{\pi}) \supset b^{-1}(\Sigma_{f})$ and 
$$
M^{0} := M \setminus N = b^{-1}(Y\setminus F^{-1}(\Sigma_{\pi})).
$$
Let $\varPi \colon {\rm Gr}(n, TY) \to Y$ be the projection. We set 
$$
E := \gamma_{M}^{*}{\mathcal U}^{TY} = M \times_{{\rm Gr}(n, TY)} {\mathcal U}^{TY} \subset M \times_{{\rm Gr}(n, TY)} \varPi^{*}(TY), 
$$
where ${\mathcal U}^{TY} \subset \varPi^{*}(TY)$ is the universal bundle of rank $n$ on ${\rm Gr}(n, TY)$. 
Let $h_{\mathcal Y}$ be a K\"ahler metric on ${\mathcal Y}$. 
Then $E$ is endowed with the metric induced from $\varPi^{*}h_{\mathcal Y}|_{Y}$. We set
$$
V := (F\circ b)^{*}(TX) = M \times_{{\rm Gr}(n, TY)} \varPi^{*}F^{*}(TX),
$$
which is endowed with the Hermitian metric induced from $F^{*}(h_{\mathcal X}|_{X})$. Let
$$
D_{F} \colon TY \to F^{*}(TX),
\qquad\qquad
J_{F} \colon Tf \to F^{*}(TX)
$$
be the differential of the map $F \colon Y \to X$, where $J_{F} := D_{F}|_{Tf}$. 
Then $D_{F}|_{Y\setminus F^{-1}(\Sigma_{\pi})}$  is an isomorphism of holomorphic vector bundles from $TY$ to $F^{*}(TX)$ 
on $Y\setminus F^{-1}(\Sigma_{\pi})$, and $J_{F}$ is an injective homomorphism from $Tf$ to $F^{*}(TX)$.
We define 
$$
\psi := {\rm id}_{M} \times \varPi^{*}D_{F} |_{\nu_{M}^{*}{\mathcal U}^{TY}} \in {\rm Hom}(E, V),
$$
where
$$
{\rm id}_{M} \times \varPi^{*}D_{F} \colon M \times_{{\rm Gr}(n, TY)} \varPi^{*}(TY) \to M \times_{{\rm Gr}(n, TY)} \varPi^{*}F^{*}(TX)
$$ 
is the homomorphism induced from $D_{F}$. Since $\psi |_{M^{0}} \colon E |_{M^{0}} \to V |_{M^{0}}$ is the pull back of the homomorphism 
$J_{F} \colon Tf \to (F^{*}TX)|_{Y\setminus F^{-1}(\Sigma_{\pi})}$ by $b$, 
$\psi \colon E \to V$ is viewed as its extension. In particular, $\psi \colon E \to V$ is injective on $M^{0}$.
\par
As in Section~\ref{sect:6.5}, we define $\nu \colon M^{\natural} \to M$ as a resolution of the meromorphic map 
${\mathbf s}_{\psi} \colon M^{0} \to {\rm Gr}(n, V)$, where 
$$
{\mathbf s}_{\psi} (x) = [ J_{F}(T_{b(x)}Y_{f(b(x))}) ] \in {\rm Gr}(n, T_{F(b(x))}X)
\qquad ( x \in M^{0} ).
$$
By constructions in Section~\ref{sect:6.5}, we have the rank $n$ holomorphic Hermitian vector bundles 
$\overline{E^{\natural}}$, $\overline{W^{\natural}}$ on $M^{\natural}$ and the homomorphism 
$\psi^{\natural} \colon E^{\natural} \to W^{\natural}$. We write $J_{F}^{\natural}$, $(Tf)^{\natural}$, $(F^{*}T\pi)^{\natural}$ in place of 
$\psi^{\natural}$, $E^{\natural}$, $W^{\natural}$, respectively. 
On $\nu^{-1}(M^{0})$, $J_{F}^{\natural} \colon (Tf)^{\natural} \to (F^{*}T\pi)^{\natural}$ is identified with
$J_{F} \colon Tf|_{Y \setminus F^{-1}(\Sigma_{\pi})} \to F^{*}(T\pi)|_{Y \setminus F^{-1}(\Sigma_{\pi})}$. 
Hence $J_{F}^{\natural} \colon\overline{Tf^{\natural}} \to \overline{F^{*}(T\pi)^{\natural}}$ is viewed as an extension of 
the homomorphism of holomorphic Hermitian vector bundles
$J_{F} \colon \overline{Tf}|_{Y \setminus F^{-1}(\Sigma_{\pi})} \to F^{*}(\overline{T\pi})|_{Y \setminus F^{-1}(\Sigma_{\pi})}$, 
where $Tf$ is endowed with $h_{{\mathcal Y}/C'}$, 
the Hermitian metric on $T{\mathcal Y}/C|_{{\mathcal Y}\setminus\Sigma_{f}}$ induced from $h_{\mathcal Y}$ 
and $T\pi$ is endowed with $h_{{\mathcal X}/C}$. We make the following definitions:
\par{(i) }
${\sigma}_{J_{F}^{\natural}} \colon 
(M^{\natural} \times {\mathbf P}^{1}) \setminus ( N \times \{ \infty \} ) \to {\rm Gr}(n, (Tf)^{\natural}(1) \oplus (F^{*}T\pi)^{\natural}(1) )$ 
is defined by \eqref{eqn:taut:section}, which is a meromorphic map from 
$M^{\natural}\times{\mathbf P}^{1}$ to ${\rm Gr}(n, (Tf)^{\natural}(1)\oplus (F^{*}T\pi)^{\natural}(1))$.
\par{(ii) }
$q^{\natural} \colon {\mathfrak M}^{\natural} \to M^{\natural} \times{\mathbf P}^{1}$ is a resolution of the indeterminacy of 
$\sigma_{J_{F}^{\natural}}$.
\par{(iii) }
${\mathfrak D}_{J_{F}^{\natural}}$ is the divisor defined as $(q^{\natural})^{*}(N^{\natural} \times \{ \infty \})$.
\par{(iv) }
$\pi_{1}^{\natural} \colon M^{\natural} \times {\mathbf P}^{1} \to M^{\natural}$ is the projection.
\par{(v) }
$\widetilde{\sigma}_{J_{F}^{\natural}} \colon {\mathfrak M}^{\natural} \to {\rm Gr}(n, (Tf)^{\natural}(1) \oplus (F^{*}T\pi)^{\natural}(1))$
is the morphism satisfying $\widetilde{\sigma}_{J_{F}^{\natural}} = {\sigma}_{J_{F}^{\natural}} \circ q^{\natural}$.
\par
From Theorem~\ref{thm:pusf:forward:Bott:Chern:2}, we obtain the following:

\begin{theorem}
\label{thm:singularity:BottChern}
Let $P$ be a ${\rm GL}({\mathbf C}^{n})$-invariant polynomial in $\frak{gl}({\mathbf C}^{n})$ with respect to the adjoint action. 
Then for any $\partial$ and $\bar{\partial}$ closed smooth differential form $\chi$ on $Y$, 
$$
f_{*}\left[ \widetilde{P}( \overline{Tf}, F^{*}\overline{T\pi} ) \wedge \chi \right] 
\equiv_{\mathcal B}
- \left\{ \int_{M^{\natural}} P^{M^{\natural}}_{N^{\natural}}( J_{F}^{\natural} ) \wedge (b \circ \nu)^{*}\chi \right\} \log | t |^{2}.
$$
\end{theorem}

\begin{pf}
Since the homomorphism $J_{F}^{\natural} \colon (Tf)^{\natural} \to (F^{*}T\pi)^{\natural}$ on $\nu^{-1}(M^{0})$ is identified with 
the differential $J_{F} \colon Tf \to F^{*}(T\pi)$ on $Y \setminus F^{-1}(\Sigma_{\pi})$, we have
$$
b^{*}\widetilde{P}(\overline{Tf}, F^{*}(\overline{T\pi})) = \widetilde{P}(\overline{(Tf)^{\natural}}, \overline{(F^{*}T\pi)^{\natural}})
$$
as differential forms on $M^{0}$. The result then follows from Theorem~\ref{thm:pusf:forward:Bott:Chern:2}.
\end{pf}

\par
For a square matrix $A$, we set
$$
{\rm Td}^{\lor}(A) := {\rm Td}(A)\,e^{-c_{1}(A)} = \det\left(\frac{A}{I-e^{-A}}\right)\,e^{-{\rm Tr}(A)} = \det\left(\frac{A}{e^{A}-I}\right).
$$
Then we have ${\rm Td}^{\lor}(A) = {\rm Td}(-A)$. We define a rational number $\beta_{Y/X, \xi}$ by
\begin{equation}
\label{eqn:def:beta}
\begin{aligned}
\beta_{Y/X, \xi}
&:=
\int_{M^{\natural}} ({\rm Td}^{\lor})^{M^{\natural}}_{N^{\natural}}( J_{F}^{\natural} ) 
\wedge (F \circ b \circ \nu)^{*}{\rm ch}( \xi )
\\
&=
\int_{{\mathfrak D}_{J_{F}^{\natural}}}
\widetilde{\sigma}_{J_{F}^{\natural}}^{*}{\rm Td}^{\lor}({\mathscr U}^{\natural}) \wedge 
(F \circ b \circ \nu \circ {\pi}_{1}^{\natural} \circ {q}^{\natural})^{*} {\rm ch}( \xi ).
\end{aligned}
\end{equation}
\par
Let $\|\cdot\|_{Q,\lambda( K_{Y/T}(F^{*}\xi ))}$ be the Quillen metric on $\lambda( K_{Y/T}(F^{*}\xi ))|_{T^{o}}$ with respect to 
the metrics $h_{{\mathcal Y}/C}$ and $F^{*}h_{\xi}$. 
Similarly, let  $\|\cdot\|'_{Q,\lambda( \mu^{*}\xi )}$ be the Quillen metric on $\lambda(K_{Y/S}(F^{*}\xi))|_{T^{o}}$ with respect to 
the metrics $F^{*}h_{{\mathcal X}/C}$ and $F^{*}h_{\xi}$.

\begin{proposition}
\label{prop:comparison:quillen}
The following equality of functions on $T^{o}$ holds:
$$
\log \left( \frac{\|\cdot\|_{Q,\lambda( K_{Y/T}(F^{*}\xi) )}}{\|\cdot\|'_{Q,\lambda( K_{Y/T}(F^{*}\xi) )}} \right)^{2} 
\equiv_{\mathcal B} -\beta_{Y/X, \xi} \log|t|^{2}.
$$
In particular, $\beta_{Y/X, \xi}$ is independent of the choices of $M$ and $M^{\natural}$.
\end{proposition}

\begin{pf}
Since $\widetilde{{\rm Td}\cdot e^{-c_{1}}}(\overline{Tf}, F^{*}(\overline{T\pi})) = \widetilde{{\rm Td}^{\lor}}(\overline{Tf}, F^{*}(\overline{T\pi}))$
on $Y \setminus F^{-1}(\Sigma_{\pi})$, it follows from the anomaly formula 
\cite[Ths.\,0.2, 0.3]{BGS88} and \cite[I, Prop.\,1.3.1]{GilletSoule90} that
$$
\log\left(\frac{\|\cdot\|_{Q,\lambda( K_{Y/T}(F^{*}\xi) )}}{\|\cdot\|'_{Q,\lambda( K_{Y/T}(F^{*}\xi) )}}\right)^{2}
=
f_{*} \left[ \widetilde{\rm Td}^{\lor}(\overline{Tf}, F^{*}(\overline{T\pi})) \wedge F^{*}{\rm ch}( \overline{\xi} ) \right]^{(n,n)}.
$$
This, together with Theorem~\ref{thm:singularity:BottChern}, yields the result.
\end{pf}

Recall that the rational number $\kappa_{\pi, K_{X/S}(\xi)}$ and the non-negative integer $\varrho_{\pi, K_{X/S}(\xi)}$ 
were defined in \eqref{eqn:def:kappa} and \eqref{eqn:def:rho}, respectively, with $g=1$.

\begin{theorem}
\label{thm:leading:term:torsion}
The following equality of rational numbers holds:
\begin{equation}
\label{eqn:kappa}
\kappa_{\pi, K_{X/S}(\xi)} = \frac{1}{\deg\mu} \left( \alpha_{f, K_{Y/T}(F^{*}\xi)} + \beta_{Y/X, \xi } \right).
\end{equation}
In particular, $\kappa_{\pi, K_{X/S}(\xi)}$ is given by the integral of various characteristic classes associated to 
the semi-stable reduction of $\pi \colon X \to S$.
Moreover, the right hand side of \eqref{eqn:kappa} is independent of the choice of a semi-stable reduction of $\pi \colon X \to S$.
\end{theorem}

\begin{pf}
Let $\|\cdot\|_{L^{2},\lambda(K_{Y/T}(F^{*}\xi) )}$ be the $L^{2}$-metric on $\lambda(K_{Y/T}(F^{*}\xi))$ with respect to the metrics 
$h_{{\mathcal Y}/C'}$ and $F^{*}h_{\xi}$. Similarly, let $\|\cdot\|'_{L^{2},\lambda(K_{Y/T}(F^{*}\xi))}$ be the $L^{2}$-metric on 
$\lambda(K_{Y/T}(F^{*}\xi))$ with respect to the metrics $F^{*}h_{{\mathcal X}/C}$ and $F^{*}h_{\xi}$.
Let $\tau(Y_{t}, K_{Y_{t}}(F^{*}\xi))$ be the analytic torsion of $F^{*}\xi_{t}$ with respect to the metrics $h_{\mathcal Y}|_{Y_{t}}$ and 
$F^{*}h_{\xi}|_{Y_{t}}$.
On the other hand, the analytic torsion of $F^{*}\xi_{t}$ with respect to the metrics $F^{*}h_{{\mathcal X}/C}|_{Y_{t}}$ and 
$F^{*}h_{\xi}|_{Y_{t}}$ is given by $\tau(X_{\mu(t)}, K_{X_{\mu(t)}}(\xi))$. Hence 
\begin{equation}
\label{eqn:comparison:Quillen:1}
\log\left(\frac{\|\cdot\|_{Q,\lambda(K_{Y_{t}}(F^{*}\xi))}}{\|\cdot\|'_{Q,\lambda(K_{Y_{t}}(F^{*}\xi))}}\right)^{2}
=
\log\frac{\tau(Y_{t}, K_{Y_{t}}(F^{*}\xi))}{\tau(X_{\mu(t)}, K_{X_{\mu(t)}}(\xi))}
+
\log\left(\frac{\|\cdot\|_{L^{2},\lambda(K_{Y_{t}}(F^{*}\xi))}}{\|\cdot\|'_{L^{2},\lambda(K_{Y_{t}}(F^{*}\xi))}}\right)^{2}.
\end{equation}
Since $Y_{0}$ is reduced and normal crossing, it follows from Theorem~\ref{thm:sing:tor:log:can} that as $t \to 0$,
\begin{equation}
\label{eqn:asymptotics:torsion:1}
\log\tau(Y_{t}, K_{Y_{t}}(F^{*}\xi_{t})) = \alpha_{f, K_{Y/T}(F^{*}\xi)} \log|t|^{2} + O\left( \log\log (| t |^{-2}) \right).
\end{equation}
\par
Let $\sigma$ be a nowhere vanishing holomorphic section of $\lambda(K_{Y/T}(F^{*}\xi))$ defined on $T$. 
By Proposition~\ref{prop:L2:log:can}, as $t\to0$, we have
\begin{equation}
\label{eqn:estimate:L2:metric:1}
\log\|\sigma(t)\|_{L^{2},\lambda(K_{Y/T}(F^{*}\xi))}^{2} = O\left( \log\log(| t |^{-2}) \right).
\end{equation}
Similarly, since the degenerate K\"ahler form $\kappa_{Y}$ on $Y$ is defined as $\kappa_{Y} = F^{*}\kappa_{X} + f^{*}\kappa_{T}$, 
the $L^{2}$-metric on $\lambda(K_{Y/T}(F^{*}\xi))$ is given by $\| \cdot \|'_{L^{2},\lambda(K_{Y/T}(F^{*}\xi))}$.
By Theorem~\ref{thm:str:sing:L2} (3), as $t \to 0$, we have
\begin{equation}
\label{eqn:estimate:L2:metric:2}
\log\|\sigma(t)\|^{\prime 2}_{L^{2},\lambda(K_{Y/T}(F^{*}\xi))} = O\left( \log\log(| t |^{-2}) \right).
\end{equation}
\par
By \eqref{eqn:comparison:Quillen:1}, \eqref{eqn:asymptotics:torsion:1}, \eqref{eqn:estimate:L2:metric:1}, \eqref{eqn:estimate:L2:metric:2}
and Proposition~\ref{prop:comparison:quillen}, we get
\begin{equation}
\label{eqn:sing:torsion:2}
\begin{aligned}
\,&
\log\tau(X_{\mu(t)}, K_{X_{\mu(t)}}(\xi))
\\
&=
\log\tau(Y_{t}, K_{Y_{t}}(F^{*}\xi))
-
\log\left(\frac{\|\cdot\|_{Q,\lambda(K_{Y_{t}}(F^{*}\xi))}}{\|\cdot\|'_{Q,\lambda(K_{Y_{t}}(F^{*}\xi))}}\right)^{2}
+
\log\left(\frac{\|\cdot\|_{L^{2},\lambda(K_{Y_{t}}(F^{*}\xi))}}{\|\cdot\|'_{L^{2},\lambda(K_{Y_{t}}(F^{*}\xi))}}\right)^{2}
\\
&=
( \alpha_{f, K_{Y/T}(F^{*}\xi)} + \beta_{Y/X, \xi} ) \log | t |^{2} + O\left( \log\log(| t |^{-2}) \right).
\end{aligned}
\end{equation}
Since $t=s^{\frac{1}{\deg\mu}}$, \eqref{eqn:kappa} follows from \eqref{eqn:asymp:eq:tors} and \eqref{eqn:sing:torsion:2}. 
This completes the proof.
\end{pf}

To explore the properties of $\kappa_{\pi, K_{X/S}(\xi)}$, we consider its behavior when $\xi$ is twisted by high tensor powers of 
a positive line bundle. In the case of a fixed compact K\"ahler manifold (i.e., for a fixed parameter $s$), 
the asymptotic behavior of analytic torsion was determined by Bismut-Vasserot \cite{BismutVasserot89} and Finski \cite{Finski18}.
(See \cite{BermanBoucksom10} for the case where the curvature of the twisting ample line bundle is not positive.)

\begin{corollary}
\label{cor:dom:term:asym}
Let $H$ be a positive line bundle on ${\mathcal X}$ and set $\xi(m) := \xi\otimes H^{m}$. 
Then there exists a polynomial $p(m) \in {\mathbf Q}[m]$ with $\deg p(m) \leq \dim {\rm Sing}\,X_{0}$ such that for all $m \geq 0$,
$$
\kappa_{\pi, K_{X/S}(\xi(m))} = p(m).
$$
\end{corollary}

\begin{pf}
By \eqref{eqn:log:coeff:tw}, \eqref{eqn:def:beta}, \eqref{eqn:kappa}, 
$\deg\mu \cdot \kappa_{\pi, K_{X/S}(\xi(m))} = \alpha_{f, K_{Y/T}(F^{*}(H^{m}\otimes \xi))} + \beta_{Y/X, H^{m}\otimes  \xi }$ 
is a polynomial in the variable $m$ with rational coefficients. 
Since $c_{1}(H)^{k}=0$ on ${\rm Sing}\,X_{0}$ for $k > \dim {\rm Sing}\,X_{0}$, we get
$\deg \alpha_{f, K_{Y/T}(F^{*}(H^{m}\otimes \xi))} \leq \dim {\rm Sing}\,X_{0}$ by \eqref{eqn:log:coeff:tw} and
$\deg \beta_{Y/X, H^{m}\otimes  \xi } \leq \dim {\rm Sing}\,X_{0}$ by \eqref{eqn:def:beta}, 
where $\deg(\cdot)$ denotes the degree in the variable $m$.
This proves $\deg p(m) \leq \dim {\rm Sing}\,X_{0}$.
\end{pf}

\begin{example}
\label{ex:quad:sing}
Consider the same situation as in Section~\ref{sect:5.3.4}. Hence, $\pi(z)$ is locally defined by the equation $\pi(z)=z_{0}z_{1}$
in a suitable local coordinates near the critical locus $\Sigma = \Sigma_{\pi}$ of $\pi$. Then $\Sigma$ is a compact complex submanifold 
of $X$ of dimension $n-1$. By Proposition~\ref{prop:kappa:q:rk2}, $\kappa_{\pi, K_{X/S}(\xi(m))} $ is a polynomial in the variable $m$ 
of degree $n-1$ given by
$$
\begin{aligned}
\kappa_{\pi, K_{X/S}(\xi(m))} 
&= 
-\frac{1}{2} \int_{\Sigma} {\rm Td}(T\Sigma) {\rm E}(N_{\Sigma/X}) {\rm ch}(K_{\Sigma}(\xi(m)))
\\
&=
-\left( \frac{{\rm E}(0)r(\xi)}{(n-1)!}\int_{\Sigma}c_{1}(H)^{n-1} \right) m^{n-1} + O(m^{n-2})
\end{aligned}
$$
with ${\rm E}(0)=1/6$. Suppose that $\dim \Sigma =1$ and $\Sigma$ is connected. Let $g$ be the genus of $\Sigma$. 
Since ${\rm E}(N_{\Sigma/X})$ is a polynomial in $c_{2}(N_{\Sigma/X})$, we get
$$
2\kappa_{\pi, K_{X/S}(\xi(m))} = -{\rm E}(0) \chi(\Sigma, K_{\Sigma}(\xi(m))) = -{\rm E}(0)\{ \deg_{\Sigma} (\xi(m)|_{\Sigma}) + r(\xi)(g-1) \},
$$
where $\chi(\Sigma, K_{\Sigma}(\xi(m))) \geq 0$ for $m>0$ by the Nakano vanishing theorem.
In particular, $2\kappa_{\pi, K_{X/S}} = -{\rm E}(0)(g-1)$, whose sign is determined by $g$. 
It is remarkable that $\kappa_{\pi, K_{X/S}} = 0$ if $\Sigma$ consists of elliptic curves.
\end{example}

\begin{example}
Suppose that $\dim {\rm Sing}\,X_{0}=0$.
In contrast to the previous example, in \cite{ErikssonFreixas26}, Eriksson and Freixas i Montplet proved $-\kappa_{\pi,K_{X/S}} > 0$ 
when $n=1$. In the case $n\geq 2$, they conjecture $-\kappa_{\pi,K_{X/S}} > 0$.
See \cite[Conjecture, Theorems A and B]{ErikssonFreixas26} for their conjecture and various supporting evidences.
\end{example}

\begin{problem}
In \cite{BermanBoucksom10, BismutVasserot89, Finski18}, the behavior of the function $\log \tau(X_{s}, K_{X_{s}}(\xi(m)))$ on 
$S^{o} \times {\mathbf Z}_{\geq0}$ as $m \to \infty$ was studied for a fixed $s \in S^{o}$. 
In Theorems~\ref{thm:main:1} and~\ref{thm:leading:term:torsion}, the behavior as $s \to 0$ is studied for a fixed $m \in {\mathbf Z}_{\geq0}$.
In Corollary~\ref{cor:dom:term:asym}, the behavior of $\log \tau(X_{s}, K_{X_{s}}(\xi(m))) / \log |s|^{2}$ as $s\to0$ is studied
for a fixed $m \in {\mathbf Z}_{\geq0}$. 
Describe the behavior of $\log \tau(X_{s}, K_{X_{s}}(\xi(m)))$ as a function of two variables on $S^{o} \times {\mathbf Z}_{\geq0}$.
\end{problem}

Recall that $\delta_{\pi, K_{X/S}(\xi)} \in {\mathbf Q}$ was defined by \eqref{eqn:elem:exp}. By \eqref{eqn:asym:eq:L2}, 
the leading term of the singularity of the $L^{2}$-metric on $\lambda(K_{X/S}(\xi))$ with respect to $h_{{\mathcal X}/C}$, 
$h_{\xi}$ is given by $-\delta_{\pi, K_{X/S}(\xi)}$ in the sense that as $s \to 0$,
$$
\log \| \sigma(s) \|_{\lambda(K_{X/S}(\xi)), L^{2}}^{2} 
= -\delta_{\pi, K_{X/S}(\xi)} \log |s|^{2} + \varrho_{\pi, K_{X/S}(\xi)} \log\log (|s|^{-2}) + c + O\left(1/\log |s|^{-1}\right).
$$
By Theorem~\ref{thm:leading:term:torsion}, $\delta_{\pi, K_{X/S}(\xi)}$ is expressed as the integral of characteristic classes associated 
to the semi-stable reduction of $\pi \colon X \to S$.

\begin{corollary}
\label{cor:sing:L2}
The following identity of rational numbers holds:
$$
\delta_{\pi, K_{X/S}(\xi)} = \frac{1}{\deg\mu} \left( \alpha_{f, K_{Y/T}(F^{*}\xi)} + \beta_{Y/X, \xi } \right) - \alpha_{\pi, K_{X/S}(\xi)}.
$$
In particular, $\delta_{\pi, K_{X/S}(\xi)}$ is given by the integral of various characteristic classes associated to the semi-stable reduction 
of $\pi \colon X \to S$ and to the resolution of the indeterminacy of the Gauss map of $\pi$. 
Consequently, there exists a polynomial $q(m) \in {\mathbf Q}[m]$ with $\deg q(m) \leq \dim {\rm Sing}\,X_{0}$ such that for $m \geq 0$,
$$
\delta_{\pi, K_{X/S}(\xi(m))} = q(m).
$$
\end{corollary}

\begin{pf}
By \eqref{eqn:def:kappa} and \eqref{eqn:asymp:eq:tors}, we have 
$\kappa_{\pi, K_{X/S}(\xi)} = \alpha_{\pi, K_{X/S}(\xi)} + \delta_{\pi, K_{X/S}(\xi)}$. 
Comparing this with \eqref{eqn:kappa}, we get the first statement. 
The second statement follows from the same argument as in the proof of Corollary~\ref{cor:dom:term:asym}.
\end{pf}

Concerning the subdominant term of the asymptotic expansion \eqref{eqn:asymp:eq:tors} with $g=1$, 
$\varrho_{\pi, K_{X/S}(m)}$ is independent of the choice of an ample line bundle $H$ on $X$ and $m>0$ under the assumption 
that $n=1$ and $X_{0}$ is reduced. In fact, in this case, $\varrho_{\pi, K_{X/S}(m)}-\varrho_{\pi, K_{X/S}} = N-1$, 
where $N$ is the number of the irreducible components of $X_{0}$ by \cite[Th.\,5.5]{DaiYoshikawa25}. 
In higher dimensions, we propose the following:

\begin{conjecture}
\label{conj:indep:rho}
If ${\rm Sing}\,X_{0}$ is isolated, $\xi={\mathcal O}_{X}$ and $n>1$, then $\varrho_{\pi, K_{X/S}(m)}$ is independent of the choice of 
an ample line bundle $H$ and $m>0$. Moreover, 
$$
\varrho_{\pi, K_{X/S}(m)}-\varrho_{\pi, K_{X/S}} = \dim \ker \square_{X_{0}}^{(0, n-1)} - h^{0,n-1}(X_{s}), 
\qquad
s\not=0.
$$ 
Here $\square_{X_{0}}^{(0,n-1)}$ is the Friedrichs extension of the Hodge-Kodaira Laplacian acting on the compactly supported 
smooth $(0,n-1)$-forms on $X_{0}\setminus{\rm Sing}\,X_{0}$.
\end{conjecture}

\subsection
{A locality of the leading term}
\label{sect:7.2}
\par
In this subsection, we prove the locality of the leading term in the asymptotic expansion of analytic torsion.
Let $\psi \colon {\mathcal W} \to B$ be a surjective holomorphic map with connected fibers from a connected projective manifold 
of dimension $n+1$ to a compact Riemann surface. Let $\Sigma_{\psi} \subset {\mathcal W}$ be the critical locus of $\psi$.  
Let $b_{0} \in \psi(\Sigma_{\psi}) \subset B$ be a point of the discriminant locus of $\psi$. We assume that there is an embedding
$(S,0) \hookrightarrow (B,b_{0})$ and we set $W := \psi^{-1}(S)$. For $s\in S$, we set $W_{s} := \psi^{-1}(s)$. 
Let $(\eta, h_{\eta})$ be a holomorphic Hermitian vector bundle on ${\mathcal W}$ such that $(\eta, h_{\eta})|_{W}$ is Nakano semi-positive.
Let $h_{W}$ be a K\"ahler metric on $W$ and let $h_{W_{s}}$ be the K\"ahler metric on $W_{s}$ induced from $h_{W}$.
Let $\tau(W_{s}, K_{W_{s}}(\eta_{s}))$ be the analytic torsion of $(W_{s}, K_{W_{s}}(\eta_{s}))$ with respect to the metrics $h_{W_{s}}$,
$h_{\eta}|_{W_{s}}$.

\begin{theorem}
\label{thm:locality:sing:tors}
With the notation and assumption as above, let $U$ be an open neighborhood of ${\rm Sing}\,X_{0}$ in ${\mathcal X}$ and 
let $V$ be an open neighborhood of ${\rm Sing}\,W_{0}$ in ${\mathcal W}$. Suppose that the following hold:
\par{\rm (1)}
There exists an isomorphism $\iota \colon U \to V$ 
such that $\iota^{*}(\psi|_{V}) = \pi|_{U}$ and $\iota^{*}(\eta|_{V}) \cong \xi|_{U}$ as smooth complex vector bundles on $U$.  
\par{\rm (2)}
There exist log-resolutions $b \colon (X',X'_{0}) \to (X, X_{0})$ and $c \colon (W', W'_{0})\to (W, W_{0})$ such that the isomorphism
$\iota \colon U \to V$ lifts to an isomorphism between $b^{-1}(U)$ and $c^{-1}(V)$, 
such that $X'\setminus X'_{0} \subset X'$ and $W'\setminus W'_{0} \subset W'$ are toroidal embeddings without self-intersection, 
and such that the conical polyhedral complexes with integral structure associated to 
$X'\setminus X'_{0} \subset X'$ and $W'\setminus W'_{0} \subset W'$ coincide.
(See \cite[p.71]{Mumford73} for the notion of the conical polyhedral complex with integral structure associated to a toroidal embedding 
without self-intersection.)
\par
Then $\kappa_{\pi, K_{X/S}(\xi)} = \kappa_{\psi, K_{W/S}(\eta)}$.
In particular, there exist $\varrho_{\xi,\eta} \in {\mathbf Z}$ and $c_{\xi,\eta} \in {\mathbf R}$ such that as $s \to 0$,
$$
\log \frac{ \tau(X_{s}, K_{X_{s}}(\xi_{s})) }{ \tau(W_{s}, K_{W_{s}}(\eta_{s})) } 
= \varrho_{\xi,\eta} \log\log(|s|^{-2}) + c_{\xi,\eta} + O(1/\log |s|^{-1}).
$$
\end{theorem}

\begin{pf}
Since $\iota^{*}(\psi|_{V}) = \pi|_{U}$ by (1) and the conical polyhedral complexes with integral structure associated to 
$X'\setminus X'_{0} \subset X'$ and $W'\setminus W'_{0} \subset W'$ are isomorphic by (2), 
we obtain a semi-stable reduction of $\psi \colon (W, W_{0}) \to (S,0)$ by the same procedure 
as the construction of the semi-stable reduction $f \colon Y \to T$ of $\pi \colon X \to S$ (\cite[Chap.\,II Sect.\,3]{Mumford73}). 
The semi-stable reduction of $\psi \colon (W, W_{0}) \to (S,0)$ constructed in this way is denoted by $g \colon M \to T$. 
Let $G \colon M \to W$ be the corresponding morphism.
Since $\iota^{*}(\eta|_{V}) \cong \xi|_{U}$ as smooth vector bundles on a neighborhood of ${\rm Sing}\,X_{0}$, 
we obtain $\alpha_{f, K_{Y/T}(F^{*}\xi)} = \alpha_{g, K_{M/T}(G^{*}\eta)}$ and $\beta_{Y/X, \xi} = \beta_{M/W, \eta}$ by 
\eqref{eqn:log:coeff:tw} and \eqref{eqn:def:beta}, respectively. The result follows from Theorem~\ref{thm:leading:term:torsion}.
\end{pf}

\begin{remark}
If $X_{0}$ and $W_{0}$ are irreducible, (2) can be dropped in Theorem~\ref{thm:locality:sing:tors}. 
\end{remark}

Recall that the semi-simple part of the monodromy operator $M_{s}$ acts on the limiting Hodge filtration $F^{n}_{\infty}H^{n+q}(X_{\infty})$ 
and that $\log M_{s}$ is its logarithm with imaginary part lying in $[0, 2\pi)$. To simplify the notation, we set
\begin{equation}
\label{eqn:lg:monodromy}
\epsilon_{\pi} := \frac{1}{2\pi i} \sum_{q\geq 0} (-1)^{q} {\rm Tr} \left[ \left. \log M_{s} \right|_{F^{n}_{\infty}H^{n+q}(X_{\infty})} \right].
\end{equation}
By Proposition~\ref{prop:EFM21}, we have
\begin{equation}
\label{eqn:}
\epsilon_{\pi} = \delta_{\pi, K_{X/S}}.
\end{equation}
When ${\rm Sing}\,X_{0}$ consists of isolated points, recall that $\mu(\pi_{x})$ and $\widetilde{p}_{g}(\pi_{x})$, $x \in {\rm Sing}\,X_{0}$,
denote the Milnor number and the spectral genus of $\pi_{x} \in {\mathcal O}_{X,x}$, respectively.

\begin{theorem}
\label{thm:sing:tors:N:pos}
Let $(\xi, h_{\xi})$ be a holomorphic Hermitian vector bundle of rank $r(\xi)$ on ${\mathcal X}$ such that 
$(\xi, h_{\xi})|_{X}$ is Nakano semi-positive. Then the following hold.
\par{\rm (1)}
If $\xi$ is topologically trivial near ${\rm Sing}\,X_{0}$, then 
$$
\kappa_{\pi, K_{X/S}(\xi)} = r(\xi)( \alpha_{\pi, K_{X/S}} + \epsilon_{\pi} ).
$$
\par{\rm (2)}
If $\dim {\rm Sing}\,X_{0} = 0$, then 
$$
\kappa_{\pi, K_{X/S}(\xi)} = -r(\xi) \sum_{x\in {\rm Sing}\,X_{0}} \left( \frac{\mu(\pi_{x})}{(n+2)!} - \widetilde{p}_{g}(\pi_{x}) \right).
$$
\end{theorem}

\begin{pf}
Setting ${\mathcal W} = {\mathcal X}$ and $\eta = {\mathcal O}_{\mathcal X}^{r(\xi)}$ in Theorem~\ref{thm:locality:sing:tors}, 
we obtain (1) from Corollary~\ref{cor:sing:tors:triv} and (2) from Proposition~\ref{prop:EF24}.
\end{pf}

The locality of $\kappa_{\pi, K_{X/S}(\xi)}$ entails the corresponding locality of $\delta_{\pi, K_{X/S}(\xi)}$.

\begin{theorem}
\label{thm:sing:L2:N:pos:gen}
With the same notation and assumption as in Theorem~\ref{thm:locality:sing:tors}, 
one has $\delta_{\pi, K_{X/S}(\xi)} = \delta_{\psi, K_{W/S}(\eta)}$. In particular,
if $\sigma$ and $\sigma'$ are nowhere vanishing holomorphic sections on $S$ of $\lambda(K_{X/S}(\xi))$ and
$\lambda(K_{W/S}(\eta))$ respectively, then there exists $c'_{\xi,\eta} \in {\mathbf R}$ such that as $s \to 0$,
$$
\log ( \| \sigma(s) \|_{L^{2}}^{2} / \| \sigma'(s) \|_{L^{2}}^{2} ) =  -\varrho_{\xi,\eta} \log\log(|s|^{-2}) + c'_{\xi,\eta} + O(1/\log |s|^{-1}),
$$
where $\varrho_{\xi,\eta} \in {\mathbf Z}$ is the same integer as in Theorem~\ref{thm:locality:sing:tors}. 
\end{theorem}

\begin{pf}
The second statement follows from Theorems~\ref{Thm:Sing:Q:adj} and~\ref{thm:locality:sing:tors}. 
The first statement follows from the second one.
\end{pf}

\begin{theorem}
\label{thm:sing:L2:N:pos}
With the same notation and assumption as in Theorem~\ref{thm:sing:tors:N:pos},
let $\sigma$ be a nowhere vanishing holomorphic section of $\lambda(K_{X/S}(\xi))$ defined on $S$.  
\par{\rm (1)}
If $\xi$ is topologically trivial near ${\rm Sing}\,X_{0}$, then there exists $c'_{\xi} \in {\mathbf R}$  such that as $s \to 0$,
$$
\log \| \sigma(s) \|_{L^{2}}^{2} = -r(\xi) \epsilon_{\pi} \log |s|^{2} + \varrho_{\pi,K_{X/S}(\xi)}\log\log(|s|^{-2}) + c'_{\xi} + O(1/\log |s|^{-1}).
$$
\par{\rm (2)}
If $\dim {\rm Sing}\,X_{0} = 0$, then there exists $c'_{\xi} \in {\mathbf R}$  such that as $s \to 0$,
$$
\log  \| \sigma(s) \|_{L^{2}}^{2} 
= -r(\xi)  \widetilde{p}_{g}({\rm Sing}\,X_{0}) \log |s|^{2} + \varrho_{\pi, K_{X/S}(\xi)} \log\log (|s|^{-2}) + c'_{\xi} + O(1/\log |s|^{-1}).
$$
\end{theorem}

\begin{pf}
The result follows from Theorems~\ref{Thm:Sing:Q:adj} and \ref{thm:sing:tors:N:pos}.
\end{pf}

\begin{question}
In the situation of Theorems~\ref{thm:locality:sing:tors} and~\ref{thm:sing:L2:N:pos}, 
do the elementary exponents of $R^{q}\psi_{*}K_{W/S}(\eta)$ coincide with those of $R^{q}\pi_{*}K_{X/S}(\xi)$ for all $q \geq 0$?
\end{question}

\subsection
{Analytic torsion and the determinant of the period-type integrals}
\label{sect:7.3}
\par
In this subsection, as an application of Theorems~\ref{thm:sing:tors:N:pos} (2) and \ref{thm:sing:L2:N:pos} (2), 
we relate the ratio of analytic torsions to the ratio of the determinants of period-type integrals.
\par

\begin{theorem}
\label{thm:McKean:Singer:prod}
Let $H$ be an ample line bundle on ${\mathcal X}$ endowed with a Hermitian metric of semi-positive curvature. 
Let $\{ \varphi_{1}, \ldots, \varphi_{m_{1}} \}$, $m_{1}=h^{0}(K_{X_{s}}(H))$, be a basis of $\pi_{*}K_{X/S}(H)$ 
as a free ${\mathcal O}_{S}$-module. 
Similarly, let $\{ \omega_{1}, \ldots, \omega_{m_{2}} \}$, $m_{2} = h^{n,0}(X_{s})$, be a basis of $\pi_{*}K_{X/S}$ 
as a free ${\mathcal O}_{S}$-module. 
Suppose that $X_{0}$ has only isolated singularities. Then there exist constants $C,C'>0$ such that as $s \to 0$,
\begin{equation}
\label{eqn:McKean:Singer:prod}
\frac{ \tau(X_{s},K_{X_{s}}) }{ \tau(X_{s},K_{X_{s}}(H)) } 
\sim 
C \frac{\displaystyle \det\left( (\sqrt{-1})^{n^{2}} \int_{X_{s}} h_{H}(\varphi_{\alpha}(s) \wedge \overline{\varphi_{\beta}(s)} ) \right) }
{\displaystyle \det\left( (\sqrt{-1})^{n^{2}} \int_{X_{s}} \omega_{i}(s) \wedge \overline{\omega_{j}(s)} \right) }
\sim
C' (\log |s|^{-1})^{\varrho},
\end{equation}
where $\varrho := \varrho_{\pi, K_{X/S}(H)} - \varrho_{\pi, K_{X/S}}$ and $A(s) \sim B(s)$ if $\lim_{s\to0} A(s)/B(s) =1$.
Moreover, if the linear system $|H|$ contains a smooth member avoiding $\Sigma_{\pi}$ and intersecting $X_{0}$ transversally,
then $\varrho \geq 0$.
\end{theorem}

\begin{pf}
We keep the notation of Theorems~\ref{thm:sing:tors:N:pos} and \ref{thm:sing:L2:N:pos}. Set $\xi = {\mathcal O}_{X}$ and $\xi' = H$.
Since $H$ is ample, $\sigma' = \varphi_{1} \wedge\cdots\wedge \varphi_{m_{2}}$ by the Kodaira vanishing theorem. Hence
\begin{equation}
\label{eqn:det:period:H}
\| \sigma'(s) \|_{L^{2}}^{2} = \det\left( (\sqrt{-1})^{n^{2}} \int_{X_{s}} h_{H}(\varphi_{\alpha}(s) \wedge \overline{\varphi_{\beta}(s)} ) \right)
\qquad
(s \in S^{o}).
\end{equation}
\par
Since $X_{0}$ has only isolated singularities, the monodromy action on $H^{q}(X_{s}, {\mathbf C})$ $(s\in S^{o})$ is trivial for $0\leq q<n$. 
Since $(R^{q}\pi_{*}K_{X/S}, h_{L^{2}})$ is the dual of the Hermitian vector bundle $(R^{n-q}\pi_{*}{\mathcal O}_{X}, h_{L^{2}})$ 
by the Serre duality, we deduce from the triviality of the monodromy action on $H^{q}(X_{s}, {\mathbf C})$ $(0\leq q<n)$ 
and \cite[Th.\,4.4]{EFM21} that the $L^{2}$-metric on $\det R^{q}\pi_{*}K_{X/S}$ defined on $S^{o}$ extends to 
a non-degenerate continuous metric on $\det R^{q}\pi_{*}K_{X/S}$ defined on $S$ for $q>0$. 
Hence there exists a continuous function $g(s)$ on $S$ such that
\begin{equation}
\label{eqn:det:period:O}
\| \sigma(s) \|_{L^{2}}^{2} = e^{g(s)} \det\left( (\sqrt{-1})^{n^{2}} \int_{X_{s}} \omega_{i}(s) \wedge \overline{\omega_{j}(s)} \right)
\qquad
(s \in S^{o}).
\end{equation}
From Theorems~\ref{thm:sing:tors:N:pos} (2) and \ref{thm:sing:L2:N:pos} and \eqref{eqn:det:period:H}, \eqref{eqn:det:period:O}, 
we obtain \eqref{eqn:McKean:Singer:prod}.
\par
We prove the non-negativity of $\varrho$. Let $D \in |H|$ be a smooth hypersurface in $X$ with $D \cap \Sigma_{\pi}=\emptyset$ 
and intersecting $X_{0}$ transversally.
Let $f \colon Y \to T$ be the semi-stable reduction of $\pi \colon X \to S$ as in \eqref{eqn:ss:red}. 
Then $F^{-1}(D)$ is a smooth hypersurface in $Y$ with $F^{-1}(D)\cap F^{-1}(\Sigma_{\pi}) = \emptyset$ and
intersecting $Y_{0}$ transversally.
Let $\{ \widetilde{\varphi}_{1}, \ldots, \widetilde{\varphi}_{m_{1}} \}$ and $\{ \widetilde{\omega}_{1}, \ldots, \widetilde{\omega}_{m_{2}} \}$
be bases of $f_{*}K_{Y/T}(F^{*}H)$ and $f_{*}K_{Y/T}$ as free ${\mathcal O}_{T}$-modules, respectively.
To see that we can assume 
$\{ \widetilde{\omega}_{1}, \ldots, \widetilde{\omega}_{m_{2}} \} \subset \{ \widetilde{\varphi}_{1}, \ldots, \widetilde{\varphi}_{m_{1}} \}$, 
consider the exact sequence $0 \to f_{*}K_{Y/T} \to f_{*}K_{Y/T}(F^{*}H) \to f_{*}(K_{D/T})$
obtained from the Poincar\'e residue sequence for $F^{-1}(D) \subset Y$. 
Since $f_{*}K_{Y/T}$, $f_{*}K_{Y/T}(F^{*}H)$, and $f_{*}(K_{D/T})$ are free ${\mathcal O}_{T}$-modules,
$f_{*}K_{Y/T}$ is a direct summand of $f_{*}K_{Y/T}(F^{*}H)$. Consequently, we can assume 
$\{ \widetilde{\omega}_{1}, \ldots, \widetilde{\omega}_{m_{2}} \} \subset \{ \widetilde{\varphi}_{1}, \ldots, \widetilde{\varphi}_{m_{1}} \}$. 
\par
By Theorem~\ref{thm:str:sing:L2} (3), we have
\begin{equation}
\label{eqn:loglog:0}
\varrho_{f,K_{Y/T}(F^{*}H)} = \varrho_{f,K_{X/S}(H)},  \qquad  \varrho_{f,K_{Y/T}} = \varrho_{f,K_{X/S}}.
\end{equation}
By \cite[Lemma~5.2]{Takayama22}, there is a basis $\{ \psi_{1}, \ldots, \psi_{m_{1}} \}$ of $f_{*}K_{Y/T}(F^{*}H)$ satisfying
${\rm ord}(\| \psi_{1} \wedge\cdots\wedge \psi_{m_1}\|_{L^{2}}^{2}) = {\rm ord}(\prod_{k=1}^{m_{1}} \| \psi_{k} \|_{L^{2}}^{2})$ 
and ${\rm ord} (\| \psi_{k} \|_{L^{2}}^{2}) \geq ({\rm ord} \| \psi_{l} \|_{L^{2}}^{2}) \geq 0$ whenever $k < l$. 
Here for $f(t) = \sum_{l=0}^{n} a_{l}(t) (-\log |t|^{2}) ^{l}$ with $a_{l}(t)\in C^{\infty}(T)$ and $(a_{0}(0),\ldots,a_{n}(0))\not=(0,\ldots,0)$, 
we set ${\rm ord} (f) := \max\{ 0\leq l \leq n;\, a_{l}(0) \not= 0\}$. Then
\begin{equation}
\label{eqn:loglog:1}
\varrho_{f,K_{Y/T}(F^{*}H)} = \sum_{k=1}^{m_{1}} {\rm ord} (\| \psi_{k} \|_{L^{2}}^{2}).
\end{equation}
Express $\omega_{k} = \sum_{l=1}^{m_{1}} c_{kl}(t) \psi_{l}$, $1\leq k \leq m_{1}$, 
where the functions $c_{kl}(t) \in {\mathcal O}(T)$ satisfy ${\rm rank}( c_{kl}(0) ) = m_{2}$. 
We can find a basis $\{ \widetilde{\omega}'_{1}, \ldots, \widetilde{\omega}'_{m_{2}} \}$ of $f_{*}K_{Y/T}$ 
satisfying $\widetilde{\omega}'_{k} \equiv \psi_{i_{k}} + \sum_{j>i_{k}} b_{i_{k},j} \psi_{j} \mod t{\mathcal O}(T)$ 
and $i_{1}<i_{2}<\cdots<i_{m_{2}}$, where the $b_{i_{k},j}$ are constants. 
Since ${\rm ord}(  \| \widetilde{\omega}'_{k} \|_{L^{2}}^{2}) \leq {\rm ord}(  \| \psi_{i_k} \|_{L^{2}}^{2})$
and $\| \widetilde{\omega}'_{1} \wedge\cdots\wedge \widetilde{\omega}'_{m_{2}} \|_{L^{2}}^{2} \leq 
\prod_{k=1}^{m_{2}} \| \widetilde{\omega}'_{k} \|_{L^{2}}^{2}$, 
we obtain
\begin{equation}
\label{eqn:loglog:2}
{\rm ord} (\| \widetilde{\omega}'_{1} \wedge\cdots\wedge \widetilde{\omega}'_{m_{2}} \|_{L^{2}}^{2} )
\leq
\sum_{k=1}^{m_{2}} {\rm ord}(  \| \widetilde{\omega}'_{k} \|_{L^{2}}^{2}) 
\leq
\sum_{k=1}^{m_{2}} {\rm ord}(  \| \psi_{i_k} \|_{L^{2}}^{2}) 
\leq 
\sum_{l=1}^{m_{1}} {\rm ord}(  \| \psi_{l} \|_{L^{2}}^{2}). 
\end{equation}
Let $\| \cdot \|'_{L^{2}}$ be the $L^{2}$-norm on $f_{*}K_{Y/T}$. By \cite[Lemma~5.2]{Takayama22} again, we can assume 
\begin{equation}
\label{eqn:loglog:3}
\varrho_{f,K_{Y/T}} 
=
{\rm ord}(\| \widetilde{\omega}_{1} \wedge\cdots\wedge \widetilde{\omega}_{m_{2}} \|_{L^{2}}^{\prime 2})
= 
\sum_{k=1}^{m_{2}} {\rm ord}(\| \widetilde{\omega}_{k}\|_{L^{2}}^{\prime 2}).
\end{equation}
Since ${\rm Sing}\,Y_{0} \subset F^{-1}(\Sigma_{\pi})$ and hence $F^{-1}(D) \cap {\rm Sing}\,Y_{0} = \emptyset$, 
it follows from \cite[Lemma\,3.4]{Takayama22} that
\begin{equation}
\label{eqn:loglog:4}
{\rm ord}(\| \widetilde{\omega}_{k}\|_{L^{2}}^{\prime 2}) = {\rm ord}(\| \widetilde{\omega}_{k}\|_{L^{2}}^{2}).
\end{equation}
By \eqref{eqn:loglog:0}, \eqref{eqn:loglog:1}, \eqref{eqn:loglog:3}, \eqref{eqn:loglog:3}, \eqref{eqn:loglog:4}, 
we conclue $\varrho_{\pi,K_{X/S}} \leq \varrho_{\pi,K_{X/S}(H)}$. 
\end{pf}

\begin{remark}
We note that the second equality of \cite[Conjecture~9.6]{DaiYoshikawa25} follows from \eqref{eqn:McKean:Singer:prod} and
Theorems~\ref{thm:str:sing:L2} (3) and \ref{thm:sing:L2:N:pos} (2). Conjecture~\ref{conj:indep:rho} says that $\varrho$ is independent of $H$. 
For a conjectural interpretation of the left hand side of \eqref{eqn:McKean:Singer:prod} in terms of the small eigenvalues of 
the Hodge-Kodaira Laplacian of $X_{s}$, see \cite[Conjecture~9.6 ]{DaiYoshikawa25}.
\end{remark}

\section
{The boundary behavior of the curvature of $L^{2}$-metrics}
\label{sect:8}
\par
We keep the notation and assumption in Section~\ref{sect:4}.
Let ${\mathcal R}(s)\,ds\wedge d\bar{s}$ be the curvature form of $(R^{q}\pi_{*}K_{X/S}(\xi), h_{L^{2}})$. 
By Berndtsson \cite{Berndtsson09}, Mourougane-Takayama \cite{MourouganeTakayama08}, $i{\mathcal R}(s)$ is semi-positive. 
In this section, as a by-product of Theorem~\ref{thm:str:sing:L2}, we determine the asymptotic behavior of the first Chern form 
$c_{1}(R^{q}\pi_{*}K_{X/S}(\xi),h_{L^{2}})$ as $s \to 0$. In particular, we prove the Poincar\'e boundedness of the curvature 
$i{\mathcal R}(s)$.
In what follows, we write $C^{\infty}_{\bf R}(T)$ for the set of {\em real-valued} $C^{\infty}$ functions on $T$.


\begin{lemma}
\label{lemma:technical}
Let $I$ be a finite set. For $i\in I$, let $\ell_{i}\in{\mathbf Z}$ and $\varphi_{i}(t)\in C^{\infty}_{\mathbf R}(T)$. 
Set $g_{i}(t) := (\log|t|^{2})^{\ell_{i}}\varphi_{i}(t)$ for $i\in I$ and $g(t) := \sum_{i\in I} g_{i}(t)$. Suppose that $g(t)>0$ on $T^{o}$.
Then the following identity of functions on $T^{o}$ hold:
\begin{equation}
\label{eqn:dd:log:g}
\begin{aligned}
\partial_{t\bar{t}}\log g
&=
-\sum_{i}\frac{\ell_{i}}{|t|^{2}(\log|t|^{2})^{2}} \frac{g_{i}}{g}
+
\frac{1}{2}\sum_{i,j}  \left| \frac{\ell_{i}-\ell_{j}}{t(\log|t|^{2})}  \right|^{2} \frac{g_{i}g_{j}}{g^{2}}
+
\sum_{i} \frac{\partial_{t\bar{t}}\varphi_{i}}{\varphi_{i}} \frac{g_{i}}{g}
\\
&\quad
-
\sum_{i} \left| \frac{\partial_{t}\varphi_{i}}{\varphi_{i}} \frac{g_{i}}{g} \right|^{2}
+
\sum_{i,j} {\rm Re}\left\{ \left( \frac{\ell_{i}-\ell_{j}}{t(\log|t|^{2})}\right) \cdot
\left(\overline{\frac{\partial_{t}\varphi_{i}}{\varphi_{i}}-\frac{\partial_{t}\varphi_{j}}{\varphi_{j}}}\right) \right\} \frac{g_{i}g_{j}}{g^{2}}.
\end{aligned}
\end{equation}
In particular, if $0\leq \ell_{i}\leq N$ for all $i\in I$, then as $t\to0$
\begin{equation}
\label{eqn:d:log:g:asymp}
\begin{aligned}
\partial_{t\bar{t}}\log g
&=
-\sum_{i}\frac{\ell_{i}}{|t|^{2}(\log|t|^{2})^{2}} \frac{g_{i}}{g}
+
\frac{1}{2}\sum_{i,j}  \left| \frac{\ell_{i}-\ell_{j}}{t(\log|t|^{2})} \right|^{2} \frac{g_{i}g_{j}}{g^{2}}
\\
&\qquad
+
O\left( \frac{(\log |t|^{-1})^{2N}}{|t| g(t)^{2}} + \frac{(\log |t|^{-1})^{N}}{g(t)} \right).
\end{aligned}
\end{equation}
\end{lemma}

\begin{pf}
Since $\partial_{t} g_{i} = \left( \frac{\ell_{i}}{t(\log|t|^{2})}+\frac{\partial_{t}\varphi_{i}}{\varphi_{i}} \right) g_{i}$, we get 
\begin{equation}
\label{eqn:comp:met:1}
\begin{aligned}
\partial_{t\bar{t}} g_{i}(t)
&=
\left( -\frac{\ell_{i}}{|t|^{2}(\log|t|^{2})^{2}} 
+ \frac{\partial_{t\bar{t}}\varphi_{i}}{\varphi_{i}} -\frac{|\partial_{t}\varphi_{i}|^{2}}{\varphi_{i}^{2}}
+
\left| \frac{\ell_{i}}{t(\log|t|^{2})} + \frac{\partial_{t}\varphi_{i}}{\varphi_{i}} \right|^{2} \right) g_{i}
\\
&=
\left( 
-\frac{\ell_{i}}{|t|^{2}(\log|t|^{2})^{2}} + \frac{\partial_{t\bar{t}}\varphi_{i}}{\varphi_{i}}
+
\left| \frac{\ell_{i}}{t(\log|t|^{2})} \right|^{2} 
+
2{\rm Re}\left( \frac{\ell_{i}}{t(\log|t|^{2})} \right) \overline{\frac{\partial_{t}\varphi_{i}}{\varphi_{i}}}
\right) g_{i}.
\end{aligned}
\end{equation}
Since
\begin{equation}
\label{eqn:comp:met:2}
\begin{aligned}
|\partial_{t}g|^{2}
&=
\sum_{i,j\in I}
\left(
\frac{\ell_{i}}{t(\log|t|^{2})} + \frac{\partial_{t}\varphi_{i}}{\varphi_{i}}
\right)
\left(
\overline{\frac{\ell_{j}}{t(\log|t|^{2})}+\frac{\partial_{t}\varphi_{j}}{\varphi_{j}}}
\right)
g_{i}g_{j}
\\
&=
\left| \sum_{i} \left( \frac{\ell_{i}}{t(\log|t|^{2})} \right) g_{i} \right|^{2}
+
2 \sum_{i,j} {\rm Re} \left\{ \left(\frac{\ell_{i}}{t(\log|t|^{2})} \right) \overline{\frac{\partial_{t}\varphi_{j}}{\varphi_{j}}} \right\} g_{i}g_{j}
+
\left| \sum_{j} \frac{\partial_{t}\varphi_{j}}{\varphi_{j}} g_{j} \right|^{2}
\end{aligned}
\end{equation}
and $\partial_{t\bar{t}}\log g = \partial_{t\bar{t}}g/g - |\partial_{t}g|^{2}/g^{2}$, 
we obtain \eqref{eqn:dd:log:g} from \eqref{eqn:comp:met:1}, \eqref{eqn:comp:met:2}.
Since $\partial_{t}\varphi = O(1)$, $\partial_{t\bar{t}} \varphi = O(1)$ and $g_{i}/\varphi_{i} = O( (\log|t|^{-1})^{N} )$, 
\eqref{eqn:d:log:g:asymp} follows from \eqref{eqn:dd:log:g}.
\end{pf}

\begin{lemma}
\label{lemma:Mumf:good}
Let $\varphi_{i}\in C^{\infty}_{\mathbf R}(T)$ $(0\leq i \leq N)$ and set $g(t)=\sum_{i=0}^{N}(\log|t|^{2})^{i}\varphi_{i}(t)$.
Suppose that $g(t)>0$ on $T^{o}$ and $\varphi_{i}(0)\not=0$ for some $0\leq i\leq N$. Set 
$\ell := \max\{0\leq i\leq N;\,\varphi_{i}(0)\not=0\}$.
Then there exists a constant $C>0$ such that on $T^{o}$,
$$
\left|
\partial_{t\bar{t}}\log g(t)
+
\frac{\ell}{|t|^{2}(\log |t|^{-1} )^{2}}
\right|
\leq
\frac{C}{|t|^{2}( \log |t|^{-1} )^{3}}.
$$
\end{lemma}

\begin{pf}
Set $I=\{0,1,\ldots,N\}$ and $g_{i}(t):=(\log |t|^{2})^{i}\varphi_{i}(t)$ for $i\in I$. 
Since $g(t)=\varphi_{\ell}(0)( \log |t|^{2})^{\ell}(1+O(1/\log |t|^{-1}))$ as $t\to0$, it follows from Lemma~\ref{lemma:technical} that
\begin{equation}
\label{eqn:g:i:g}
\left|\frac{g_{i}(t)}{g(t)}\right|
=
\begin{cases}
\begin{array}{ll}
O(|t| (\log |t|^{-1})^{i-\ell})&(i>\ell),
\\
1+O((\log |t|^{-1})^{-1})&(i=\ell),
\\
O((\log |t|^{-1})^{-(\ell-i)})&(i<\ell).
\end{array}
\end{cases}
\end{equation}
Since $O(1/g(t))=O(1)$ by the formula $g(t)=\varphi_{\ell}(0)( \log |t|^{2})^{\ell}(1+O(1/\log |t|^{-1}))$, 
it follows from \eqref{eqn:d:log:g:asymp} that
\begin{equation}
\label{eqn:hess:det:g}
\begin{aligned}
\partial_{t\bar{t}}\log g
&=
-\frac{\ell}{|t|^{2}(\log|t|^{2})^{2}} \frac{g_{\ell}}{g} 
-
\sum_{i\not=\ell} \frac{i}{|t|^{2}(\log|t|^{2})^{2}} \frac{g_{i}}{g}
\\
&\quad
+\frac{1}{2}\sum_{i\not=j}
\frac{(i-j)^{2}}{|t|^{2}(\log|t|^{2})^{2}}
\left(\frac{g_{i}}{g}\right)\left(\frac{g_{j}}{g}\right)
+
O\left(\frac{(-\log|t|)^{2N}}{|t|}\right).
\end{aligned}
\end{equation}
By \eqref{eqn:g:i:g}, the second and third terms on the right-hand side of 
\eqref{eqn:hess:det:g} are bounded by $|t|^{-2}(-\log|t|)^{-3}$ as $t\to0$, which implies the result.
\end{pf}

\begin{theorem}
\label{thm:c1:dir:im:K}
As $s \to 0$, the following equality holds:
\begin{equation}
\label{eqn:Poincare:curv}
c_{1}(R^{q}\pi_{*}K_{X/S}(\xi),h_{L^{2}}) 
= \left\{ \frac{\varrho^{q}}{|s|^{2}(\log|s|)^{2}} + O\left(\frac{1}{|s|^{2}(\log|s|)^{3}}\right) \right\} i\,ds\wedge d\bar{s}.
\end{equation}
In particular, the curvature form $i{\mathcal R}(s)\,ds\wedge d\bar{s}$ is Poincar\'e bounded near $0\in S$. 
Namely, there exists a constant $C>0$ such that for all $s\in S^{o}$,
$$
0 \leq i{\mathcal R}(s) \leq \frac{C}{|s|^{2}(\log|s|)^{2}}\,{\rm Id}_{R^{q}\pi_{*}K_{X/S}(\xi)}.
$$
\end{theorem}

\begin{pf} 
By \eqref{eqn:asym:exp:det:H}, 
$\mu^{*}c_{1}(R^{q}\pi_{*}K_{X/S}(\xi),h_{L^{2}}) = -\frac{i}{2\pi}\partial\bar{\partial} \log \{ \sum_{m=0}^{nh^{q}}a_{m}(t)\,(\log |t|^{2})^{m} \}$.
In Lemma~\ref{lemma:Mumf:good}, set $g(t) = \det H(t) = \sum_{m=0}^{nh^{q}} (\log |t|^{2})^{m} a_{m}(t)$.
Since $a_{m}(0)\not=0$ for some $0\leq m \leq nh^{q}$, we deduce from Lemma~\ref{lemma:Mumf:good} that
\begin{equation}
\label{eqn:pullback:c1:dir:im:K}
\mu^{*}c_{1}(R^{q}\pi_{*}K_{X/S}(\xi),h_{L^{2}}) 
= \varrho^{q}\frac{i\,dt\wedge d\bar{t}}{|t|^{2}(-\log|t|)^{2}} + O\left(\frac{i\,dt\wedge d\bar{t}}{|t|^{2}(-\log|t|)^{3}}\right).
\end{equation}
Since $\mu^{*}\{ds\wedge d\bar{s}/(|s|^{2}(-\log |s|)^{m}\} = (\deg \mu)^{2-m} dt\wedge d\bar{t}/(|t|^{2} (-\log |t|)^{m})$, 
\eqref{eqn:Poincare:curv} follows from \eqref{eqn:pullback:c1:dir:im:K}.
\par
Let $\lambda_{1},\ldots,\lambda_{h^{q}}$ be the eigenvalues of $i{\mathcal R}(s)$. 
Since $(R^{q}\pi_{*}K_{X/S}(\xi),h_{L^{2}})$ is Nakano semi-positive \cite{Berndtsson09, MourouganeTakayama08}, 
we have $\lambda_{\alpha} \geq 0$ for all $1 \leq \alpha \leq h^{q}$. This, together with \eqref{eqn:pullback:c1:dir:im:K}, yields that
$0 \leq i\,{\rm Tr}\,{\mathcal R}(s) = \sum_{\alpha}\lambda_{\alpha} \leq C/(|s|^{2}(\log |s|^{-2})^{2})$ on $S^{o}$.
Consequently, $\max_{\alpha}\{\lambda_{\alpha}\} \leq C/(|s|^{2}(-\log|s|)^{2})$. 
Since $0 \leq i{\mathcal R}(s) \leq \max_{\alpha}\{\lambda_{\alpha}\}\cdot{\rm Id}_{R^{q}\pi_{*}K_{X/S}(\xi)}$,
the desired inequality follows from this inequality . 
\end{pf}

\begin{corollary}
\label{cor:Wang}
If $X_{0}$ has only canonical singularities, 
the pseudo distance on $S$ induced by the semi-positive $(1,1)$-form $c_{1}(R^{q}\pi_{*}K_{X/S}(\xi), h_{L^{2}})$ is incomplete.
\end{corollary}

\begin{pf}
Since $\varrho^{q}=0$ by Corollary~\ref{cor:C0:L2}, the assertion follows from Theorem~\ref{thm:c1:dir:im:K}.
\end{pf}

When $\xi = {\mathcal O}_{X}$ and $q=0$, Corollary~\ref{cor:Wang} was obtained by Wang \cite[Th.\,1.1]{Wang03}.

\begin{remark}
A holomorphic Hermitian vector bundle on a punctured disc is referred to as an {\em acceptable bundle} in the sense of Mochizuki
\cite[Sect.\,21]{Mochizuki11} if the pointwise length of its curvature form with respect to the Poincar\'e metric on the punctured disc 
is bounded near the puncture. By Theorem~\ref{thm:c1:dir:im:K}, $(R^{q}\pi_{*}K_{X/S}(\xi),h_{L^{2}})$ is an acceptable bundle 
on $S^{o}$ in the sense of Mochizuki. By Theorem~\ref{thm:str:sing:L2} (1), the elementary exponents of $R^{q}\pi_{*}K_{X/S}(\xi)$
coincide with its parabolic weights. 
(See e.g. \cite[Sect.\,21.4.1]{Mochizuki11}, \cite[7.11]{FujinoFujisawaOno25} for the notion of parabolic weights.)
We refer the reader to \cite[Sect.\,21]{Mochizuki11}, \cite{FujinoFujisawaOno25} 
and the references therein for further details about acceptable bundles. 
It is worth remarking that a weak form of Theorem~\ref{thm:str:sing:L2} (1)
holds for general acceptable bundles \cite[Th.\,21.3.2]{Mochizuki11}, \cite[Th.\,1.13]{FujinoFujisawaOno25}.
\end{remark}

By Takayama \cite{Takayama22}, the $L^{2}$-metric $\mu^{*}h_{L^{2}}$ on $f_{*}K_{Y/T}(F^{*}\xi)$ is good in the sense of Mumford. 
For the higher direct image sheaves, we have the following:

\begin{proposition}
\label{prop:Mumf:good}
{\rm (1) }
The Hermitian metric $\mu^{*}\det h_{L^{2}}$ on $\det R^{q}f_{*}K_{Y/T}(F^{*}\xi)|_{T^{o}}$ is good in the sense of Mumford. 
\par{\rm (2) }
If $X_{0}$ is reduced and has only canonical singularities, then the $L^{2}$-metric $h_{L^{2}}$ on $R^{q}\pi_{*}K_{X/S}(\xi)$ 
is good in the sense of Mumford.
\end{proposition}

\begin{pf}
Combining Theorem~\ref{thm:str:sing:L2} (1), (2) and \cite[Prop.\,4.5]{Takayama22}, we get (1). 
Since $h_{L^{2}}$ satisfies a coercivity property of the diagonal part in the sense of Takayama \cite[Def.\,4.4]{Takayama22}
by Theorem~\ref{cont:L2:can:sing}, we get (2) by \cite[Prop.\,4.5]{Takayama22}.
\end{pf}

\begin{conjecture}
\label{conj:Mumford:goodness}
For any $q > 0$, the Hermitian metric $\mu^{*}h_{L^{2}}$ on $R^{q}f_{*}K_{Y/T}(F^{*}\xi)$ is good in the sense of Mumford.
\end{conjecture}

\section
{Appendix}
\label{sect:9}
\par

\subsection
{A sequence of K\"ahler forms on $Z$ approximating $\kappa_{Z}$}
\label{sect:9.1}
\par
In Subsections \ref{sect:9.1} and \ref{sect:9.2}, we keep the notations in Section~\ref{sect:4}. 
We construct the K\"ahler forms $\{\kappa_{Z,k}\}$ in the proof of Proposition~\ref{Proposition:Takegoshi}.
Although it is standard, we give a construction here for a use in the next subsection.
Since $X' \times_{S} T$ is a hypersurface in the complex manifold $X' \times T$ and 
since $\widetilde{\rho} \colon Z \to X'\times_{S}T$ is a composite of blowing-ups with non-singular centers, there is a sequence 
$$
\begin{matrix}
\, & \, E^{(m)} & \, & E^{(m-1)} & \, & E^{(1)} & \, & \,
\\
\, & \cap &\, &\cap &\, &\cap &\, &\,
\\
W &= W^{(m)} & \to & W^{(m-1)} & \to\cdots\to & W^{(1)} & \to & X' \times T
\\
\cup & \cup &\, &\cup &\, &\cup &\, &\cup
\\
Z &= Z^{(m)} & \to & Z^{(m-1)} & \to\cdots\to & Z^{(1)} &\to & X' \times_{S}T
\end{matrix}
$$
where $W^{(i)}$ is a K\"ahler manifold, $\pi_{i}\colon W^{(i)}\to W^{(i-1)}$ is a blowing-up with non-singular center 
$C^{(i-1)}\subset W^{(i-1)}$, $C^{(i-1)} \subset {\rm Sing}\,Z^{(i-1)}$, $Z^{(i)}$ is the proper transform of $Z^{(i-1)}$ by $\pi_{i}$, 
and $E^{(i)}$ is the set of exceptional hypersurfaces in $W^{(i)}$, i.e., $E^{(i)} = \widetilde{E}^{(i-1)} \cup \pi_{i}^{-1}(C^{(i-1)})$, 
where $\widetilde{E}^{(i-1)}$ consists of the proper transforms of the hypersurfaces in $E^{(i-1)}$.
We set $\varPi_{i} := \pi_{i+1} \circ\cdots\circ \pi_{m}$ ($0 \leq i \leq m-1$). Then $\widetilde{\rho} = \varPi_{0}|_{Z}$. 
Recall that, by \eqref{eqn:proj:Z}, $\varpi \colon Z \to T$ is the projection defined as $\varpi = {\rm pr}_{2} \circ \widetilde{\rho}$.
We fix a $C^{0}$-norm $\| \cdot \|$ on the space of differential forms on $W^{(i)}$ for all $i$.
\par
We set $W^{(0)} := X' \times T$ and $Z^{(0)} := X' \times_{S} T$. Let $\widetilde{\kappa}_{X'}$ be a K\"ahler form on $X'$ and 
define $\omega_{W^{(0)}} := \widetilde{\kappa}_{X'} + \kappa_{T}$. 
Let ${\mathcal V}^{(i)}_{l}\subset W^{(i)}$ be a small open neighborhood of $\pi_{i}^{-1}(C^{(i-1)})$ such that
${\rm pr}_{2}\circ\pi_{1}\circ\cdots\circ\pi_{i}({\mathcal V}^{(i)}_{l})\subset T(\frac{1}{l})$.
Suppose that a K\"ahler form $\omega_{W^{(i-1)}}^{(l)}$ on $W^{(i-1)}$ is given. 
Let $\sigma_{i}$ be the canonical section of ${\mathcal O}_{W^{(i)}}(\pi_{i}^{-1}(C^{(i-1)}))$ with ${\rm div}(\sigma_{i}) = \pi_{i}^{-1}(C^{(i-1)})$.
By \cite[Chap.VII, Prop.\,12.4]{Demailly12}, there is a Hermitian metric $h^{(i)}_{l}$ on ${\mathcal O}_{W^{(i)}}(\pi_{i}^{-1}(C^{(i-1)}))$ 
such that $\| \sigma_{i} \|_{h^{(i)}_{l}} = 1$ on $W^{(i)}\setminus{\mathcal V}^{(i)}_{l}$ and such that 
$\pi_{i}^{*}\omega_{W^{(i-1)}}^{(l)} + \epsilon dd^{c}\log \| \sigma_{i}\|_{h^{(i)}_{l}}^{2}$ is a K\"ahler form on $W^{(i)}$ 
for all $0 < \epsilon \ll 1$.
Taking $\epsilon_{l}>0$ small, we may suppose that $\epsilon_{l} \| dd^{c} \log \| \sigma_{i} \|_{h^{(i)}_{l}} \| < 2^{-l}/m$.
Then we set ${\omega}_{W^{(i)}}^{(l)} := \pi_{i}^{*} \omega_{W^{(i-1)}}^{(l)} + \epsilon_{l} dd^{c} \log \| \sigma_{i} \|_{h^{(i)}_{l}}^{2}$.
By construction, ${\omega}_{W^{(i)}}^{(l)} = \pi_{i}^{*}\omega_{W^{(i-1)}}^{(l)}$ on $W^{(i)}\setminus{\mathcal V}^{(i)}_{l}$ and 
$[{\omega}_{W^{(i)}}^{(l)}] = \pi_{i}^{*}[ \omega_{W^{(i-1)}}^{(l)} ] + \epsilon_{l} c_{1}({\mathcal O}_{W^{(i)}}(-\pi_{i}^{-1}(C^{(i-1)})))$.
Starting with $\omega_{W^{(0)}}^{(l)} := \omega_{W^{(0)}}$, 
we obtain a sequence of K\"ahler manifolds $\{(W^{(i)}, {\omega}_{W^{(i)}}^{(l)})\}$. We set
\begin{equation}
\label{eqn:(A.1)}
\widetilde{\kappa}_{Z,l} := {\omega}_{W}^{(l)}|_{Z} = \widetilde{\rho}^{*}( \widetilde{\kappa}_{X'} + \kappa_{T} ) + dd^{c}\Psi^{(l)},
\end{equation}
where $\Psi^{(l)} := \sum_{i=1}^{m} \epsilon_{l} \varPi_{i}^{*} \log \| \sigma_{i}\|_{h^{(i)}_{l}}^{2} |_{Z}$ with $\| dd^{c} \Psi^{(l)} \| < 2^{-l}$.
\par
Let $({\mathcal U}, (z_{0},\ldots,z_{n}))$ be a coordinate neighborhood of $Z$ such that 
$\varpi(z) = z_{0}^{e_{0}} \cdots z_{n}^{e_{n}}$, where $e_{i} \in {\mathbf Z}_{\geq0}$.
Since ${\rm Supp}(\varPi_{i}^{*}\sigma_{i}|_{Z}) \subset Z_{0}=\varpi^{-1}(0)$, we can express
$\varPi_{i}^{*}\sigma_{i}|_{\mathcal U} = z_{0}^{d_{0}} \cdots z_{n}^{d_{n}}{\mathbf s}$, $d_{i}\in{\mathbf Z}_{\geq0}$, 
where ${\mathbf s}$ is a local frame of $\varPi_{i}^{*}{\mathcal O}_{W^{(i)}}(\pi_{i}^{-1}(C^{(i-1)}))|_{Z}$ on ${\mathcal U}$.
Notice that $e_{i}=0$ implies $d_{i}=0$. Hence, on ${\mathcal U}$, we can express
\begin{equation}
\label{eqn:(A.2)}
\Psi^{(l)} |_{\mathcal U} = \sum_{i=0}^{n}\nu_{i}\log|z_{i}|^{2}+\phi,
\qquad
\phi\in C^{\infty}({\mathcal U}),
\quad
\nu_{i}\in{\mathbf R}.
\end{equation}
Here $e_{i}=0$ implies $\nu_{i}=0$. 
By \eqref{eqn:(A.1)}, we have 
$[ \widetilde{\kappa}_{Z,l} ] |_{Z\setminus Z_{0}} = [ \widetilde{\rho}^{*}(\widetilde{\kappa}_{X'} + \kappa_{T} ) ] |_{Z\setminus Z_{0}}$.

\medskip
\par
In Proposition~\ref{Proposition:Takegoshi}, as $\widetilde{\kappa}_{X'}$, we consider a sqeuence of K\"ahler forms
$\{ \kappa_{X',k} \}$ constructed in the same way as above from $\beta^{*}\kappa_{X}$. Applying the above construction with
initial K\"ahler form $\widetilde{\kappa}_{X'} = \kappa_{X',k}$ and setting $\kappa_{Z,k} = \widetilde{\kappa}_{Z,l(k)}$ with
$l(k)$ sufficiently large, we obtain a sequence of K\"ahler forms $\{ \kappa_{Z,k} \}$ used in Proposition~\ref{Proposition:Takegoshi}.
By the same argument as above, $\kappa_{X',k}$ can be expressed as $\kappa_{X',k} = \beta^{*}\kappa_{X} + dd^{c}\psi^{(k)}$, 
where $\psi^{(k)}$ is a smooth function on $X'\setminus X'_{0}$ with $\psi^{(k)}=1$ on $X'\setminus \pi^{\prime -1}(S(\frac{1}{k}))$ 
and $\psi^{(k)}$ admits an expression similar to \eqref{eqn:(A.2)} near $X'_{0}$. Consequently, we can express
\begin{equation}
\label{eqn:(A.3)}
\kappa_{Z,k} = \widetilde{\rho}^{*}( \beta^{*} \kappa_{X} + \kappa_{T} ) + dd^{c} \Phi_{k},
\end{equation}
where $\Phi_{k} \in C^{\infty}(Z \setminus Z_{0})$ admits the following expression near $Z_{0} = \varpi^{-1}(0)$:
\begin{equation}
\label{eqn:(A.4)}
\Phi_{k} |_{\mathcal U} = \sum_{i=0}^{n}\nu'_{i}\log|z_{i}|^{2} + \phi',
\qquad
\phi' \in C^{\infty}({\mathcal U}),
\quad
\nu'_{i}\in{\mathbf R}.
\end{equation}
Since $\Phi_{k} |_{Z\setminus Z_{0}}$ is a smooth function on $Z\setminus Z_{0}$, 
we get \eqref{eqn:equality:coh:2} by \eqref{eqn:Takegoshi:2}, \eqref{eqn:(A.3)}.

\subsection
{Takegoshi's theorem with respect to the degenerate K\"ahler form}
\label{sect:9.2}
\par
Takegoshi's theorem (Theorem~\ref{Thm:Takegoshi}) for $Z$ with respect to the degenerate K\"ahler form 
$\kappa_{Z}=\widetilde{\rho}^{*}(\beta^{*}\kappa_{X}+\kappa_{T})$ was obtained by 
Mourougane and Takayama \cite[Prop.\,4.4]{MourouganeTakayama09}.
However, it is not immediate from the proof if the operator $L^{q}$ in \cite[Prop.\,4.4]{MourouganeTakayama09} is given by 
the multiplication by $\kappa_{Z}^{q}$. To clarify this point, we prove the following:

\begin{proposition}
\label{prop:Takegoshi:sing}
For any $u\in H^{q}(Z, K_{Z}(\widetilde{F}^{*}\xi))$, there exists $\sigma \in \Gamma(Z,\Omega_{Z}^{n+1-q}(\widetilde{F}^{*}\xi))$ with
$u=[\sigma\wedge\kappa_{Z}^{q}]$, $(\varpi^{*}dt)\wedge\sigma=0$ and $D^{\widetilde{F}^{*}\xi, \widetilde{F}^{*}h_{\xi}} \sigma =0$.
\end{proposition}

\begin{pf}
We write $Z_{0} = \varpi^{-1}(0) =\bigcup_{\lambda\in\Lambda}D_{\lambda}$.
Then there is a K\"ahler form $\kappa_{Z,k}$ on $Z$ of the form $\kappa_{Z,k} = \kappa_{Z} + dd^{c}\Phi_{k}$ with 
$\Phi_{k} \in C^{\infty}(Z\setminus Z_{0})$. 
By Theorem~\ref{Thm:Takegoshi}, there exists $\sigma \in \Gamma(Z,\Omega_{Z}^{n+1-q}(\widetilde{F}^{*}\xi))$ such that
$u=[\sigma \wedge \kappa_{Z,k}^{q}]$, $(\varpi^{*}dt)\wedge\sigma=0$ and $D^{\widetilde{F}^{*}\xi, \widetilde{F}^{*}h_{\xi}} \sigma =0$.
Since 
$\sigma\wedge(\kappa_{Z,k}^{q}-\kappa_{Z}^{q}) 
= \sigma\wedge(\kappa_{Z,k}-\kappa_{Z})\wedge\sum_{i+j=q-1} \kappa_{Z,k}^{i}\kappa_{Z}^{j}$,
it suffices to prove the equality of cohomology classes $[\sigma\wedge(\kappa_{Z,k}-\kappa_{Z})]=0$ on $Z$.
Since $\sigma\wedge(\kappa_{Z,k}-\kappa_{Z}) = \frac{-1}{2\pi i} \bar{\partial}\{\sigma\wedge\partial\Phi_{k}\}$,
it suffices to prove that $\sigma\wedge\partial\Phi_{k}$ is a smooth differential form on $Z$.
\par
Set $E := \bigcup_{\lambda\not=\lambda'}D_{\lambda}\cap D_{\lambda'}$ and $D_{\lambda}^{o} := D_{\lambda}\setminus E$.
Let $p\in D_{\lambda}^{o}$. There is a system of coordinates $({\mathcal U},(z_{0},\ldots,z_{n}))$ of $Z$ centered at $p$ such that 
$\varpi(z)=\epsilon(z)\,z_{0}^{e_{\lambda}}$, $e_{\lambda} \in {\mathbf Z}_{>0}$ and $\epsilon(z)\in{\mathcal O}^{*}({\mathcal U})$. 
Since $\varpi(z)=\epsilon(z)\,z_{0}^{e_{\lambda}}$ and hence 
$\varpi^{*}dt = e_{\lambda}\epsilon(z)z_{0}^{e_{\lambda}-1}dz_{0}+z_{0}^{e_{\lambda}}d\epsilon$,
the condition $(\varpi^{*}dt) \wedge \sigma=0$ implies that 
\begin{equation}
\label{eqn:(A.5)}
\frac{dz_{0}}{z_{0}} \wedge \sigma = -\frac{1}{e_{\lambda}\epsilon(z)}d\epsilon(z)\wedge\sigma
\in
\Gamma( {\mathcal U}, \Omega_{Z}^{n+2-q}(\widetilde{F}^{*}\xi)).
\end{equation}
\par
Let $p \in E$. There is a system of coordinates $({\mathcal U},(z_{0},\ldots,z_{n}))$ of $Z$ centered at $p$ such that 
$\varpi(z) = z_{0}^{e_{0}}\cdots z_{n}^{e_{n}}$, $e_{i} \in {\mathbf Z}_{\geq0}$.
If $e_{i}>0$, then 
$(\frac{dz_{i}}{z_{i}}\wedge\sigma)|_{{\mathcal U} \setminus E} \in 
\Gamma({\mathcal U} \setminus E, \Omega_{Z}^{n+2-q}(\widetilde{F}^{*}\xi))$
by \eqref{eqn:(A.5)}. Since $E$ has codimension $2$ in ${\mathcal U}$, the Riemann extension theorem implies that, if $e_{i}>0$, 
then
\begin{equation}
\label{eqn:(A.6)}
\frac{dz_{i}}{z_{i}}\wedge\sigma \in \Gamma( {\mathcal U}, \Omega_{Z}^{n+2-q}(\widetilde{F}^{*}\xi) ).
\end{equation}
\par
Let $p\in Z_{0}$ and let $({\mathcal U},(z_{0},\ldots,z_{n}))$ be system of coordinates as above. 
On ${\mathcal U}$, we may suppose that $\Phi_{k}$ is expressed as in \eqref{eqn:(A.4)}.
By \eqref{eqn:(A.4)}, \eqref{eqn:(A.6)}, we get
$$
\sigma \wedge \partial\Phi_{k} |_{\mathcal U} = \sum_{i=0}^{n} \nu'_{i}\,\sigma \wedge \frac{dz_{i}}{z_{i}} + \sigma\wedge\partial\phi'
\in A^{n+2-q,0}({\mathcal U},\widetilde{F}^{*}\xi),
$$
where $\nu'_{i}=0$ if $e_{i}=0$. Since $p\in Z_{0}$ is arbitrary, $\sigma\wedge\partial\Phi_{k} \in A^{n+2-q,0}(Z, \widetilde{F}^{*}\xi)$.
\end{pf}

We are now in a position to prove Takegoshi's theorem (Theorem~\ref{Thm:Takegoshi}) for $(Y,F^{*}\xi)$ with respect to 
the degenerate K\"ahler form $\kappa_{Y} = \rho^{*}( \beta^{*}\kappa_{X} + \kappa_{T} )$.

\begin{proposition}
\label{prop:Takegoshi:sing:2}
For any $v\in H^{q}(Y, K_{Y}(F^{*}\xi))$, there exists a holomorphic form $\tau \in \Gamma(Y, \Omega_{Y}^{n+1-q}(F^{*}\xi))$ 
with $v=[\tau\wedge\kappa_{Y}^{q}]$, $(f^{*}dt)\wedge\tau=0$ and $D^{F^{*}\xi, F^{*}h_{\xi}} \tau =0$.
\end{proposition}

\begin{pf}
Set $u := \varphi^{*}v\in H^{q}(Z, K_{Z}(\widetilde{F}^{*}\xi))$. 
By Proposition~\ref{prop:Takegoshi:sing}, there exists $\sigma \in \Gamma(Z,\Omega_{Z}^{n+1-q}(\widetilde{F}^{*}\xi))$ with $u=[\sigma\wedge\kappa_{Z}^{q}]$, $(\varpi^{*}dt)\wedge\sigma=0$ and $D^{\widetilde{F}^{*}\xi, \widetilde{F}^{*}h_{\xi}} \sigma =0$. 
Since $\varphi \colon Z \to Y$ is a proper modification, there exists $\tau \in \Gamma(Y,\Omega_{Y}^{n+1-q}(F^{*}\xi))$ with 
$\varphi^{*}\tau=\sigma$. Since $\varphi^{*}\kappa_{Y}=\kappa_{Z}$, we get $\varphi^{*}( v - [\tau\wedge\kappa_{Y}^{q}] ) = 0$. 
Since the map $\varphi^{*}\colon H^{q}(Y, K_{Y}(F^{*}\xi))\to H^{q}(Z, K_{Z}(\widetilde{F}^{*}\xi))$ is an isomorphism 
(cf. \cite[Theorem I, II]{Takegoshi95}), 
we get $v=[\tau\wedge\kappa_{Y}^{q}]$. Then the identities $(f^{*}dt) \wedge \tau = 0$ and $D^{F^{*}\xi, F^{*}h_{\xi}} \tau = 0$ follow from 
the corresponding identities $\varpi^{*}dt \wedge \sigma =0$ and $D^{\widetilde{F}^{*}\xi, \widetilde{F}^{*}h_{\xi}} \sigma =0$. 
\end{pf}

\subsection
{A formula for the Bott-Chern type currents}
\label{sect:9.3}
\par
We keep the notation in Section~\ref{sect:6}.
In Subsection~\ref{sect:9.3}, we assume that $N$ is a reduced smooth divisor on $M$
and that $E={\mathcal O}_{M}(-N)$, $E'={\mathcal O}_{M}$. 
Let $s \in \Gamma(M, [N])$ be the section such that $\varphi(e) = e\otimes s$ and ${\rm div}(s)=N$.
Let $\nu$ be the normal bundle of $N$ in $M$. Let $h_{E}$, $h_{E'}$, $h_{\nu}$ be Hermitian metrics on these bundles.
Suppose that $h_{E'}$ is the trivial metric. We set $h_{[N]} := h_{E}^{-1}$.
In this setting, we provide an explicit formula for the Bott-Chern type current
$\langle \widetilde{\rm ch}( \overline{E}, \overline{E'} ) \rangle$ and compare it with the Bott-Chern singular current of 
Bismut-Gillet-Soul\'e \cite{BGS90}, \cite{BGS90b}, which we denote by $T( \overline{E}, \overline{E'}, h_{\nu})$.

\begin{theorem}
\label{thm:comparison:BC}
For $k\geq 1$, set $a_{k} := \frac{1}{k}\sum_{j=1}^{k-1}\frac{1}{j}$ $(k\geq2)$ and $a_{1}:=0$. Then
\begin{equation}
\label{eqn:ch:0}
\langle \widetilde{\rm ch}( \overline{E}, \overline{E'} ) \rangle =  
{\rm Td}^{-1}( [N], h_{[N]} ) \log \| s \|^{2} + \sum_{k=1}^{\infty} (-1)^{k}\frac{a_{k}}{k!} c_{1}( [N], h_{[N]} )^{k}.
\end{equation}
In particular, if $M$ is compact and the metrics $h_{E}$, $h_{E'}$, $h_{\nu}$ verify Bismut's Assumption (A), i.e., 
$\| ds|_{N} \| = 1$ on $N$, then the following identity holds in $P^{M}_{N}/P^{M,0}_{N}$
\begin{equation}
\label{eqn:ch:1}
\langle \widetilde{\rm ch}( \overline{E}, \overline{E'} ) \rangle 
\equiv 
T( \overline{E}, \overline{E'}, h_{\nu} ) + \sum_{k=1}^{\infty} (-1)^{k}\frac{a_{k}}{k!} c_{1}( [N], h_{[N]} )^{k},
\end{equation}
where $P^{M}_{N}$ and $P^{M,0}_{N}$ are the spaces of currents on $M$ introduced in \cite[Def.\,1.3]{BGS90b}.
\end{theorem}

\begin{pf}
Set $F := E \oplus E' = [N]^{-1} \oplus {\mathcal O}_{M}$ and $h_{F}:=h_{E} \oplus h_{E'}$. We set $\overline{[N]} := ([N], h_{[N]})$ as well.
Let $U \subset M$ be a small open subset and let ${\mathbf e} \in \Gamma(U, E)$ be a nowhere vanishing section.
We fix a holomorphic trivialization $F|_{U} \cong {\mathcal O}_{U}^{\oplus 2}$ given by $(v_{1}{\mathbf e}, v_{2}) \mapsto (v_{1},v_{2})$. 
On the open subset of ${\mathbf P}(F(1))|_{U \times {\mathbf P}^{1}}$ defined by $v_{1}\not=0$, $\zeta\not=0,\infty$,
${\mathscr U} = {\mathcal O}_{{\mathbf P}(F(1))}(-1)$ is generated by the section
$(x, \zeta, v) \mapsto ([({\mathbf e}(x)\otimes\sigma_{\infty}(\zeta), v\sigma_{0}(\zeta))], 
({\mathbf e}(x)\otimes\sigma_{\infty}(\zeta), v\sigma_{0}(\zeta)))$, 
where $[({\mathbf e}\otimes\sigma_{\infty}, v\sigma_{0})] \in {\mathbf P}(F(1))|_{U \times {\mathbf P}^{1}}$ is the point corresponding to 
$({\mathbf e}\otimes\sigma_{\infty}, v\sigma_{0}) \in F(1)|_{U \times {\mathbf P}^{1}}$.
By \eqref{eqn:hom:v.b.}, $\sigma_{\varphi}^{*}{\mathscr U}|_{U\times{\mathbf P}^{1}}$ is generated by the section
$([\varepsilon_{\varphi}({\mathbf e})], \varepsilon_{\varphi}({\mathbf e}))$ with
$\varepsilon_{\varphi}({\mathbf e}) = ( {\mathbf e}\otimes\sigma_{\infty}, \varphi({\mathbf e})\sigma_{0})$. 
Set $u:= 1/\zeta$ and $\omega_{{\mathbf P}^{1}}:=\frac{i}{2\pi} \frac{d\zeta\wedge d\bar{\zeta}}{(1+|\zeta|^{2})^{2}}
= \frac{i}{2\pi} \frac{du \wedge d\bar{u}}{(1+|u|^{2})^{2}}$. 
Since $s = \varphi({\mathbf e}){\mathbf e}^{-1}$, we have 
\begin{equation}
\label{eqn:met:tautlog.l.b}
h_{\mathscr U}( [\varepsilon_{\varphi}({\mathbf e})] , \varepsilon_{\varphi}({\mathbf e}) ) 
= 
\frac{\|{\mathbf e}\|^{2} + |\varphi({\mathbf e})|^{2} |\zeta|^{2}}{1+|\zeta|^{2}}
=
\frac{\| {\mathbf e} \|^{2}}{1+|u|^{2}} ( |u|^{2} + \| s \|^{2} ).
\end{equation}
Set $\theta := \partial \log \|s\|^{2}$ and $\Theta := \bar{\partial}\partial \log \| s \|^{2}$. Then $c_{1}(\overline{[N]})=\frac{i}{2\pi}\Theta$. 
Since
$$
\bar{\partial}\partial \log ( |u|^{2} + \| s \|^{2} ) 
= 
\frac{\| s \|^{2}}{ ( |u|^{2} + \|s\|^{2} )^{2} } (d\bar{u} - \bar{u} \bar{\theta}) {\scriptstyle \wedge} (du - u \theta ) 
+ \frac{\| s \|^{2}}{ |u|^{2} + \| s \|^{2}} \Theta
$$
and $\bar{\partial}\partial \log \| {\mathbf e} \|^{2} = -\bar{\partial}\partial \log \| s \|^{2} = -\Theta$, we obtain
\begin{equation}
\label{eqn:ch:2}
\begin{aligned}
\,&
\exp\left\{ \bar{\partial}\partial \log \| {\mathbf e} \|^{2} + \bar{\partial}\partial \log ( |u|^{2} + \| s \|^{2} ) \right\}
\\
&=
\left\{ 1 + \frac{\| s \|^{2}}{ ( |u|^{2} + \|s\|^{2} )^{2} } (d\bar{u} - \bar{u} \bar{\theta} ) {\scriptstyle \wedge} (du - u \theta ) \right\}
\exp\left\{ - \frac{|u|^{2}}{ |u|^{2} + \| s \|^{2}} \Theta \right\}
\\
&=
\left\{ 1 + \frac{\| s \|^{2}}{ ( |u|^{2} + \|s\|^{2} )^{2} } (d\bar{u} {\scriptstyle \wedge} du + |u|^{2} \bar{\theta} {\scriptstyle \wedge} \theta ) 
+ du {\scriptstyle \wedge} A + d\bar{u} {\scriptstyle \wedge} B \right\}
\exp\left\{ - \frac{|u|^{2}}{ |u|^{2} + \| s \|^{2}} \Theta \right\},
\end{aligned}
\end{equation}
where $A$ and $B$ contain neither $du$ nor $d\bar{u}$. By \eqref{eqn:met:tautlog.l.b}, \eqref{eqn:ch:2}, we get
\begin{equation}
\label{eqn:ch:3}
\begin{aligned}
\,&
\exp( c_{1}( \sigma_{\varphi}^{*}\overline{\mathscr U} ))
=
\exp\left\{ -dd^{c}\log\| {\mathbf e} \|^{2} -dd^{c}\log ( |u|^{2} + \| s \|^{2} ) - \omega_{{\mathbf P}^{1}} \right\}
\\
&=
\left\{ 
1 + \frac{i}{2\pi}\frac{\| s \|^{2}}{ ( |u|^{2} + \|s\|^{2} )^{2} } (d\bar{u} {\scriptstyle \wedge} du + |u|^{2} \theta {\scriptstyle \wedge} \theta ) 
+ du {\scriptstyle \wedge} A' + d\bar{u} {\scriptstyle \wedge} B' \right\}
\\
&\qquad
{\scriptstyle \wedge} \exp\left\{ - \frac{|u|^{2}}{ |u|^{2} + \| s \|^{2}} c_{1}(\overline{[N]}) \right\} {\scriptstyle \wedge} ( 1 - \omega_{{\mathbf P}^{1}} )
\\
&=
-\omega_{{\mathbf P}^{1}} {\scriptstyle \wedge}
\left( 1 + \frac{i}{2\pi} \frac{|u|^{2} \|s\|^{2}}{( |u|^{2} + \| s \|^{2} )^{2}} \theta {\scriptstyle \wedge} \bar{\theta}  \right)
\exp\left\{ - \frac{|u|^{2}}{ |u|^{2} + \| s \|^{2}} c_{1}(\overline{[N]}) \right\}
\\
&\quad
-
\frac{\| s \|^{2}}{ ( |u|^{2} + \|s\|^{2} )^{2} } \frac{i}{2\pi} du {\scriptstyle \wedge} d\bar{u} 
{\scriptstyle \wedge} \exp\left\{ - \frac{|u|^{2}}{ |u|^{2} + \| s \|^{2}} c_{1}(\overline{[N]}) \right\}
\\
&\quad
+ du {\scriptstyle \wedge} A'' + d\bar{u} {\scriptstyle \wedge} B'' +C'',
\end{aligned}
\end{equation}
where $A'$, $B'$, $A''$, $B''$, $C''$ contain neither $du$ nor $d\bar{u}$. Set $s = uw {\mathbf e}^{-1}$. 
Then, on $q^{-1}(U \times {\mathbf P}^{1})$, $q^{-1}(N \times \{\infty\})$ is defined locally by $u=w=0$.
By \eqref{eqn:ch:3}, we have
\begin{equation}
\label{eqn:ch:4}
\begin{aligned}
\,&
q^{*}\exp( c_{1}( \sigma_{\varphi}^{*}\overline{\mathscr U} ))
\\
&=
-\omega_{{\mathbf P}^{1}} 
\left( 1 + \frac{i}{2\pi} \frac{\frac{|w|^{2}}{\|{\mathbf e} \|^{2}}}{( 1 + \frac{|w|^{2}}{\|{\mathbf e} \|^{2}} )^{2}} 
\partial \log \frac{|w|^{2}}{\| {\mathbf e}\|^{2}} {\scriptstyle \wedge} \bar{\partial} \log \frac{|w|^{2}}{\| {\mathbf e} \|^{2}}  \right)
\exp\left\{ - \frac{c_{1}(\overline{[N]})}{ 1 + \frac{|w|^{2}}{\|{\mathbf e} \|^{2}} } \right\}
\\
&\quad
-
\frac{\| s \|^{2}}{ ( |u|^{2} + \| s \|^{2} )^{2} } \frac{i}{2\pi} du {\scriptstyle \wedge} d\bar{u} 
{\scriptstyle \wedge} \exp\left\{ - \frac{|u|^{2}}{ |u|^{2} + \| s \|^{2}} c_{1}(\overline{[N]}) \right\}
\\
&\quad
+ du {\scriptstyle \wedge} q^{*}A'' + d\bar{u} {\scriptstyle \wedge} q^{*}B'' + q^{*}C''.
\end{aligned}
\end{equation}
Since $\int_{{\mathbf P}^{1}} \log |u|^{2} \omega_{{\mathbf P}^{1}} =0$, it follows from \eqref{eqn:ch:4} that
\begin{equation}
\label{eqn:ch:5}
\begin{aligned}
\,&
\int_{{\mathbf P}^{1}} \log |\zeta|^{2} \cdot q^{*}\exp( c_{1}( \sigma_{\varphi}^{*}\overline{\mathscr U} ))
=
-\int_{{\mathbf P}^{1}} \log |u|^{2} \cdot q^{*}\exp( c_{1}( \sigma_{\varphi}^{*}\overline{\mathscr U} ))
\\
&=
\int_{{\mathbf P}^{1}} \frac{\| s \|^{2} \log |u|^{2}}{ ( |u|^{2} + \| s \|^{2} )^{2} } \frac{i}{2\pi} du {\scriptstyle \wedge} d\bar{u} 
{\scriptstyle \wedge} \exp\left\{ - \frac{|u|^{2}}{ |u|^{2} + \| s \|^{2}} c_{1}(\overline{[N]}) \right\}
\\
&=
\sum_{k=0}^{\infty} \frac{(-c_{1}(\overline{[N]}))^{k}}{k!} \log \| s \|^{2} 
\int_{\mathbf C} \frac{|\zeta|^{2k}}{(1+|\zeta|^{2})^{k+2}} \frac{i}{2\pi} d\zeta d\bar{\zeta} 
\\
&\quad
+\sum_{k=0}^{\infty} \frac{(-c_{1}(\overline{[N]}))^{k}}{k!} 
\int_{\mathbf C} \frac{|\zeta|^{2k} \log |\zeta|^{2}}{(1+|\zeta|^{2})^{k+2}}  \frac{i}{2\pi} d\zeta d\bar{\zeta}.
\end{aligned}
\end{equation}
Since 
$\int_{\mathbf C} \frac{i}{2\pi} \frac{|\zeta|^{2k} d\zeta d\bar{\zeta}}{(1+|\zeta|^{2})^{k+2}} = 
\int_{\mathbf C} \frac{i}{2\pi} \frac{du d\bar{u}}{(1+|u|^{2})^{k+2}} = \frac{1}{k+1}$ and
$\int_{\mathbf C} \frac{|\zeta|^{2k} \log |\zeta|^{2}}{(1+|\zeta|^{2})^{k+2}} \frac{i d\zeta d\bar{\zeta}}{2\pi}$ is given by
$-\int_{0}^{\infty} \frac{\log\rho}{(1+\rho)^{k+1}} d\rho =a_{k}$,
\eqref{eqn:ch:0} follows from \eqref{eqn:ch:5}.
When $h_{E}$, $h_{E'}$, $h_{\nu}$ verify Bismut's Assumption (A), \eqref{eqn:ch:1} is derived from 
\eqref{eqn:ch:0} and \cite[Rem.\,3.5, especially (3.23), Th.\,3.15, Th.\,3.17]{BGS90}.
\end{pf}



\begin{thebibliography}{99}


\bibitem{AKMW02}
Abramovich, D., Karu, K., Matsuki, K., Wlodarczyk, J.
\newblock 
\textit{Torification and factorization of birational maps},
\newblock 
J. Amer. Math. Soc. 
\newblock 
{\bf 15}
\newblock 
(2002),
\newblock 
531--572.


\bibitem{Barlet82}
Barlet, D.
\newblock 
\textit{D\'eveloppement asymptotique des fonctions obtenues par int\'egration sur les fibres},
\newblock 
Invent. Math. 
\newblock 
{\bf 68}
\newblock 
(1982),
\newblock 
129--174.


\bibitem{BermanBoucksom10}
Berman, R., Boucksom, S.
\textit{Growth of balls of holomorphic sections and energy at equilibrium},
\newblock 
Invent. Math. 
\newblock 
{\bf 181}
\newblock 
(2010),
\newblock 
337--394.


\bibitem{Berndtsson09}
Berndtsson, B.
\newblock
\textit{Curvatures of vector bundles associated to holomorphic fibrations},
\newblock
Ann. of Math.
\newblock
{\bf 169}
\newblock
(2009)
\newblock
531--560.


\bibitem{BierstoneMilman97}
Bierstone, E., Milman, P.
\newblock
\textit{Canonical desingularization in characteristic zero by blowing up the maximum strata of a local invariant},
\newblock 
Invent. Math.
\newblock 
{\bf 128}
\newblock 
(1997),
\newblock 
207--302.


\bibitem{Bismut90}
Bismut, J.-M.
\newblock 
\textit{Superconnection currents and complex immersions},
\newblock 
Invent. Math.
\newblock 
{\bf 99}
\newblock 
(1990),
\newblock 
59--113.


\bibitem{Bismut95}
Bismut, J.-M.
\newblock 
\textit{Equivariant immersions and Quillen metrics},
\newblock 
J. Differential Geom.
\newblock 
{\bf 41}
\newblock 
(1995),
\newblock 
53--157.


\bibitem{Bismut97}
Bismut, J.-M.
\newblock 
\textit{Quillen metrics and singular fibers in arbitrary relative dimension},
\newblock 
J. Algebraic Geom.
\newblock 
{\bf 6}
\newblock 
(1997),
\newblock 
19--149.


\bibitem{BismutBost90}
Bismut, J.-M., Bost, J.-B.
\newblock
\textit{Fibr\'es d\'eterminants, m\'etriques de Quillen et d\'eg\'en\'erescence des courbes},
\newblock 
Acta Math.
\newblock 
{\bf 165}
\newblock 
(1990)
\newblock 
1--103.


\bibitem{BGS88}
Bismut, J.-M., Gillet, H., Soul\'e, C.
\newblock 
\textit{Analytic torsion and holomorphic determinant bundles I,II,III},
\newblock 
Commun. Math. Phys.
\newblock 
{\bf 115}
\newblock 
(1988),
\newblock 
49--78, 
\newblock 
79--126, 
\newblock 
301--351.


\bibitem{BGS90}
Bismut, J.-M., Gillet, H., Soul\'e, C.
\newblock 
\textit{Complex immersions and Arakelov geometry},
\newblock 
(P. Cartier et al., eds.), The Grothendieck Festschrift,
\newblock 
Birkh\"auser,
\newblock 
Boston
\newblock 
(1990)
\newblock 
249--331.


\bibitem{BGS90b}
Bismut, J.-M., Gillet, H., Soul\'e, C.
\newblock 
\textit{Bott-Chern currents and complex immersions},
\newblock
Duke Math. J.
\newblock
{\bf 60}
\newblock
(1990),
\newblock
255--284.


\bibitem{BismutLebeau91}
Bismut, J.-M., Lebeau, G.
\newblock 
\textit{Complex immersions and Quillen metrics},
\newblock 
Publ. Math. IHES
\newblock 
{\bf 74}
\newblock 
(1991),
\newblock 
1--297.


\bibitem{BismutVasserot89}
Bismut, J.-M., Vasserot, E.
\newblock
\textit{The asymptotics of the Ray-Singer analytic torsion associated with high powers of a positive line bundle},
Commun. Math. Phys.
\newblock 
{\bf 125}
\newblock 
(1989),
\newblock 
355--367. 


\bibitem{DaiYoshikawa25}
Dai, X., Yoshikawa, K.-I.
\newblock
{\em Degeneration of Riemann surfaces and the small eigenvalues of the Laplacian},
\newblock
Courant J.  Pure Applied Math. (to appear),
\newblock
arXiv:2509.06151


\bibitem{Demailly12}
Demailly, J.-P.
\newblock 
\textit{Complex Analytic and Differential Geometry},
\newblock 
\url{https://www-fourier.univ-grenoble-alpes.fr/~demailly/manuscripts/agbook.pdf}


\bibitem{EFM18}
Eriksson, D., Freixas i Montplet, G., Mourougane, C.
\newblock
\textit{Singularities of metrics on Hodge bundles and their topological invariants},
\newblock
Algebraic Geometry,
\newblock
{\bf 5}
\newblock
(2018),
\newblock
742--775.


\bibitem{EFM21}
Eriksson, D., Freixas i Montplet, G., Mourougane, C.
\newblock
\textit{BCOV invariants of Calabi-Yau manifolds and degenerations of Hodge structures},
\newblock
Duke Math. J.
\newblock
{\bf 170}
\newblock
(2021),
\newblock
379--454.


\bibitem{ErikssonFreixas26}
Eriksson, D., Freixas i Montplet, G.
\newblock
\textit{The spectral genus of an isolated hypersurface singularity and its relation to the Milnor number and analytic torsion},
\newblock
Doc. Math. 
\newblock
{\bf 31}
\newblock
(2026),
\newblock
1387--1419.


\bibitem{FLY08}
Fang, H., Lu, Z., Yoshikawa, K.-I.
\newblock 
\textit{Analytic torsion for Calabi--Yau threefolds},
\newblock 
J. Differential Geometry
\newblock
{\bf 80}
\newblock
(2008),
\newblock
175--259.


\bibitem{Finski18}
Finski, S.
\newblock 
\textit{On the full asymptotics of analytic torsion},
\newblock 
J. Funct. Anal.
\newblock 
{\bf 275}
\newblock 
(2018),
\newblock 
3457--3503.


\bibitem{FujinoFujisawaOno25}
Fujino, O., Fujisawa, T., Ono, T.
\newblock
\textit{Notes on acceptable bundles},
\newblock
preprint,
\newblock
arXiv:2511.00760
\newblock
(2025).


\bibitem{Fulton98}
Fulton, W.
\newblock
\textit{Intersection Theorey},
\newblock
2nd Ed.
\newblock
Springer,
\newblock
New York
\newblock
(1998)


\bibitem{GilletSoule90}
Gillet, H., Soul\'e, C.
\newblock
\textit{Characteristic classes for algebraic vector bundles with hermitian metric, I, II},
\newblock
Ann. of Math.
\newblock
{\bf 131}
\newblock
(1990),
\newblock
163--238.


\bibitem{Imaike25}
Imaike, D.
\newblock
\textit{Analytic torsion for irreducible holomorphic symplectic fourfolds with involution, I: Construction of an invariant},
\newblock
International J. Math.
\newblock
{\bf 36}
\newblock
(2025)
\newblock
2550010 (27 pages). 


\bibitem{Imaike24}
Imaike, D.
\newblock
\textit{Analytic torsion for irreducible holomorphic symplectic fourfolds with involution, II: the singularity of the invariant},
\newblock
preprint,
\newblock
arXiv:2411.13911v1
\newblock
(2024). 


\bibitem{KawaguchiYoshikawa26}
Kawaguchi, S., Yoshikawa, K.-I.
\newblock
\textit{Reflective modular forms on the moduli space of Eisenstein $K3$ surfaces and equivariant analytic torsions},
\newblock
(tentative title),
\newblock
in preparation.


\bibitem{Mumford73}
Kempf, G., Knudsen, F., Mumford, D., Saint-Donat, B.
\newblock
\textit{Toroidal Embeddings I},
\newblock 
Lecture Notes Math. 
\newblock 
{\bf 339}  
\newblock 
(1973)


\bibitem{KnudsenMumford76}
Knudsen, F.F., Mumford, D.
\newblock 
\textit{The projectivity of the moduli space of stable curves, I.},
\newblock 
Math. Scand.
\newblock 
{\bf 39}
\newblock 
(1976),
\newblock 
19--55.


\bibitem{KohlerRoessler01}
K\"ohler, K., Roessler, D.
\newblock
\textit{A fixed point formula of Lefschetz type in Arakelov geometry I},
\newblock 
Invent. Math.
\newblock
{\bf 145}
\newblock 
(2001),
\newblock 
333--396.


\bibitem{Kollar97}
Koll\'ar, J.
\newblock
\textit{Singularities of pairs},
\newblock 
Proc. Sympos. Pure Math.
\newblock
{\bf 62},
\newblock
Part 1
\newblock 
(1997),
\newblock 
221--287.


\bibitem{Ma00}
Ma, X.
\newblock 
\textit{Submersions and equivariant Quillen metrics},
\newblock 
Ann. Inst. Fourier
\newblock 
{\bf 50}
\newblock 
(2000),
\newblock 
1539--1588.


\bibitem{Ma21}
Ma, X.
\newblock
\textit{Orbifold submersion and analytic torsions},
\newblock
Progress in Math.
\newblock
{\bf 338}
\newblock
(2021),
\newblock
141--177.


\bibitem{Mochizuki11}
Mochizuki, T.
\newblock
\textit{Wild harmonic bundles and wild pure twistor $D$-modules},
Ast\'erisque
\newblock
{\bf 340}
\newblock
(2011).


\bibitem{MourouganeTakayama08}
Mourougane, C., Takayama, S.
\newblock
\textit{Hodge metrics and the curvature of higher direct images},
\newblock
Ann. Sci. \'Ec. Norm. Sup.
\newblock
{\bf 41}
\newblock
(2008),
\newblock
905--924.


\bibitem{MourouganeTakayama09}
Mourougane, C., Takayama, S.
\newblock
\textit{Extension of twisted Hodge metrics for K\"ahler morphisms},
\newblock
J. Differential Geom.
\newblock
{\bf 83}
\newblock
(2009),
\newblock
131--161.


\bibitem{Quillen73}
Quillen, D.
\newblock
\textit{Higher $K$-theories. I.},
\newblock
Lecture Notes in Math.
\newblock
{\bf 341},
\newblock
Springer,
\newblock
(1973),
\newblock
85--147.


\bibitem{Schmid73}
Schmid, W.
\newblock 
\textit{Variation of Hodge structure: The singularities of the period mapping},
\newblock 
Invent. Math.
\newblock 
{\bf 22}
\newblock 
(1973),
\newblock 
211--319.


\bibitem{Siu75}
Siu, Y.-T.
\newblock
\textit{Extension of meromorphic maps into K\"ahler manifolds},
\newblock
Ann. of Math.
\newblock
{\bf 102}
\newblock
(1975),
\newblock 
421--462.


\bibitem{Siu82}
Siu, Y.-T.
\newblock
\textit{Complex-analyticity of harmonic maps, vanishing and Lefschetz theorems},
\newblock
J. Differential Geom.
\newblock
{\bf 17}
\newblock
(1982),
\newblock
55--138.


\bibitem{Steenbrink77}
Steenbrink, J.H.M.
\newblock 
\textit{Mixed Hodge structure on vanishing cohomology},
\newblock 
Real and Complex Singularities,
\newblock 
Sijthoff-Noordhoff, Alphen aan den Rijn
\newblock 
(1977),
\newblock 
525--563.


\bibitem{Stevens88}
Stevens, J.
\newblock 
\textit{On canonical singularities as total spaces of deformations},
\newblock 
Abh. Math. Sem. Univ. Hamburg
\newblock 
{\bf 58}
\newblock 
(1988),
\newblock 
275--283.


\bibitem{Takayama19}
Takayama, S.
\newblock
\textit{Moderate degenerations of Ricci-flat K\"ahler-Einstein manifolds over higher dimensional bases},
\newblock
J. Math. Sci. Univ. Tokyo
\newblock
{\bf 26}
\newblock
(2019),
\newblock
335--359.


\bibitem{Takayama21}
Takayama, S.
\newblock
\textit{Asymptotic expansions of fiber integrals over higher-dimensional bases},
\newblock
J. reine angew. Math. 
\newblock
{\bf 773} 
\newblock
(2021), 
\newblock
67--128.


\bibitem{Takayama22}
Takayama, S.
\newblock
\textit{Mumford goodness of canonical $L^{2}$-metrics on direct image sheaves over a curve},
\newblock
Adv. Math. 
\newblock
{\bf 405} 
\newblock
(2022), 
\newblock
108485.


\bibitem{Takegoshi95}
Takegoshi, K.
\newblock 
\textit{Higher direct images of canonical sheaves tensorized with semi-positive vector bundles by proper K\"ahler morphisms},
\newblock 
Math. Ann.
\newblock 
{\bf 303}
\newblock 
(1995),
\newblock 
389--416.


\bibitem{Wang03}
Wang, C.-L.
\newblock
\textit{Quasi-Hodge metrics and canonical singularities},
\newblock
Math. Res. Lett.
\newblock
{\bf 10}
\newblock
(2003),
\newblock
57--70.


\bibitem{Wells08}
Wells, R.O.
\newblock
\textit{Differential Analysis on Complex Manifolds},
\newblock
3rd Ed.
\newblock
Springer,
\newblock
(2008).


\bibitem{Wolpert87}
Wolpert, S.
\newblock 
\textit{Asymptotics of the spectrum and the Selberg zeta function on the space of Riemann surfaces},
\newblock 
Commun. Math. Phys.
\newblock 
{\bf 112}
\newblock 
(1987),
\newblock 
283--315.


\bibitem{Yoshikawa04}
Yoshikawa, K.-I.
\newblock 
\textit{$K3$ surfaces with involution, equivariant analytic torsion, and automorphic forms on the moduli space},
\newblock 
Invent. Math.
\newblock 
{\bf 156}
\newblock 
(2004),
\newblock 
53--117.


\bibitem{Yoshikawa07}
Yoshikawa, K.-I.
\newblock 
\textit{On the singularity of Quillen metrics},
\newblock 
Math. Ann. 
\newblock 
{\bf 337}
\newblock 
(2007),
\newblock 
61--89.


\bibitem{Yoshikawa10a}
Yoshikawa, K.-I.
\newblock 
\textit{Singularities and analytic torsion},
\newblock
preprint,
\newblock 
arXiv:1007.2835v1
\newblock
(2010).


\bibitem{Yoshikawa10b}
Yoshikawa, K.-I.
\newblock 
\textit{On the boundary behavior of the curvature of $L^{2}$-metrics},
\newblock
preprint,
\newblock 
arXiv:1007.2836v1
\newblock
(2010).


\bibitem{Yoshikawa26}
Yoshikawa, K.-I.
\newblock 
\textit{Analytic torsion for Enriques $2n$-folds},
\newblock
in preparation


\end{thebibliography}
\end{document}